\documentclass[11pt]{article}
\usepackage[english]{babel}
\usepackage{url}
\usepackage{microtype}
\usepackage{authblk}
\usepackage{comment}

\usepackage{abstract}
\usepackage{amsmath,amssymb}
\usepackage{mathtools}
\usepackage{graphicx}
\usepackage{caption}
\usepackage{subfig}
\usepackage{floatrow}
\usepackage{rotating}
\newfloatcommand{capbtabbox}{table}[][\FBwidth]
\usepackage{amsthm}
\usepackage{bm}
\usepackage{epstopdf}
\usepackage{float}
\usepackage{xcolor,colortbl}
\definecolor{MyBlueUrl}{rgb}{0.12,0.22,1}
\definecolor{MyredLink}{rgb}{0.9,0.11,0.17}
\definecolor{greenTcite}{rgb}{0.25,0.65,0.38}
\usepackage{enumitem}
\usepackage{hyperref}
\hypersetup{colorlinks=true,
linkcolor=MyredLink,
citecolor=greenTcite,
urlcolor=MyBlueUrl}
\setenumerate[0]{label=\roman*)}
\setenumerate[1]{label=\arabic*.}
\setenumerate[2]{label=(\alph*)}
\usepackage[verbose=true]{geometry}
\AtBeginDocument{
  \newgeometry{
    textheight=9in,
    textwidth=7.4in,
    top=1in,
    headheight=14pt,
    headsep=25pt,
    footskip=30pt
  }
}
\newcommand{\headeright}{}
\usepackage{fancyhdr}
\fancyheadoffset{0pt}
\title{Stability and Bifurcation Analysis of Nonlinear PDEs\\ via Random Projection-based PINNs:\\ A Krylov-Arnoldi Approach}
\newcommand{\shorttitle}{Stability and Bifurcation Analysis of Nonlinear PDEs via PI-RPNNs
}
\newcommand{\myand}{$\cdot \text{ }$}
\author[1,2,*]{Gianluca Fabiani}
\author[3]{Michail E. Kavousanakis}
\author[4]{Constantinos Siettos}
\author[2,5]{Ioannis G. Kevrekidis}

\affil[1]{Hopkins Extreme Materials Institute, \emph{Johns Hopkins University}, Baltimore, 21218, MD, USA}
\affil[2]{Dept. of Chemical and Biomolecular Engineering, \emph{Johns Hopkins University}, Baltimore, 21218, MD, USA}
\affil[3]{School of Chemical Engineering, \emph{National Technical University of Athens}, 9 Iroon Polytechniou Street, Athens, 15772, Greece}
\affil[4]{Dipartimento di Matematica e Applicazioni ‘‘Renato Caccioppoli", \emph{Universit\`a degli Studi di Napoli} \emph{Federico II}, Naples, Italy}
\affil[5]{Dept. of Applied Mathematics and Statistics, \emph{Johns Hopkins University}, Baltimore, 21218, MD, USA}
\affil[*]{Corresponding authors, email: \texttt{gfabian2@jh.edu,yannisk@jhu.edu}}

\date{\vspace{-5ex}}

\newtheorem{theorem}{Theorem}
\newtheorem{proposition}[theorem]{Proposition}
\newtheorem{lemma}[theorem]{Lemma}

\newtheorem{definition}{Definition}[section]
\newtheorem{remark}[definition]{Remark}

\newcommand{\R}{\mathbb{R}}
\newcommand{\C}{\mathbb{C}}

\def\xb{{\bm {x}}}

\def\alphab{{\bm {\alpha}}}

\def\Fb{{\bm {F}}}
\def\dd{{\mathrm{d}}}

\begin{document}
\maketitle
\begin{abstract}
We present a numerical framework for the stability and bifurcation analysis of nonlinear partial differential equations (PDEs) in which the solutions are sought in the function space spanned by physics-informed random projection neural networks (PI-RPNNs), and discretized via a collocation approach. These are shallow, single-hidden-layer networks with randomly sampled and fixed \emph{a priori} hidden-layer weights; only the linear output layer weights are optimized, reducing training to a single least-squares solve.
This linear output structure enables the direct and explicit formulation of the eigenvalue problem governing the linear stability of stationary solutions. This takes a generalized eigenvalue form, which naturally separates the physical domain interior dynamics from the algebraic constraints imposed by boundary conditions (BCs), at no additional training cost and without requiring additional PDE solves.
However, the random projection collocation matrix is inherently numerically rank-deficient, rendering naive eigenvalue computation unreliable and contaminating the true eigenvalue spectrum with spurious near-zero modes.
To overcome this limitation, we introduce a matrix-free shift-invert Krylov-Arnoldi method that operates directly in weight space, avoiding explicit inversion of the numerically rank-deficient collocation matrix and enabling the reliable computation of several leading eigenpairs of the physical Jacobian -- the discretized Fréchet derivative of the PDE operator with respect to the solution field (also referred to as the PDE linearization operator), whose eigenvalue spectrum determines linear stability.
We further prove that the PI-RPNN-based generalized eigenvalue problem is almost surely regular, guaranteeing solvability with standard eigensolvers, and that the singular values of the random projection collocation matrix decay exponentially for analytic activation functions, rigorously explaining the observed ill-conditioning.
\end{abstract}
\hspace{0.5cm} {\footnotesize
\myand Stability analysis
\myand Numerical bifurcation analysis
\myand Physics-informed random projection neural networks
\myand PINNs
\myand Shift-invert Krylov-Arnoldi's method
\myand Scientific Machine Learning (SciML)
\myand Numerical Analysis
}
-----
{\small
\myand 	65F18 
\myand 	65F22 
\myand 	65L07  	
}
\vskip 0.05in
\hrule height 1pt
\vskip 0.1in
\section{Introduction}
\sloppy
%
Physics-informed neural networks (PINNs)~\cite{raissi2019physics, karniadakis2021physics, lu2021deepxde}, first coined in Raissi et al. (2019)~\cite{raissi2019physics}, have emerged as a popular and promising approach for the numerical solutions of partial differential equations (PDEs).
These methods combine the flexibility of neural networks with the ability to enforce physical constraints.
The interest in using Artificial Neural Networks (ANNs) for differential equations can be traced back to the 1990s~\cite{meade1994numerical, dissanayake1994neural, lagaris1998artificial, gerstberger1997feedforward, gonzalez1998identification}. The foundational work of Lagaris et al. (1998)~\cite{lagaris1998artificial} systematically explored ANNs for linear and nonlinear DEs, including initial and boundary value problems.
In recent years, supported by the sharp growth in available computational power, PINNs have attracted increasing attention~\cite{samaniego2020energy, kavousanakis2026going, zou2025learning, kiyani2025optimizing, wang2025turbulence}. Improvements in optimization strategies, training techniques, and architectural designs have enabled them to tackle more challenging PDE problems, including for example Differential Algebraic Equations (DAEs)~\cite{kavousanakis2026going} and 3D turbulence simulations~\cite{wang2025turbulence}.
PINNs have shown particular promise for high-dimensional PDEs~\cite{wei2018machine, han2018solving, hu2025scorepinn, georgiou2025heatnets, deryck2025approximation}, where the ANN scalability allows them to partially overcome the ``curse of dimensionality", and have been successfully applied to dynamical systems tasks~\cite{alvarez2023discrete, alvarez2024nonlinear, patsatzis2024slow, kalia2021learning, bertalan2019learning} and operator-learning~\cite{li2024pinnoperator, wang2021learning, goswami2023physics}.
Despite their success in forward problems, the application of PINNs to stability and bifurcation analysis has been relatively unexplored. Only a few works have investigated this~\cite{fabiani2021numerical, galaris2022numerical, shahab2024neural, shahab2025corrigendum, shahab2025pinnlattices}, with~\cite{fabiani2021numerical} being the first to explicitly explore bifurcations in a PINN framework, coupling PINN with arc-length continuation approaches.

This underexploration stems from the difficulty of computing multiple PDE solutions across parameter values, which often entails high computational cost and the use of techniques such as {transfer learning}~\cite{liu2023adaptive} or deep ensembling~\cite{zou2025learning} for multi-parameter or multi-stability settings.
As with any numerical solver, a new solution must be computed for each parameter instance; however, PINNs are particularly costly in this regard, since each retraining relies on gradient-based optimization, which 
is prone to slow convergence and loss plateaus.
Moreover, even when initialized close to a known solution, the stochastic nature of the optimizer does not guarantee that the learned weights vary smoothly with the parameter, making it difficult to exploit continuity of the solution branch in weight space.
Although multiple solutions may still be obtained with significant effort, reliably computing several leading eigenvalues, required for system-level stability analysis, remains challenging and limits systematic bifurcation studies.
Stability and bifurcation analysis of nonlinear PDEs is central to understanding how variations in system parameters can modify the eigenvalues of {the PDE linearization operator}. This operator, which we refer to as the \emph{``physical Jacobian,"} is the Fréchet derivative of the PDE operator with respect to the solution field, and its spectral changes can lead to instabilities, bifurcations, or pattern formation. {The study of such bifurcation phenomena is of broad interest, driving the development of specialized numerical techniques for diverse problems, including traveling waves in nonlinear media~\cite{karakoc2023collocation, saha2023new}.}

Training ANNs, and particularly PINNs with complex loss landscapes, can be challenging and potentially NP-hard~\cite{blum1988training, froese2023training}, with further difficulties arising from negative experimental results on convergence and overfitting~\cite{almira2021negative, adcock2021gap}, as well as the need for careful tuning of optimizers and training strategies~\cite{kiyani2025optimizing, karumuri2025efficient, wang2025turbulence}.
These challenges become more pronounced when targeting bifurcation analysis, where multiple solutions and eigenvalue spectra must be accurately captured, and where training procedures that succeed at one parameter value may fail at others.

Despite the rapid progress of PINNs, the development of systematic and robust machine learning methodologies for solving large-scale and stiff nonlinear PDEs remains an open problem. This motivates the study of alternative architectures that can reduce computational cost while retaining accuracy. One prominent direction, which we explore in this paper, is given by Random Projection Neural Networks (RPNNs)~\cite{fabiani2025random, fabiani2023parsimonious, fabiani2021numerical,galaris2022numerical}. In this architecture, the hidden-layer weights and biases are randomly sampled and kept fixed, rendering the hidden layer a nonlinear random feature map. Consequently, the only trainable parameters are the output layer weights, whose optimization reduces to a linear least-squares problem -- a convex formulation that admits a global minimizer obtainable in closed form (e.g., via a single (regularized) matrix pseudo-inversion).

RPNNs belong to a broader family of randomized or fixed-structure neural networks, including \emph{Random Weight Neural Networks} (RWNNs)~\cite{schmidt1992feed}, \emph{Random Vector Functional Link networks} (RVFL)~\cite{pao1994learning}, \emph{reservoir computing}~\cite{jaeger2001echo}, \emph{Extreme Learning Machines} (ELMs)~\cite{huang2006extreme} and \emph{Random Fourier Features} Networks~\cite{rahimi2007random}. These methods are also referred to as \emph{Sampled Neural Networks}~\cite{bolager2024sampling}, and \emph{Gradient-Free Neural Networks}~\cite{bolager2026koopmaninformed}.
The use of randomized architectures is further supported by theoretical results on universal approximation with random bases~\cite{barron1993universal, igelnik1995stochastic, rahimi2008uniform, deryck2025approximation, fabiani2025random, fabiani2025randonets, fabiani2024stability}.

Physics-informed versions of these models (PI-RPNNs\footnote{often also equivalently named PI-ELMs.}) have been developed for a wide range of problems~\cite{dwivedi2020physics, fabiani2021numerical, calabro2021extreme, dong2021local, schiassi2021extreme, dong2022computing, yan2022framework, fabiani2023parsimonious, sun2024local, osorio2025physics, li2025fourier}. 
For stationary nonlinear PDEs, Fabiani et al. (2021)~\cite{fabiani2021numerical} introduced a PI-RPNN framework trained via a regularized Gauss--Newton method, demonstrating competitive accuracy and computational efficiency for the solution of nonlinear PDEs with respect to finite-difference and finite-element schemes. 
Notably, the work in refs.~\cite{fabiani2021numerical,galaris2022numerical} extended for the first time the application of RPNNs beyond the sole numerical solution of PDEs, to the numerical bifurcation analysis of such systems.
PI-RPNNs have also been applied to stiff ODEs and DAEs with time-adaptation schemes~\cite{fabiani2023parsimonious}, and~\cite{fabiani2024stability} -- the first study on the numerical linear stability of PINN-based schemes for stiff ODEs -- showed that appropriate hyperparameter selection can ensure A-stability of PI-RPNNs.
More recent work extended these methods to high-dimensional PDEs~\cite{georgiou2025heatnets, deryck2025approximation}, inverse PDE problems~\cite{galaris2022numerical, fabiani2024task, ahmadi2024ai} and operator learning through randomized operator networks (RandONets)~\cite{fabiani2025randonets, fabiani2025enabling, fabiani2025equation}, as well as to POD-DeepONets and Fourier Neural Operators, both constructed via random sampling in Ref.~\cite{bolager2026koopmaninformed}.
The low training cost and explicit linear output structure of PI-RPNNs make them particularly well suited for multi-parameter studies such as stability and bifurcation analysis, where the solution must be tracked across parameter ranges and eigenvalue computations must be performed repeatedly.
Crucially, the linear output structure of PI-RPNNs — where the PDE solution is approximated as a linear combination of fixed, nonlinear basis functions — enables the direct and explicit formulation of the eigenvalue problem governing the linear stability of stationary solutions. The Jacobian matrices required for this stability analysis are already available from the Gauss–Newton training of the PI-RPNN; they need only be reused rather than recomputed.

However, care must be taken: the weight-space Jacobian obtained from training does not itself represent the linearized PDE operator, as its singular values reflect the optimization landscape rather than the physical dynamics. Instead, the physical Jacobian ($\mathbf{J}_u$) is related to its weight-space counterpart ($\mathbf{J}_w$) through the basis collocation matrix ($\mathbf{\Psi}$) via the chain rule ($\mathbf{J}_w=\mathbf{J}_u \mathbf{\Psi}$). Note that the physical Jacobian $\mathbf{J}_u \in \R^{M \times M}$ is square, whereas $\mathbf{J}_w \in \R^{M \times N}$ is in general rectangular, as the number of neurons $N$ may be larger or smaller than the number of collocation points $M$.
Extracting the former from the latter would require (pseudo)-inverting this basis collocation matrix, a numerically dangerous operation given its severe ill-conditioning\footnote{Random projection bases are typically non-orthogonal, and their associated collocation matrix is rank-deficient. In particular, the singular values are empirically observed to decay rapidly.  We prove that, for large $N$ and analytic activation functions, this decay is exponential (Proposition~\ref{prop:Psi_decay}), resulting in marked numerical ill-conditioning even with parsimonious sampling.}.
To circumvent this difficulty, in this work we develop a matrix-free shift-invert Krylov–Arnoldi method that operates directly in the weight space, completely avoiding explicit inversion. Furthermore, to properly separate physical dynamics from algebraic constraints, we formulate the problem as a generalized eigenvalue system with an appropriate mass matrix that isolates the effect of BCs. This ensures that only eigenvalues associated with the physical interior dynamics are retained, while spurious modes linked to boundary constraints are excluded.
This approach, building on Fabiani et al. (2021)~\cite{fabiani2021numerical}, and Galaris et al. (2022)~\cite{galaris2022numerical} enables systematic computation of multiple leading eigenvalues for stability and bifurcation analysis of nonlinear PDEs, overcoming limitations of previous approaches~\cite{fabiani2021numerical, shahab2025pinnlattices, shahab2025corrigendum} that required separate PDE solves for each eigenvalue.

To support the reliability of this approach, we prove two complementary theoretical results. First, we show that the generalized eigenvalue problem in the PI-RPNN basis is almost surely regular, ensuring that standard eigensolvers such as those implemented in LAPACK or ARPACK libraries can be applied with high probability. Second, we demonstrate that for analytic activation functions, the singular values of the collocation matrix decay exponentially. This explains the severe rank-deficiency that propagates to the approximated physical Jacobian, producing a cluster of spurious near-zero eigenvalues while leaving the physically relevant part of the true eigenvalue spectrum intact.

For comparison with previous studies, our earlier work~\cite{fabiani2021numerical,galaris2022numerical} addressed only bifurcation of steady-state solutions, without the associated stability or eigenvalue computations. More recently,~\cite{shahab2025corrigendum, shahab2025pinnlattices} computed eigenvalues but required a separate PDE solve for each eigenpair, limiting scalability. In contrast, here we show that once a stationary solution is obtained via PI-RPNN training, the relevant Jacobian matrices and the generalized eigenvalue formulation needed for stability analysis follow directly from the Gauss–Newton iterations, at no additional training cost. Concurrently, the Eig-PIELM approach~\cite{mishra2026eig} has proposed a similar generalized eigenvalue problem, but its scope is limited to linear PDEs and does not address the numerical conditioning challenges that arise in bifurcation regimes.

\par 
The remainder of the paper is organized as follows. 
Section~\ref{sec:RPNN} introduces the PI-RPNN architecture, the construction of the random bases, and the approximation results underpinning the approach. 
Section~\ref{sec:methods} presents the methodological framework: the PDE problem setting in Section~\ref{sec:Prob_setting}, the PI-RPNN residual formulation and the numerical solution strategy, together with the continuation procedure for bifurcation tracking in Section~\ref{sec:PI-RPNN_solution}. 
Section~\ref{sec:PI-RPNN_stability} develops the stability analysis based on linearized operators, including the generalized eigenvalue problem in the weight space and the matrix-free Arnoldi shift--invert strategy used to mitigate the ill-conditioning introduced by random projections. 
Section~\ref{sec:numerical_results} presents numerical experiments for the Liouville-Bratu-Gelfand problem, the FitzHugh--Nagumo system, and the Allen--Cahn equation. 
Conclusions are given in Section~\ref{sec:conclusion}.

\color{black}
\subsection*{Notation}
Throughout this work, vectors are denoted by bold symbols (e.g., $\bm{w}$), and matrices by bold uppercase symbols (e.g., $\mathbf{\Psi}$). Differential and boundary operators are written in calligraphic style (e.g., $\mathcal{N}$). The main notation is summarized in Table~\ref{tab:scalars}; the key matrices used throughout the text, along with their dimensions, are listed in Table~\ref{tab:matrices}.

\begin{table}[ht!]
\color{black}
\caption{Summary of scalar and operator notation.}
\centering
\renewcommand{\arraystretch}{1.2}
\begin{tabular}{|c|l|}
\hline
\textbf{Symbol} & \textbf{Description} \\
\hline
\hline
$M$ & Number of collocation points \\
$N$ & Number of neurons (random basis functions) \\
$d$ & Spatial dimension of the domain $\Omega$ \\
\hline
$\psi$ & activation function \\ 
$\bm{w}$ & Output weights ($\in\mathbb{R}^N$) \\
$\alphab_j$ & Internal weights of the $j$-th neuron ($\in \R^d$) \\
$\beta_j$ & Bias of the $j$-th neuron \\
$\bm{x}$ & Collocation point in $\Omega$ ($\in \R^d$) \\
\hline
$\mathcal{N}$ & Nonlinear differential operator \\
$\mathcal{B}$ & Boundary operator \\
$f(\bm{x})$ & Source term in PDE \\
$g(\bm{x})$ & Prescribed boundary data \\
$\Omega$ & Spatial domain \\
$\mu$ & Bifurcation parameter \\
$\Fb$ & system of nonlinear algebraic equations \\
\hline
$\mathcal{L}_u$ & Linearized PDE operator\\
$\lambda$ & Eigenvalue \\
$\bm{\phi}$ & Eigenfunction (physical perturbation) \\
$\bm{v}$ & Eigenvector in weight space ($\in \R^N$)\\
\hline
\end{tabular}
\label{tab:scalars}
\end{table}

\begin{table}[ht]
\color{black}
\caption{Summary of the main matrices and their dimensions.}
\centering
\renewcommand{\arraystretch}{1.2}
\begin{tabular}{|c|c|l|}
\hline
\textbf{Symbol} & \textbf{Dimension} & \textbf{Description} \\
\hline
\hline
$\mathbf{\Psi}$ & $M \times N$ & Random feature (collocation) matrix \\
$\mathbf{J}_w$ & $M \times N$ & Weight-space Jacobian ($\nabla_w \Fb$) \\
$\mathbf{J}_u$ & $M \times M$ & Physical Jacobian ($\nabla_u \Fb$) \\
$\mathbf{B}$ & $M \times M$ & Boundary constraint matrix (selector) \\
$\widetilde{\mathbf{B}}_\Psi$ & $M \times N$ & Constrained feature matrix ($\mathbf{B} \Psi$) \\
$\mathbf{A}_\sigma$ & $M \times N$ & Shifted pencil matrix ($\mathbf{J}_w - \sigma \widetilde{\mathbf{B}}_\Psi$) \\
\hline
\end{tabular}
\label{tab:matrices}
\end{table}

\section{Preliminaries on Random Projection Neural Networks}
\label{sec:RPNN}
Random Projection Neural Networks (RPNNs) are a class of artificial neural networks (ANNs) that use fixed, randomly generated hidden-layer weights. Conceptually, these methods implement a random feature mapping of the input, which can capture essential structure and distances in the data. Studies have shown that appropriately constructed nonlinear random projections/features can provide efficient and expressive representations for a wide range of learning tasks~\cite{barron1993universal, igelnik1995stochastic, ito1996nonlinearity, rahimi2007random, rahimi2008uniform, fabiani2025random, fabiani2025randonets, bolager2024sampling, ahmadi2024ai, osorio2025physics}.

Let us consider, without loss of generality~\cite{igelnik1995stochastic, fabiani2025random}, a single output, single hidden layer feed-forward neural network with $N$ neurons, denoted by a function $f_N:\R^d \rightarrow\R$ with an \emph{a priori} fixed matrix of \emph{internal weights} $\mathbf{A} \in \R^{d\times N}$ with $N$ columns $\alphab_j \in \R^{d\times 1}$ and \emph{biases} $\bm \beta=(\beta_1,\dots,\beta_N) \in \R^{1\times N}$:
\begin{equation}
f_N(\xb;\bm{w},\beta^o,\mathbf{A},\bm\beta)=\sum_{j=1}^N w_j \psi (\xb \cdot \alphab_j + \beta_j) +\beta^{o}=\sum_{j=1}^N w_j \psi_j (\xb)+\beta^{o},
\label{eq:RPNN}
\end{equation}
where $N$ is the number of neurons (nodes), $d$ is the dimension of the input $\xb \in \R^{1\times d}$, $\beta^{o} \in \R$ is a scalar constant offset, or the so-called \emph{output bias}, $\psi:\R\rightarrow\R$ is the so-called \emph{activation} (transfer) function, that for fixed parameters $\bm \alpha_j$ and $\beta_j$ we denote as a fixed \emph{basis function} $\psi_j$, and $\bm{w}=(w_1,\dots,w_N)^{\top} \in\R^{N \times 1}$ are the \emph{external (readout) weights} that connect hidden layer and output layer. In RPNNs, the weights $\bm{w}$ and the offset bias $\beta^o$ are the only trainable parameters of the network.

When approximating a sufficiently smooth function $f:\Omega \subseteq \R\rightarrow\R$, for which we can evaluate $n+1$ points of its graph $(\xb_0,y_0),(\xb_1,y_1),(\xb_2,y_2),\dots,(\xb_n,y_n)$ such that $y_i=f(\xb_i)$, the training of an RPNN reduces to the solution of a linear regression/interpolation system of $n+1$ algebraic equations with $N+1$ unknowns $\bm{\widetilde{w}}=(\beta^o,\bm{w})$:
\begin{equation}
    \mathbf{\Psi} \bm{w}+\beta^o=\bm{y}, \qquad [\mathbf{\Psi}]_{ij}=\psi_j(\xb_i),
    \label{eq:RPNN_solve}
\end{equation}
where $\bm{y}=(y_0,y_1,\dots,y_n)$ is the vector containing the desired outputs, and $\mathbf{\Psi}$ is the random collocation matrix. The output bias term $\beta^o$ can be handled in several ways: it may be set to zero, included as an additional unknown in the least-squares problem, or prescribed \emph{a priori}, for instance as the mean of the output data. For the sake of simplicity of notation, we set $\beta^o = 0$ throughout this work. \par

\subsection{Theoretical foundations of RPNNs}
The effectiveness of Random Projection Neural Networks (RPNNs) is grounded in a synthesis of classical results from random projections~\cite{johnson1984extensions}, random feature methods~\cite{rahimi2008uniform}, and randomized neural networks~\cite{igelnik1995stochastic, fabiani2025random}. At a fundamental level, RPNNs extend the concept of linear random projections, as formalized by the Johnson--Lindenstrauss (JL) Lemma, which asserts the existence of a low-dimensional ($k<d$) isometric embedding
\begin{equation}
\label{eq:JL_randomProj}
\xb\in\R^{1\times d}\mapsto\Fb(\xb) = \frac{1}{\sqrt{k}} \xb \mathbf{R}\in \R^k, \quad \mathbf{R}\in\R^{d\times k}, \quad (\mathbf{R})_{ij} \sim \mathcal{N}(0,1),
\end{equation}
that approximately preserves Euclidean distances among a set of $n$ points in $\R^d$. This establishes that random linear mappings can robustly encode the geometric structure of high-dimensional data, providing a first justification for fixing internal weights in RPNNs.

RPNNs generalize this principle to nonlinear random projections, where the hidden layer acts as a ``lifting operator" mapping
\begin{equation}
    \xb\in\R^d \mapsto \mathbf{\Psi}(\bm{x})\equiv\bigl[\psi(\xb \cdot \alphab_1 + \beta_1), \dots, \psi(\xb \cdot \alphab_N + \beta_N)\bigr]\in\R^N.
\end{equation}

Such nonlinear random embeddings can preserve not only Euclidean distances but also more general kernel distances with low distortion, providing a richer representation of the input geometry.
As a concrete example, Rahimi and Recht~\cite{rahimi2007random} formalized this in the context of Random Fourier Features, using \emph{sine} and \emph{cosine} as transfer functions.

More generally, for other activation functions, it can be shown that the function class spanned by such nonlinear features is dense in an appropriate Reproducing Kernel Hilbert Space (RKHS)~\cite{rahimi2008uniform}.

In practical implementations with finite $N$, it is important that the sampled basis functions remain linearly independent to ensure (theoretically) full row/column rank of the collocation matrix, although the collocation matrix may still be numerically rank-deficient in practice. Following Ito~\cite{ito1996nonlinearity}, a sufficient condition for linear independence of $\{\psi(\xb \cdot \alphab_j  + \beta_j)\}_{j=1}^N$ is that $\psi$ is a slowly increasing non-polynomial plane wave with open Fourier support, and that $(\alphab_j,\beta_j) \neq \pm (\alphab_{j'},\beta_{j'})$ for all $j \neq j'$. 
When $\alphab$s and $\beta$s are drawn randomly from a continuous distribution, the above is true with probability 1~\cite{huang2006extreme, fabiani2025random}.
This ensures that the collocation matrix in the least-squares problem provides a well-posed system for solving the output weights.

Moreover, by construction the RPNN functional space spanned by a finite set of randomized and fixed basis functions is a vector space of dimension $N+1$, and if $\psi \in C^\nu(\R)$, it is a subspace of $C^\nu([a,b])$. Equipped with the $L^p$ norm, it becomes a Banach space, allowing the definition of the RPNN of best approximation,
whose existence and uniqueness are guaranteed for $1<p<\infty$~\cite{fabiani2025random}.
For a one-dimensional domain $[a,b]$, it has been shown~\cite{fabiani2025random} that the RPNN of best $L^p$ approximation with infinitely differentiable non-polynomial activation functions converges exponentially fast with $N$ when approximating smooth target functions. In higher dimensions ($d>1$), networks with randomly assigned hidden parameters can approximate sufficiently regular functions with mean-square error decaying as $O(1/N)$~\cite{igelnik1995stochastic}.

\section{Methods}
\label{sec:methods}
We begin by specifying the PDE models and the corresponding eigenvalue problems used for stability analysis. We then outline the PI-RPNN formulation for the numerical solution of stationary and time-dependent PDEs, including the construction of the residuals and the treatment of BCs. Finally, we detail the stability procedure based on the linearized operators, the associated generalized eigenproblems, and the matrix-free shift–invert approach. {A schematic overview of the proposed framework, highlighting the distinction between the naive (physical-space) and the proposed (weight-space) approaches, is provided in Fig.~\ref{fig:schematic}.}

\subsection{Problem setting}
\label{sec:Prob_setting}
Let $\Omega \subset \R^d$ be a bounded domain with boundary $\partial\Omega$.  
We consider the stationary formulation of a parameterized nonlinear partial differential equation (PDE)
\begin{equation}
\label{eq:PDE-strong}
\mathcal{N}[u;\mu](\bm{x}) = f(\bm{x}), \qquad \bm{x}\in\Omega,
\end{equation}
{where $\mathcal{N}$ is a nonlinear differential operator depending on the solution $u$ and the bifurcation parameter $\mu$, and $f$ is a given source term.} 
The equation is subject to boundary conditions (BCs)
\begin{equation}
\label{eq:BCs}
\mathcal{B}[u](\bm{x}) = g(\bm{x}), 
\qquad \bm{x}\in\partial\Omega,
\end{equation}
{where $\mathcal{B}$ denotes the boundary operator and $g$ is the prescribed boundary data.}
A solution $u^{\ast}=u^{\ast}(\bm{x};\mu)$ satisfies the residual equations $\mathcal{F}(u^{\ast};\mu)=0,$ where $\mathcal{F}$ collects both the differential and boundary operators given by Eqs.~\eqref{eq:PDE-strong}-\eqref{eq:BCs}:
\begin{equation}
\mathcal{F}(u^{\ast};\mu) := \big(\mathcal{N}[u^{\ast};\mu]-f; \mathcal{B}[u^{\ast}]-g\big) = \bm{0}.
\label{eq:residual_PDE}
\end{equation}

The goal is twofold: (i) compute the steady-state solution $u^{\ast}$ for given $\mu$, and (ii) determine its stability by analyzing the true eigenvalue spectrum of the associated linearized operator.
For a stationary state $u^{\ast}$, the linearization of the PDE around $u^{\ast}$ defines the operator
\begin{equation}
\mathcal{L}_u \phi := \frac{\delta \mathcal{N}}{\delta u}[u^{\ast};\mu] \phi, \qquad \phi\in V,
\label{eq:lin_pde_infinite_eigenproblem}
\end{equation}
where $V$ is the space of admissible perturbations satisfying the BCs.  
In other words, $\phi \in V$ encodes only variations consistent with the BCs, and $\mathcal{L}_u$ acts on this constrained space.  
In standard numerical analysis, $V$ could be a Sobolev space with incorporated BCs. In this work, $V$ is the finite-dimensional subspace spanned by the PI-RPNN basis functions, {restricted to perturbations that satisfy the homogeneous linearized BCs. For homogeneous Dirichlet conditions, this means $\phi|_{\partial\Omega}=0$; for homogeneous Neumann conditions, $\partial_n \phi|_{\partial\Omega}=0$. Note that the BCs on the perturbations are always homogeneous, even when the original problem has non-homogeneous BCs, since the base solution $u^{\ast}$ already satisfies the non-homogeneous conditions and perturbations are variations around it.}

Formally, let $(\lambda,\phi)$ denote an eigenpair of $\mathcal{L}_u$, i.e., $\mathcal{L}_u \phi = \lambda \phi$.  
The eigenvalue spectrum $\sigma(\mathcal{L}_u)$ determines the linear stability of the stationary solution $u^{\ast}$: the state is linearly stable if all eigenvalues satisfy $\Re(\lambda)<0$, whereas the presence of any eigenvalue with $\Re(\lambda)>0$ indicates linear instability.  
Eigenvalues crossing the imaginary axis as $\mu$ varies indicate local \emph{candidate} bifurcations of the stationary solution. In this work, we use these crossings as indicators of candidate bifurcations and do not pursue the full verification of nondegeneracy, which is problem-dependent.

\subsection{PDE Solution and continuation with PI-RPNNs}
\label{sec:PI-RPNN_solution}
In this section, we outline the framework already introduced in Ref.~\cite{fabiani2021numerical}, for numerically solving and performing bifurcation analysis of nonlinear PDEs using PI-RPNNs, drawing on standard concepts and techniques from numerical continuation (e.g., see ref.~\cite{chan1982arclength}).
Let
\(
\mathcal{X}_M := \{\bm{x}_1,\dots,\bm{x}_{M_{\mathrm{int}}}\} \subset \Omega,\) and \(
\mathcal{X}_{M_{\mathrm{bc}}} := \{\bm{x}_{M_{\mathrm{int}}+1},\dots,\bm{x}_{M}\} \subset \partial\Omega,
\)
with $M = M_{\mathrm{int}} + M_{\mathrm{bc}}$, be a set of collocation points, which can be chosen either on a structured grid or randomly, consisting of $M_{\mathrm{int}}$ interior points in $\Omega$ and $M_{\mathrm{bc}}$ boundary points on $\partial\Omega$. We assume that the collocation points are quasi-uniform with mesh size $h>0$, so that $M_{\mathrm{int}} \sim h^{-d}$ and $M_{\mathrm{bc}} \sim h^{-(d-1)}$. The parameter $h$ plays the role of the discretization parameter. 
We employ PI-RPNNs to approximate the stationary solution $u^{\ast}$ of the nonlinear PDE in Eqs.~\eqref{eq:PDE-strong}-\eqref{eq:BCs}. Let $\widetilde{u}(\xb;\bm{w})$ denote the RPNN representation of $u$, with degrees of freedom $\bm{w}=(w_1,\ldots,w_N) \in \R^N$.
Then the residuals are enforced at the collocation points $\mathcal{X} \subset \overline{\Omega}$ resulting in the system of nonlinear algebraic equations~\cite{fabiani2021numerical, fabiani2023parsimonious}:
\begin{equation}
F_k(w_1,\dots,w_N;\mu)
=\mathcal{F}\!\left(\widetilde{u}(\bm{x}_k;\bm{w});\mu\right)=0,
\quad k=1,\dots,M,
\label{eq:residuals}
\end{equation}
where $\mu$ denotes a continuation or bifurcation parameter and $\widetilde{u}(\bm{x};\bm{w})$ is the PI-RPNN approximation of the physical solution $u$. The residual operator $\mathcal{F}$ defined in Eq.~\eqref{eq:residual_PDE} is thus evaluated on the neural-network ansatz and becomes an explicit function of the weight vector $\bm{w}$. Note that the collocation points $\{\bm{x}_k\}_{k=1}^M \subset \overline{\Omega}$ may lie either in the interior of the domain or on the boundary. Accordingly,
\begin{equation}
F_k(\bm{w};\mu)=
\begin{cases}
\mathcal{N}[\widetilde{u};\mu](\bm{x}_k)-f(\bm{x}_k),
& \bm{x}_k \in \Omega, \\
\mathcal{B}[\widetilde{u}](\bm{x}_k)-g(\bm{x}_k),
& \bm{x}_k \in \partial\Omega,
\end{cases}
\end{equation}
so that interior collocation points enforce the PDE residual, whereas boundary collocation points enforce the BCs.

The nonlinear systemin Eq.~\eqref{eq:residuals} is solved iteratively, e.g., via Newton's method, by linearizing around the current iterate $\bm{w}^{(n)}$:
\begin{equation}
\nabla_{\bm{w}} \Fb(\bm{w}^{(n)},\mu) \, \Delta\bm{w}^{(n)} = -\Fb(\bm{w}^{(n)},\mu), \quad
\bm{w}^{(n+1)} = \bm{w}^{(n)} + \Delta\bm{w}^{(n)},
\label{eq:Newton_iter}
\end{equation}
where $\Fb=[F_1,F_2,\dots,F_M]^\top$ and $\nabla_{\bm{w}} \Fb \in \R^{M\times N}$ is the Jacobian matrix
\begin{equation}
\left[\nabla_{\bm{w}} \Fb \right]_{kj} = \frac{\partial F_k}{\partial w_j}\Big|_{(\bm{w}^{(n)},\mu)}, \quad k=1,\dots,M,\ j=1,\dots,N.
\end{equation}
For rectangular and/or rank-deficient systems ($M \neq N$), one may employ a truncated SVD to compute a Moore--Penrose pseudo-inverse:
\begin{equation}
\nabla_{\bm{w}} \Fb = \mathbf{U} \mathbf{\Sigma} V^\top, \quad \Delta \bm{w}^{(n)} = - \mathbf{V}\mathbf{\Sigma}^{-1} U^\top \Fb(\bm{w}^{(n)},\mu),
\end{equation}
ensuring numerical stability of the update step.

Branches of solutions past critical points, where $\nabla_{\bm{w}} \Fb$ becomes singular, can be traced using numerical continuation techniques. In particular, the pseudo arc-length continuation method~\cite{chan1982arclength} parametrizes both $\widetilde{u}(\bm{w})$ and $\mu$ by the arc-length $s$, yielding the augmented system~\cite{fabiani2021numerical, galaris2022numerical}:
\begin{equation}
\begin{bmatrix}
\nabla_{\bm{w}} \Fb & \nabla_{\mu} \Fb \\
\nabla_{\bm{w}} \mathcal{C} & \nabla_{\mu} \mathcal{C}
\end{bmatrix}
\begin{bmatrix}
\Delta\bm{w}^{(n)}(s) \\ \Delta \mu^{(n)}(s)
\end{bmatrix}
= -
\begin{bmatrix}
\Fb(\bm{w}^{(n)}(s),\mu(s)) \\
\mathcal{C}(\widetilde{u}(\bm{w}^{(n)};s),\mu^{(n)}(s))
\end{bmatrix},
\label{eq:augmented_arc}
\end{equation}
where
\begin{equation}
\nabla_{\mu }\Fb=\begin{bmatrix}
\frac{\partial F_1}{\partial \mu} & \frac{\partial F_2}{\partial \mu} & \dots & \frac{\partial F_M}{\partial \mu}
\end{bmatrix}^\top.
\end{equation}
$\mathcal{C}(\cdot)$ is a chosen condition to enforce the pseudo arc-length constraint based on previous solutions $(\widetilde{u}_{-1},\mu_{-1})$ and $(\widetilde{u}_{-2},\mu_{-2})$, and $\Delta s$ is the arc-length step. For example:
\begin{equation}
\begin{split}
\mathcal{C}(\widetilde{u}(\bm{w}^{(n)}),\mu^{(n)}(s))=  (\widetilde{u}(\bm{w}^{(n)})-\widetilde{u}(\bm{w})_{-1})^\top\cdot \frac{(\widetilde{u}(\bm{w})_{-1}-\widetilde{u}(\bm{w})_{-2})}{\Delta s}+\\+ (\mu^{(n)}-\mu_{-1}) \cdot \frac{(\mu_{-1}-\mu_{-2})}{\Delta s}-\Delta s,
\end{split}
\end{equation}
is one of the choices for the so-called ``pseudo arc-length condition" (for more detailssee refs.~\cite{fabiani2021numerical, galaris2022numerical}), where the subscripts $(-1)$ and $(-2)$ denote the two previously converged solutions along the continuation branch, corresponding to arc-length parameters $s_{-1}$ and $s_{-2}$, respectively.

\subsection{Stability Analysis with PI-RPNNs}
\label{sec:PI-RPNN_stability}
We present a detailed formulation for computing the dominant eigenvalue spectrum of the PDE linearization when the steady state is represented by a PI-RPNN.

\begin{figure}[ht!]
    \centering
    \includegraphics[width=1\linewidth]{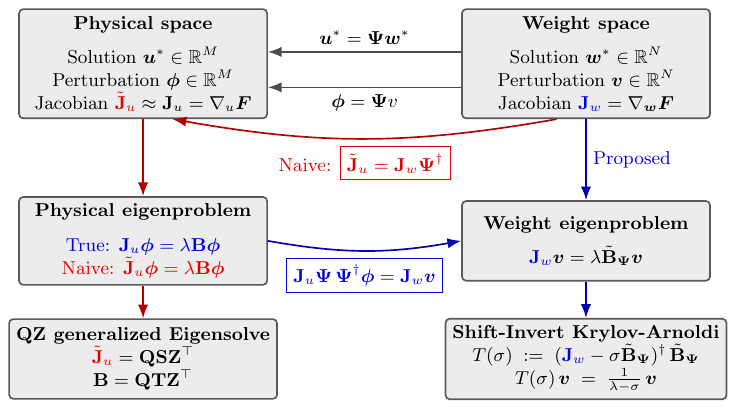}
    \caption{\color{black} Schematic overview of the physical-space (red, naive) and weight-space (blue, proposed) formulations for the stability analysis of nonlinear PDEs. The Jacobians are related by $\mathbf{J}_w=\mathbf{J}_u\mathbf{\Psi}$. Two pathways are shown: (i) the naive approach (red) computes $\widetilde{\mathbf{J}}_u= \mathbf{J}_w\mathbf{\Psi}^{\dagger}$ and solves the physical eigenproblem $\widetilde{\mathbf{J}}_u\bm{\phi}= \lambda\mathbf{B}\bm{\phi}$ via QZ decomposition; (ii) the proposed approach (blue) works directly in weight space with the rectangular pencil $\mathbf{J}_w\bm{v}= \lambda\widetilde{\mathbf{B}}_{\mathbf{\Psi}}\bm{v}$ and employs a matrix-free shift-invert Krylov--Arnoldi method.
    \label{fig:schematic}}
\end{figure}

We begin by clarifying the discrete eigenvalue problem that we wish to solve, and its relation to the continuous linearized operator $\mathcal{L}_u$ introduced in Eq.~\eqref{eq:lin_pde_infinite_eigenproblem}.
Let $\widetilde{u}(\bm{x};\bm{w})$ denote the PI-RPNN approximation of the stationary solution, obtained with Gauss-Newton iterations as in Eq.~\eqref{eq:Newton_iter}. We denote its corresponding \emph{weight-space} Jacobian by
\(
\mathbf{J}_w := \nabla_{\bm{w}}\Fb(\widetilde{u}(\bm{w}),\mu) \in \R^{M\times N}.
\)
Our goal is to obtain the leading eigenpairs of the \emph{physical} Jacobian
\(
\mathbf{J}_u := \nabla_u \Fb(u^\ast,\mu) \in \R^{M \times M}
\)
acting on perturbations of the field,
\begin{equation}
\mathbf{J}_u\bm{\phi} = \lambda \, \mathbf{B} \bm{\phi}, \qquad \bm{\phi} \in \R^M,
\label{eq:gen_eigRPNN}
\end{equation}
where $\mathbf{J}_u$ is the (discretized) physical Jacobian mapping perturbations in the sampled/physical space to residual changes, $\bm{\phi}=\phi(\mathbf{X})\in\R^M$ is an eigenfunction discretized at the collocation points $\mathcal{X}=\{\bm{x}_k\}_{k=1}^{M}$, and $\mathbf{B}$ is an $M \times M$ matrix used to enforce algebraic constraints, e.g., BCs.

\color{black}
The physical Jacobian $\mathbf{J}_u$ is not directly computable from the PI-RPNN representation, because the residual $\Fb$ depends on $u$ not only through its values at the collocation points but also through its derivatives. The derivatives of $\widetilde{u}$ are explicit functions of the trainable parameters $\bm{w}$, but they are not functions of the values $\widetilde{u}(\bm{x}_k)$ alone. Consequently, the derivative $\partial \Fb / \partial u$ is not well-defined without specifying a basis in which the derivatives are computed. This is the reason why we introduce an idealized reference basis.

We assume that the stationary solution $u^{\ast}$ admits an expansion in a fixed, deterministic, and well-conditioned basis $\Phi = \{\varphi_1, \varphi_2, \dots\}$ of the admissible perturbation space $V$, which we call the \emph{idealized reference basis}.
For the purpose of the analysis, we may take $\Phi$ to be the eigenfunctions $\{\phi_j\}_{j\ge 1}$ of the linearized operator $\mathcal{L}_u$ at $u^{\ast}$, i.e., 
\(
\mathcal{L}_u \phi_j = \lambda_j \phi_j,
\quad j = 1,2,\dots,
\)
assuming that the steady state $u^{\ast}$ is well posed and that the eigenfunctions of $\mathcal{L}_u$ exist and form a complete orthogonal basis of $V$ in a neighborhood of $u^{\ast}$.
The eigenfunctions of $\mathcal{L}_u$ are not known a priori in general, and their computation is itself part of the problem. For the purpose of the analysis, we therefore treat the idealized reference basis as an abstract, conceptual object. 

Let
\(
\bm{\Phi} := \big[ \varphi_j(\bm{x}_i) \big]_{i,j=1}^M \in \R^{M \times M}
\)
denote the collocation matrix of the idealized spectral basis.
Any function in the span of $\Phi_M$ can be written as $u(\bm{x}) = \bm{\Phi}(\bm{x}) \bm{c}$, with $\bm{c} \in \R^M$. The idealized discrete residual is the map $\Fb : \R^M \to \R^M$ defined by evaluating $\mathcal{F}$ on $u_M$ at the collocation points, and the idealized spectral Jacobian is $\mathbf{J}_c := \partial \Fb_M / \partial \bm{c}$ evaluated at the idealized discrete solution $\bm{c}^{\ast}$. The \emph{idealized physical Jacobian} is defined as the derivative of the residual with respect to the physical values, evaluated at the idealized discrete solution $\bm{u}^{\ast} = \bm{\Phi}_M \bm{c}^{\ast}$:
\begin{equation}
\mathbf{J}_u := \frac{\partial \Fb}{\partial \bm{u}}\bigg|_{\bm{u}^{\ast}} \in \R^{M \times M}.
\label{eq:Ju_definition}
\end{equation}
Thus, the idealized spectral and physical Jacobians are related by
\begin{equation}
\mathbf{J}_c = \mathbf{J}_u \bm{\Phi},
\qquad
\mathbf{J}_u = \mathbf{J}_c \bm{\Phi}^{-1}.
\label{eq:Ju_Jc_relation}
\end{equation}

In the PI-RPNN setting, the field is represented as $\widetilde{u}(\bm{x};\bm{w}) = \bm{\Psi}(\bm{x})^\top \bm{w}$, where $\bm{\Psi}(\bm{x})$ is the vector of random features. The physical values at the collocation points are $\widetilde{\bm{u}} = \bm{\Psi} \bm{w}$, with $\bm{\Psi} \in \R^{M \times N}$ the feature matrix. The error between the PI-RPNN field and the idealized discrete solution, expressed in physical space, is
\begin{equation}
\bm{e} := \widetilde{\bm{u}} - \bm{u}^{\ast} \in \R^M.
\label{eq:error_idealized}
\end{equation}
The residual evaluated at $\widetilde{\bm{u}}$ admits the Taylor expansion
\begin{equation}
\Fb(\widetilde{\bm{u}})
=
\mathbf{J}_u \bm{e}
+
\frac{1}{2} \mathbf{H}_u \cdot \bm{e} \cdot \bm{e}
+
O(\|\bm{e}\|^3),
\label{eq:taylor_residual}
\end{equation}
where $\mathbf{H}_u := \partial^2 \Fb / \partial \bm{u}^2|_{\bm{u}^{\ast}}$ is the third-order tensor of second derivatives of the residual, evaluated at the idealized discrete solution. Differentiating Eq.~\eqref{eq:taylor_residual} with respect to $\widetilde{\bm{u}}$ gives the corresponding expansion of the physical Jacobian,
\begin{equation}
\frac{\partial \Fb}{\partial \bm{u}}\bigg|_{\widetilde{\bm{u}}}
=
\mathbf{J}_u
+
\mathbf{H}_u \cdot \bm{e}
+
O(\|\bm{e}\|^2),
\label{eq:taylor_jacobian}
\end{equation}
where $\mathbf{H}_u \cdot \bm{e}$ denotes the contraction of the tensor $\mathbf{H}_u$ with the error vector $\bm{e}$ along one index, yielding a matrix in $\R^{M \times M}$. Eq.~\eqref{eq:taylor_jacobian} shows that the physical Jacobian evaluated at the PI-RPNN field differs from the idealized physical Jacobian by a term that is linear in the error $\bm{e}$ and involves the second derivative of the residual. In the idealized limit $\bm{e} \to \bm{0}$, the two coincide.

The weight-space Jacobian and the physical Jacobian are related by the chain rule,
\begin{equation}
\mathbf{J}_w
=
\frac{\partial \Fb}{\partial \bm{u}}\bigg|_{\widetilde{\bm{u}}} \bm{\Psi}
=
\left( \mathbf{J}_u + \mathbf{H}_u \cdot \bm{e} + O(\|\bm{e}\|^2) \right) \bm{\Psi}.
\label{eq:Jw_Ju_relation}
\end{equation}
In the idealized limit $\bm{e} \to \bm{0}$, Eq.~\eqref{eq:Jw_Ju_relation} reduces to the exact factorization
\begin{equation}
\mathbf{J}_w = \mathbf{J}_u \bm{\Psi},
\label{eq:Jw_Ju_exact}
\end{equation}
which is the identity used throughout the analysis. In the practical setting, the error term $\mathbf{H}_u \cdot \bm{e}$ is non-zero, and Eq.~\eqref{eq:Jw_Ju_relation} shows that the weight-space Jacobian differs from the idealized factorization by a term that is linear in the error between the PI-RPNN field and the idealized discrete solution.
We note that for a linear PDE with linear boundary conditions, the residual is affine in the field, so that the Hessian tensor $\mathbf{H}_u = \mathbf{0}$ and all higher-order terms vanish. In this case the physical Jacobian is independent of the computed solution, and the expansions in Eqs.~\eqref{eq:taylor_jacobian}--\eqref{eq:Jw_Ju_relation} are exact with no error term, so that $\mathbf{J}_w = \mathbf{J}_u \bm{\Psi}$ holds identically.

\subsection{The naive approach and its relation to the idealized problem}
\label{sec:naive_approach}

Having access to the weight-space Jacobian corresponding to the stationary solution $\widetilde{u}(\bm{w})$, a natural first approach is to recover the physical Jacobian directly via
\begin{equation}
\mathbf{J}_u \approx \widetilde{\mathbf{J}}_u:= \mathbf{J}_w \mathbf{\Psi}^{\dagger}\in\R^{M\times M},
\label{eq:naive}
\end{equation}
where $\mathbf{\Psi}^{\dagger}\in\R^{N\times M}$ is the (possibly truncated) pseudo-inverse of the collocation matrix.
{Once $\widetilde{\mathbf{J}}_u$ is constructed via Eq.~\eqref{eq:naive}, one can solve the (approximated) generalized eigenproblem
$
\widetilde{\mathbf{J}}_u\bm{\phi} = \lambda \mathbf{B}\bm{\phi}
$
directly using a standard eigensolver (e.g., MATLAB's \texttt{eig}).}
Although this strategy can, in principle, yield approximations of the leading eigenpairs (as will be seen in Fig.~\ref{fig:bratu_illcond}(b)), it suffers from several numerical pitfalls in practice.
We will refer to this as the \emph{naive} (or direct) approach, to distinguish it from the \emph{shift-invert Arnoldi} method introduced below.

\color{black}
The relation between the reconstructed physical Jacobian $\widetilde{\mathbf{J}}_u$ and the idealized physical Jacobian $\mathbf{J}_u$ follows directly from the chain rule and the properties of the pseudo-inverse. In the idealized limit $\bm{e} \to \bm{0}$, we have $\mathbf{J}_w = \mathbf{J}_u \bm{\Psi}$, and hence
\begin{equation}
\widetilde{\mathbf{J}}_u = \mathbf{J}_w \bm{\Psi}^\dagger = \mathbf{J}_u \bm{\Psi} \bm{\Psi}^\dagger = \mathbf{J}_u \mathbf{P}_{\bm{\Psi}},
\label{eq:naive_idealized}
\end{equation}
where $\mathbf{P}_{\bm{\Psi}} := \bm{\Psi} \bm{\Psi}^\dagger$ is the orthogonal projector onto the column space of $\bm{\Psi}$. This is the precise sense in which $\widetilde{\mathbf{J}}_u$ approximates $\mathbf{J}_u$: it is the projection of the idealized physical Jacobian onto the range of the random-feature matrix. The reconstruction is exact only when $\bm{\Psi}$ has full row rank, i.e., $\mathbf{P}_{\bm{\Psi}} = \mathbf{I}_M$, which requires $N \geq M$. When $N < M$, the reconstruction is a projection, and the missing component is $\mathbf{J}_u (\mathbf{I} - \mathbf{P}_{\bm{\Psi}})$, which lies in the orthogonal complement of the range of $\bm{\Psi}$ and is invisible to the weight-space Jacobian.

In the practical setting, the error $\bm{e}$ is non-zero, and the relation between $\widetilde{\mathbf{J}}_u$ and $\mathbf{J}_u$ includes an additional term from the Taylor expansion in Eq.~\eqref{eq:taylor_jacobian}:
\begin{equation}
\widetilde{\mathbf{J}}_u = \left( \mathbf{J}_u + \mathbf{H}_u \cdot \bm{e} + O(\|\bm{e}\|^2) \right) \mathbf{P}_{\bm{\Psi}}.
\label{eq:naive_general}
\end{equation}
So the reconstruction error has two distinct sources: the projection error $\mathbf{J}_u (\mathbf{I} - \mathbf{P}_{\bm{\Psi}})$, which is present even in the idealized limit, and the Taylor error $\mathbf{H}_u \cdot \bm{e} \, \mathbf{P}_{\bm{\Psi}}$, which vanishes as the PI-RPNN field approaches the idealized discrete solution.

The exponential decay of singular values of $\mathbf{\Psi}$ (Proposition~\ref{prop:Psi_decay}) renders it numerically rank-deficient. The truncation of near-zero singular values, necessary for the numerical stabilization of the pseudo-inverse, introduces exact rank-deficiency that propagates into the computed physical Jacobian $\widetilde{\mathbf{J}}_u = \mathbf{J}_w\mathbf{\Psi}^\dagger$, generating a cluster of spurious near-zero eigenvalues in its spectrum. The resulting eigenvalues become highly sensitive to the truncation tolerance used in $\mathbf{\Psi}^\dagger$, undermining reliability for systematic parameter studies.

More precisely, since for two generic matrices $\mathbf{Y},\mathbf{Z}$ we have $\mathrm{rank}(\mathbf{YZ}) \le \min(\mathrm{rank}(\mathbf{Y}),\mathrm{rank}(\mathbf{Z}))$, the product $\widetilde{\mathbf{J}}_u = \mathbf{J}_w \mathbf{\Psi}^{\dagger}$ inherits the rank of the truncated $\mathbf{\Psi}^{\dagger}$.
Introducing a truncation tolerance $\tau$ (e.g., $10^{-8}$), we define $r < \min(M,N)$ as the number of singular values of $\mathbf{\Psi}$ exceeding $\tau$. Then, $\mathbf{\Psi}^{\dagger}$, and consequently $\widetilde{\mathbf{J}}_u$, have at most rank $r$.
This effective rank deficiency manifests as a cluster of $M-r$ spurious near-zero eigenvalues in the spectrum of $\widetilde{\mathbf{J}}_u$, which we will illustrate later in Sec.~\ref{sec:numerical_results} (e.g., see Fig.~\ref{fig:bratu_illcond}(c)).

\subsubsection{The selector matrix and the numerical solution of the generalized eigenvalue problem}
A common and simple discrete construction of $\mathbf{B}$ in the generalized eigenproblem in Eq~\eqref{eq:gen_eigRPNN} is as a diagonal matrix with entries
\begin{equation}
(\mathbf{B})_{ii} = \begin{cases}
0, & \text{if } \bm{x}_i \in \partial\Omega \text{ (boundary point)},\\
1, & \text{if } \bm{x}_i \in \Omega \text{ (interior point)}.
\end{cases}
\label{eq:B_def}
\end{equation}
This explicitly separates the algebraic constraints from the dynamical eigenproblem, as here detailed.
Let the residual vector $\Fb \in \mathbb{R}^M$ be partitioned as
\begin{equation}
\Fb =
\begin{bmatrix}
\Fb_{\mathrm{int}} \\
\Fb_{\mathrm{bc}}
\end{bmatrix},
\qquad
\widetilde{\mathbf{J}}_u =
\begin{bmatrix}
\widetilde{\mathbf{J}}_{\mathrm{int}} \\
\widetilde{\mathbf{J}}_{\mathrm{bc}}
\end{bmatrix},
\qquad
\bm{\phi} =
\begin{bmatrix}
\bm{\phi}_{\mathrm{int}} \\
\bm{\phi}_{\mathrm{bc}}
\end{bmatrix},
\end{equation}
where the subscript $\mathrm{int}$ denotes the $M_{\mathrm{int}}$ interior collocation points and $\mathrm{bc}$ the $M_{\mathrm{bc}}$ boundary collocation points, with $M = M_{\mathrm{int}} + M_{\mathrm{bc}}$.

At a solution $\bm{w}^{\ast}$, such that $\Fb(\bm{w}^{\ast};\mu)= \mathbf{0}$, we consider perturbations $\bm{\phi}$ satisfying the linearized BCs
\begin{equation}
\widetilde{\mathbf{J}}_{\mathrm{bc}} \, \bm{\phi} =\mathbf{0}.
\label{eq:lin_bc}
\end{equation}
The linearized interior equations are
\begin{equation}
\widetilde{\mathbf{J}}_{\mathrm{int}} \, \bm{\phi} = \lambda \, \bm{\phi}_{\mathrm{int}},
\label{eq:lin_int}
\end{equation}
where we restrict the eigenvalue problem to the interior degrees of freedom. Combining Eq.~\eqref{eq:lin_bc} and Eq.~\eqref{eq:lin_int} yields
\begin{equation}
\begin{bmatrix}
\widetilde{\mathbf{J}}_{\mathrm{int}} \\
\widetilde{\mathbf{J}}_{\mathrm{bc}}
\end{bmatrix}
\bm{\phi}
=
\lambda
\begin{bmatrix}
\bm{\phi}_{\mathrm{int}} \\
\mathbf{0}
\end{bmatrix}
=
\lambda
\begin{bmatrix}
\mathbf{I} & \mathbf{0} \\
\mathbf{0} & \mathbf{0}
\end{bmatrix}
\bm{\phi},
\end{equation}
where the block diagonal matrix is precisely $\mathbf{B}$ in Eq.~\eqref{eq:B_def} under the ordering where interior points precede boundary points. 
Note that, this formulation is valid for both homogeneous and non-homogeneous BCs, as the latter are satisfied by the base solution $u(\bm x, \bm w^{\ast})$ around which linearization is performed.

Consequently, the eigenvectors $\bm{\phi}\in\R^M$ obtained from Eq.~\eqref{eq:gen_eigRPNN} satisfy the discrete analogue of the condition $\phi \in V$ in Eq.~\eqref{eq:lin_pde_infinite_eigenproblem}. The generalized eigenproblem thus realizes the infinite-dimensional constrained operator $\mathcal{L}_u$ in the finite-dimensional PI-RPNN basis.

Numerically, one can solve the generalized eigenproblem in Eq.~\eqref{eq:gen_eigRPNN}
using a QZ generalized Schur decomposition. This computes pairs $(s_{ii}, t_{ii})$ such that
\begin{equation}
    \widetilde{\mathbf{J}}_u= \mathbf{Q}\mathbf{S}\mathbf{Z}^\top, \qquad \text{and} \qquad \mathbf{B}=\mathbf{Q}\mathbf{T}\mathbf{Z}^\top,
\label{eq:QZ_Schur}
\end{equation}
where $\mathbf{Q}$ and $\mathbf{Z}$ are unitary, $\mathbf{S}$ and $\mathbf{T}$ are upper triangular matrices, and $s_{ii}$ and $t_{ii}$ are diagonal elements of $S$ and $T$, respectively.
Then, the generalized eigenvalues are recovered as
\begin{equation}
\lambda_i = \frac{s_{ii}}{t_{ii}}.
\label{eq:gen_eigs_lambda_i}
\end{equation}
If $t_{ii} = 0$, the corresponding generalized eigenvalue is formally $\lambda_i = \infty$.
Physically relevant eigenpairs correspond to the subspace where $t_{ii} \neq 0$, ensuring that only admissible perturbations are considered.


To summarize, the spectrum of the generalized eigenproblem in Eq.~\eqref{eq:gen_eigRPNN}, with $\mathbf{J}_u$ approximated by $\widetilde{\mathbf{J}}_u = \mathbf{J}_w \mathbf{\Psi}^{\dagger}$ as in Eq.~\eqref{eq:naive}, consists of three distinct groups:
\begin{enumerate}
    \item Eigenvectors in the nullspace of $\mathbf{B}$ correspond to boundary perturbations, yielding formally infinite eigenvalues $\lambda = \infty$; these directions do not affect the interior domain dynamics.
    \item The numerical rank-deficiency of the collocation matrix $\mathbf{\Psi}$ and the resulting physical Jacobian $\widetilde{\mathbf{J}}_u = \mathbf{J}_w \mathbf{\Psi}^{\dagger}$ generate a cluster of zero eigenvalues, associated with directions outside the range of $\mathbf{\Psi}^{\dagger}$; these correspond to numerically spurious modes and do not contribute to stability.
    \item The remaining finite eigenvalues lie in the subspace of admissible interior perturbations, and these eigenvalues determine the true "physical" linear stability of the stationary solution.
\end{enumerate}

Please note that the problematic case $s_{ii}=t_{ii}=0$ for some $i$ in Eq.~\eqref{eq:gen_eigs_lambda_i}, which corresponds to a singular matrix pencil and an indeterminate generalized eigenvalue, can in principle arise only under exact algebraic degeneracies between the computed physical Jacobian $\widetilde{\mathbf{J}}_u\equiv \mathbf{J}_w\mathbf{\Psi}^{\dagger}$ and the constraint matrix $\mathbf{B}$. In the present setting, such degeneracies are non-generic and occur with probability zero under continuous choices of collocation points and PI--RPNN parameters, as stated precisely in the following proposition.

\subsection{Almost sure regularity of the generalized pencil}
\label{sec:regular_pencil}
Sometimes it is convenient to view Eq.~\eqref{eq:gen_eigRPNN} as a \emph{matrix pencil}
\begin{equation}
\mathbf{J}_u - \lambda \mathbf{B},
\end{equation}
where $\lambda$ is a free spectral parameter. The generalized eigenvalues and eigenvectors of the pair $(\mathbf{J}_u, \mathbf{B})$ are equivalently the eigenvalues and eigenvectors of this pencil. We recall the following definition.

\begin{definition}[Characterization of a singular pencil]
\label{def:singular_pencil}
A pencil $(\mathbf{A}, \mathbf{B})$ is said to be \emph{singular} if
\begin{equation}
\det(\mathbf{A} - \lambda \mathbf{B}) \equiv 0 \quad \text{for all } \lambda \in \mathbb{C}.
\label{eq:singular_pencil}
\end{equation}
Otherwise, the pencil is called \emph{regular}.
\end{definition}

A regular pencil is necessary for the generalized eigenvalues to be well-defined and for standard numerical solvers to be applicable. In the PI-RPNN setting, however, the pencil involves the reconstructed physical Jacobian $\widetilde{\mathbf{J}}_u = \mathbf{J}_w \mathbf{\Psi}^{\dagger}$ rather than the idealized $\mathbf{J}_u$, and its regularity is therefore not immediate. The following proposition establishes that the pencil remains regular almost surely under generic choices of collocation points, PI-RPNN parameters, and boundary conditions.

\begin{proposition}[Almost sure regularity of the generalized pencil in the PI-RPNN basis]
\label{prop:regular_pencil}
Let $M=M_{\mathrm{int}}+M_{\mathrm{bc}}$, with $M_{\mathrm{int}}\geq 1$ and $M_{\mathrm{bc}}\geq 1$.
Assume the following.
\begin{enumerate}
\item The collocation points $\{\bm{x}_i\}_{i=1}^M\subset\mathbb{R}^d$ are distinct. They may be fixed or random; in the latter case, they are distinct almost surely, and the assumptions below are imposed conditionally on their realization.
The neuron parameters $\bm{\theta}_j=(\bm{\alpha}_j,\beta_j)$ are sampled independently from absolutely continuous distributions having densities with respect to Lebesgue
measure on a common connected open set
$\mathcal H\subset\mathbb{R}^{d+1}$.
\item The activation $\psi:\mathbb{R}\to\mathbb{R}$ is real analytic, nonpolynomial. Define the single-feature evaluation vector and the collocation matrix by
\begin{equation}
\bm{z}(\bm{\theta})
=\bigl(\psi(\bm{\alpha}\cdot\bm{x}_1+\beta),\ldots,
\psi(\bm{\alpha}\cdot\bm{x}_M+\beta)\bigr)^\top,
\qquad
\mathbf\Psi
=\bigl[\bm{z}(\bm{\theta}_1)\ \cdots\ \bm{z}(\bm{\theta}_N)\bigr].
\label{eq:regular_feature_vectors}
\end{equation}
\item The boundary operator $\mathcal B[u]$ {\color{black}is differentiable, and its associated Fréchet derivative $\nabla_u \mathcal B$, evaluated at the solution, is represented by a matrix $\mathbf{D}=(\mathbf{D}_{\mathrm{int}}, \mathbf{D}_{\mathrm{bc}})\in \mathbb{R}^{M_{\mathrm{bc}} \times M}$ in the discrete collocation setting. We assume that $\mathbf{D}$ is \emph{non-degenerate} in the sense that this matrix, restricted to its boundary-indexed columns, $\mathbf{D}_{\mathrm{bc}}\in\mathbb{R}^{M_{\mathrm{bc}} \times M_{\mathrm{bc}}}$, is invertible.}
\color{black}
\item The error \(\bm{e}(\bm{\theta}) = \widetilde{\bm{u}}(\bm{\theta}) - \bm{u}^{\ast}\) between the idealized discrete solution and the PI-RPNN field is sufficiently small that
\begin{equation}
\left\|
\mathbf{D}_{\mathrm{bc}}^{-1}
\left[
(\mathbf{H}_u \cdot \bm{e}(\bm{\theta}))_{\mathrm{bc}}
+
O(\|\bm{e}(\bm{\theta})\|^2)
\right]
\right\|
<
1,
\qquad
\text{for all } \bm{\theta} \in \mathcal{H},
\label{eq:perturbation_condition}
\end{equation}
where \(\mathbf{H}_u\) is the second-derivative tensor of the residual evaluated at \(\bm{u}^{\ast}\), and \(\mathbf{D}_{\mathrm{bc}}\) is the idealized boundary Jacobian. For linear boundary operators, \((\mathbf{H}_u)_{\mathrm{bc}} = \mathbf{0}\) and the condition is automatically satisfied.
\end{enumerate}
Let $\mathbf\Psi^\dagger$ be the exact, untruncated Moore--Penrose pseudoinverse, and set
\(
\widetilde{\mathbf J}_u=\mathbf J_w\mathbf\Psi^\dagger.
\)
Then the pencil $(\widetilde{\mathbf J}_u,\mathbf B)$, where
$\mathbf B=\operatorname{diag}(\mathbf I_{M_{\mathrm{int}}},
\mathbf 0_{M_{\mathrm{bc}}})$ is the selector in Eq.~\eqref{eq:B_def},
is \emph{regular almost surely} \color{black}under any of the following conditions:
\begin{enumerate}
\item[(a)] \emph{Full row-rank regime:} $N\geq M$.
\item[(b)] \emph{Dirichlet-type boundary conditions:}
$M_{\mathrm{bc}}<N<M$ and the boundary operator depends only on the boundary values of \(u\),
i.e., $\mathbf D_{\mathrm{int}}=\mathbf 0$ and $\widetilde{\mathbf{J}}_{\mathrm{int,bc}}(\bm \theta)=\mathbf 0$, $\forall \bm \theta\in\mathcal{H}$.
\item[(c)] \emph{General linear boundary conditions:}
$M_{\mathrm{bc}}<N<M$, and the boundary operator is linear with a fixed discrete derivative ${\mathbf{D}}$ once the collocation points are fixed. In addition, the single-feature evaluation vectors satisfy
\begin{equation}
\operatorname{span}\left\{
\bm{z}(\bm{\theta})\bm{z}(\bm{\theta})^\top:
\bm{\theta}\in\mathcal H\right\}
=\operatorname{Sym}_M,
\label{eq:quadratic_feature_span}
\end{equation}
where $\operatorname{Sym}_M$ is the vector space of real symmetric $M\times M$ matrices.
\end{enumerate}
\end{proposition}

\color{black}
\begin{proof}[Sketch of the proof]
From the block structure of $\mathbf{B}$, a simple linear algebraic argument based on Schur complement~\cite{zhang2006schur} (see Lemma~\ref{lem:block_pencil} in Appendix~\ref{app:proof:prop_4}) shows that the invertibility of the lower-right block \(\widetilde{\mathbf{J}}_{bc,bc}\) is sufficient for the regularity of the pencil \((\widetilde{\mathbf{J}}_u,\mathbf{B})\).

Using the factorization $\widetilde{\mathbf{J}}_u = \mathbf{J}_w \mathbf{\Psi}^\dagger$ and the chain rule $\mathbf{J}_w^{\mathrm{bc}} = {\widetilde{\mathbf{D}}} \, \mathbf{\Psi}$, the invertibility of $\widetilde{\mathbf{J}}_{bc,bc}$ is shown to hold almost surely. For $N\ge M$, $\mathbf{\Psi}$ has full row rank almost surely, so
$\mathbf{\Psi}\mathbf{\Psi}^\dagger=\mathbf{I}_M$ and hence $\widetilde{\mathbf{J}}_{bc,bc}={\widetilde{\mathbf{D}}}_{\mathrm{bc}}$, which
is invertible by Assumption~3-4. For $N<M$ and Dirichlet-type boundary conditions, this follows directly from the full row rank of the boundary collocation submatrix $\mathbf{\Psi}^{\mathrm{bc}}$ (Assumptions 1--2 and refs.~\cite{ito1996nonlinearity, huang2006extreme,fabiani2025random}, also proven in Lemma~\ref{lem:analytic_features_full_rank} of~\ref{app:proof:prop_4}). For $N<M$ and general linear boundary conditions, the proof extends the pencil determinant to a polynomial $\mathcal{P}(\mathbf{X})$ on all arbitrary matrix $\mathbf{X}\in\mathbb{R}^{M\times N}$. Under condition~(c) such polynomial is shown to be not identically zero; hence its zero set has measure zero~\cite{mityagin2020zero}, giving almost-sure regularity.
A complete proof is provided in Appendix~\ref{app:proof:prop_4}.
\end{proof}

The above proposition is stated to guarantee that the generalized eigenproblem is regular and thus solvable
with standard algorithms, such as those implemented in LAPACK or ARPACK.
\color{black}
It is worth noting that Condition~(c) of Proposition~\ref{prop:regular_pencil} concerns a non-trivial property of the activation function family, stronger than classical universal approximation.
Nevertheless, it is mild: Lemma~\ref{lem:standard_activation_quadratic_span} in Appendix~\ref{app:proof:prop_4} verifies it for logistic sigmoid, hyperbolic tangent, and Gaussian ridge activations under an explicit condition on the collocation points.

\subsection{Ill-conditioning of the random-feature basis}
\label{sec:ill_conditioning}

Next, we show that the numerically truncated pseudo-inversion of the collocation matrix $\mathbf{\Psi}$ produces a \emph{cluster of near-zero eigenvalues} in $\widetilde{\mathbf{J}}_u$, corresponding to directions outside the range of $\mathbf{\Psi}^{\dagger}$. Even though $\mathbf{\Psi}$ is theoretically of full row or column rank~\cite{huang2006extreme,fabiani2025random}, its singular values decay very rapidly, especially for \emph{analytic activation functions} such as sigmoids, $\tanh$, or Gaussian RBFs. The rapid decay of singular values --predicted by classical results on compact integral operators~\cite{little1984eigenvalues, castro2020super} and here proven rigorously in both asymptotic and finite-sample regimes-- explains the numerical rank deficiency of the computed $\widetilde{\mathbf{J}}_u$ and the presence of spurious near-zero eigenmodes.

\begin{proposition}[Singular value decay of the collocation matrix]
\label{prop:Psi_decay}
Let $\Omega \subset \R^d$ be a compact domain, and let
$\{\psi_j\}_{j=1}^N$ denote the PI--RPNN basis functions with internal parameters $\bm{\theta}=(\bm\alpha,\bm\beta)$ drawn independently from a continuous distribution.
Let $\bm{x}_1,\dots,\bm{x}_M$ be i.i.d.\ samples from a probability measure $\mu$ supported on $\Omega$, independently of the internal parameters.
Assume that these distributions are fixed as $M,N$ vary,
that $|\psi(\bm x;\bm\theta)|\le1$ on $\Omega$, and that
the random features admit holomorphic extensions to a common open neighborhood $U\subset \C^d$ of $\Omega$, satisfying
\begin{equation}
\sup_{\bm z\in U}|\psi(\bm z;\bm\theta)|\le A<\infty
\qquad\text{for almost every }\bm\theta.
\label{eq:uniform_holomorphic_features}
\end{equation}
Let $\lambda_1>0$ denote the leading eigenvalue of the
$L^2(\mu)$ integral operator associated with
\begin{equation}
k(\bm x,\bm y)
=\mathbb E_{\bm\theta}
\left[\psi(\bm x;\bm\theta)\psi(\bm y;\bm\theta)\right].
\end{equation}

Let $\mathbf{\Psi} \in \R^{M \times N}$ denote the collocation matrix with entries
$(\mathbf{\Psi})_{ij} = \psi_j(\bm{x}_i)$, and let
\(
\hat{\sigma}_1 \ge \hat{\sigma}_2 \ge \cdots \ge \hat{\sigma}_{\min(M,N)} \ge 0
\)
denote its singular values, ordered decreasingly.

\color{black}
For a compact interval $\Omega\subset\mathbb R$,
there exists $R_0>1$ such that, for every fixed
$1<R<R_0$ and $\delta\in(0,1)$, the following hold.

\begin{enumerate}
\item[(i)]
\emph{Asymptotic regime ($N\to\infty$, $M$ fixed).}
If
\begin{equation}
\sqrt{\frac{3}{M\delta}}\le\frac{\lambda_1}{2},
\label{eq:asymptotic_sample_condition}
\end{equation}
then, with probability at least $1-\delta$ over the
collocation points, almost surely over the random features,
\begin{equation}
\lim_{N\to\infty}
\frac{\hat\sigma_i(\mathbf\Psi)}
     {\hat\sigma_1(\mathbf\Psi)}
\le C_\delta R^{-i/2}
=O(R^{-i/2}),
\qquad 1\le i\le M.
\end{equation}

\item[(ii)]
\emph{Finite-sample bound.}
For finite $M,N$ satisfying
\begin{equation}
\sqrt{\frac{3}{M\delta}}
+\sqrt{\frac{3}{N\delta}}
\le\frac{\lambda_1}{2},
\label{eq:finite_sample_condition}
\end{equation}
with probability at least $1-\delta$ over the random
features and collocation points,
\begin{equation}
\frac{\hat\sigma_i(\mathbf\Psi)}
     {\hat\sigma_1(\mathbf\Psi)}
\le
C_{1,\delta}R^{-i/2}
+C_{2,\delta}R^{-i/4}
=O(R^{-i/4}),
\qquad 1\le i\le\min(M,N).
\end{equation}
\end{enumerate}

The estimates hold simultaneously over the indicated
indices. The constants may depend on the fixed kernel,
$\mu$, $\delta$, and the chosen exponential rate, but
not on $i$, $M$, or $N$ within the stated sample-size
regimes.

For $\Omega=\mathbb S^{d-1}$, $d\ge2$, the same conclusions
hold with $i$ replaced by $i^{1/(d-1)}$ and with an
appropriate rate $R>1$. In particular, the singular-value
ratios have super-algebraic decay on spheres.
\begin{proof}[Proof sketch]
The proof rests on three main ingredients: (i)~the uniform analyticity of the
basis functions, which implies analyticity of the expected kernel and, in one
dimension, exponential decay of the eigenvalues of its integral operator by the
result of Little and Reade~\cite{little1984eigenvalues};
(ii)~finite-sample spectral perturbation bounds for kernel matrices with
uniformly bounded eigenfunctions~\cite{braun2006accurate}; and
(iii)~the fact that the empirical Gram matrix $\mathbf G_{M,N}$ is a Monte Carlo
approximation of the kernel matrix $\mathbf K_M$, with spectral convergence
following from the law of large numbers and Weyl's inequality.

For finite $N$, the argument proceeds by selecting a truncation index $r=i-1$
for the matrices $\mathbf{G}^{[r]}_{M,N}$ and $\mathbf{K}^{[r]}_M$ obtained from
the first $r$ eigenfunctions
$\{\varphi_\ell\}_{\ell=1}^r$ of the kernel.
Since both truncated matrices have rank at most $r=i-1$, their $i$-th
eigenvalues vanish. The error thus reduces to bounding the truncation norm
$\|\mathbf{G}_{M,N}-\mathbf{G}^{[i-1]}_{M,N}\|_2$. This is achieved by expanding
the difference
$\psi_j\psi_j^\top - \psi_j^{[r]}(\psi_j^{[r]})^\top$ into tail-tail and cross
terms,
where $\psi_j^{[r]}(x) := \sum_{\ell=1}^r a_{\ell j}\phi_\ell(x)$ and
$a_{\ell j} := \langle \psi_j, \phi_\ell \rangle_{\mathcal{H}_k}$.
Orthogonality identifies the expected squared truncation error with the
eigenvalue tail $\Lambda_{>r}=\sum_{\ell>r}\lambda_\ell$.
A second-moment Chebyshev estimate controls the sampled truncation error with
high probability. The cross terms are handled using Cauchy--Schwarz and the
triangle inequality
$\|\bm\psi_j^{[r]}\|_2
\le\|\bm\psi_j\|_2
+\|\bm\psi_j-\bm\psi_j^{[r]}\|_2$.
A lower bound on the leading empirical eigenvalue then yields the relative
singular-value estimates. The complete argument, including the extension to
spheres~\cite{castro2020super}, is given in Appendix~\ref{app:proof_5}.
\end{proof}
\end{proposition}

\subsection{Generalized eigenproblem via a shift--invert Arnoldi strategy}
For the reasons discussed in the previous section, namely, the ill-conditioning and potential numerical rank-deficiency of the collocation matrix $\mathbf{\Psi}$, we avoid forming $\widetilde{\mathbf{J}}_u$ explicitly -- unlike the ``naive" approach in Eq.~\eqref{eq:naive} -- and instead re-formulate the generalized eigenproblem, as in Eq.~\eqref{eq:gen_eigRPNN}, directly in the \emph{weight space}, with the aim of providing a numerically stable and efficient approach for computing the eigenvalue spectrum associated with admissible perturbations.

In the PI-RPNN framework, any weight perturbation $\bm{v}\in\R^N$ induces a physical perturbation
\begin{equation}
\bm{\phi} \;=\; \mathbf{\Psi} \bm{v} \in \R^M.
\end{equation}
Substituting $\bm{\phi}=\mathbf{\Psi}\bm{v}$ into Eq.~\eqref{eq:gen_eigRPNN} yields
\begin{equation}
\widetilde{\mathbf{J}}_u \mathbf{\Psi} \bm{v} = \lambda \mathbf{B} \mathbf{\Psi} \bm{v},
\end{equation}
and hence using the chain rule $\widetilde{\mathbf{J}}_u\Psi=\mathbf{J}_w$ as in Eq.~\eqref{eq:Jw_Ju_relation} the \emph{weight--space} generalized eigenproblem
\begin{equation}
\label{eq:JW_weight_eig}
\mathbf{J}_w\, \bm{v} = \lambda\, \widetilde{\mathbf{B}}_{\mathbf{\Psi}}\, \bm{v},\qquad \widetilde{\mathbf{B}}_{\mathbf{\Psi}} := \mathbf{B}\,\mathbf{\Psi} \in \R^{M\times N}.
\end{equation}
Once an eigenvector $\bm{v}$ is found, the associated physical eigenfunction is reconstructed as $\bm{\phi}=\mathbf{\Psi}\bm{v}$.

Eq.~\eqref{eq:JW_weight_eig} is written with matrices of size $M\times N$. Typical regimes are $M\ge N$ (overdetermined collocation) or $M\le N$. The equation is homogeneous in $\R^M$; nontrivial solutions $\bm{v}$ exist only when the pencil $\mathbf{J}_w-\lambda \widetilde{\mathbf{B}}_{\mathbf{\Psi}}$ is rank-deficient for $\lambda$, i.e.\ when $\mathrm{rank}(\mathbf{J}_w-\lambda \widetilde{\mathbf{B}}_{\mathbf{\Psi}})<N$ in the relevant subspace. To operate with standard eigensolvers we will map~\eqref{eq:JW_weight_eig} to an $N\times N$ operator acting on $\bm{v}$. 

From~\eqref{eq:JW_weight_eig} we seek pairs $(\lambda,v)\neq(0,0)$ such that
\begin{equation}
(\mathbf{J}_w - \lambda \widetilde{\mathbf{B}}_{\mathbf{\Psi}})\, \bm{v}  = 0.
\end{equation}
A robust route to compute eigenvalues nearest a target shift $\sigma\in\C$ is the \emph{shift--invert} transform applied to the generalized pencil:
\begin{equation}
T(\sigma) \;:=\; (\mathbf{J}_w - \sigma \widetilde{\mathbf{B}}_{\mathbf{\Psi}})^{\dagger}\, \widetilde{\mathbf{B}}_{\mathbf{\Psi}},
\end{equation}
where $(\mathbf{J}_w - \sigma \widetilde{\mathbf{B}}_{\mathbf{\Psi}})^{\dagger}$ denotes a numerically stabilized (truncated) pseudo-inverse of the rectangular matrix $\mathbf{J}_w - \sigma \widetilde{\mathbf{B}}_{\mathbf{\Psi}}$. If $\bm{v}$ satisfies Eq.~\eqref{eq:JW_weight_eig} with eigenvalue $\lambda\neq\sigma$, then
\begin{equation}
T(\sigma)\, \bm{v}  \;=\; \frac{1}{\lambda-\sigma}\, \bm{v} .
\end{equation}
Hence eigenpairs $(\theta,v)$ of $T(\sigma)$ are related to original eigenpairs by
\begin{equation}
\theta = \frac{1}{\lambda-\sigma}\quad\Longrightarrow\quad \lambda = \sigma + \frac{1}{\theta}.
\end{equation}
Thus computing a few largest-magnitude eigenvalues $\theta$ of $T(\sigma)$ via Arnoldi yields eigenvalues $\lambda$ of the original problem closest to the shift $\sigma$.
For theoretical background on shift--invert and Arnoldi methodssee refs.~\cite{saad2011numerical,ruhe1984rational,scott1982inverted,ericsson1980spectral}.

\textit{Practical matrix-free implementation.}
Directly forming $(\mathbf{J}_w-\sigma \widetilde{\mathbf{B}}_{\mathbf{\Psi}})^\dagger$ can propagate the numerical errors associated with small singular values; instead we compute a truncated SVD of the (possibly rectangular) matrix
\begin{equation}
\mathbf{A}_\sigma := \mathbf{J}_w - \sigma \widetilde{\mathbf{B}}_{\mathbf{\Psi}} \in\R^{M\times N}.
\end{equation}
Let the (thin) SVD be
\begin{equation}
\mathbf{A}_\sigma \approx \mathbf{U} \mathbf{\Sigma} \mathbf{V}^{\mathrm{H}},\qquad \mathbf{U}\in\R^{M\times r},\ \mathbf{\Sigma}\in\R^{r\times r},\ \mathbf{V}\in\R^{N\times r},
\end{equation}
with singular values $s_1\ge\cdots\ge s_r>0$ above a tolerance $\tau$. The (regularized) pseudo-inverse is
\begin{equation}
\mathbf{A}_\sigma^{\dagger} \approx \mathbf{V} \mathbf{\Sigma}^{-1} \mathbf{U}^{\mathrm{H}} \in\R^{N\times M}.
\end{equation}
Define the $N\times N$ operator
\begin{equation}
\mathcal{T}_\sigma := \mathbf{A}_\sigma^{\dagger} \widetilde{\mathbf{B}}_{\mathbf{\Psi}} \in\R^{N\times N}.
\end{equation}
This operator is never formed explicitly; we store $\mathbf{U}, \mathbf{\Sigma},\mathbf{V}$ (and $\widetilde{\mathbf{B}}_{\mathbf{\Psi}}$ or routines to apply $\widetilde{\mathbf{B}}_{\mathbf{\Psi}}$ and $\mathbf{A}_\sigma$ to vectors) and implement
\begin{equation}
\bm y \mapsto \mathcal{T}_\sigma\, \bm y = \mathbf{V}\bigl(\mathbf{\Sigma}^{-1}(\mathbf{U}^{\mathrm{H}} (\widetilde{\mathbf{B}}_{\mathbf{\Psi}} \bm y))\bigr).
\label{eq:shift_invert_operator}
\end{equation}
By keeping the SVD factors separate and applying $\mathbf{\Sigma}^{-1}$ only in the reduced space of size $r$, we avoid explicitly constructing the dense and potentially ill-conditioned matrix $\mathbf{\Psi}^\dagger$ (or $\mathbf{A}_\sigma^\dagger$), which would otherwise amplify rounding errors from small singular values. This factored approach allows the use of smaller truncation tolerances (e.g., $\tau = 10^{-10}$--$10^{-12}$) while maintaining numerical stability, in contrast to forming the full pseudo-inverse where such tight thresholds could lead to severe ill-conditioning.
This matrix-free action is directly suitable for a Krylov--Arnoldi solver (e.g., \texttt{eigs} in \texttt{MATLAB} or any other software leveraging ARPACK libraries) that only requires matrix–vector products and relies on standard linear algebra routines for efficiency.

\textit{Algorithm.}
\begin{enumerate}
  \item Choose a shift $\sigma$ near the part of the eigenvalue spectrum of interest.
  \item Form $\mathbf{J}_w$ and $\widetilde{\mathbf{B}}_{\mathbf{\Psi}}=B\mathbf{\Psi}$ (we only need routines to apply these matrices to vectors and, for the SVD, to assemble $\mathbf{A}_\sigma$).
  \item Compute a truncated SVD $\mathbf{A}_\sigma \approx \mathbf{U} \mathbf{\Sigma} \mathbf{V}^{\mathrm{H}}$ with tolerance $\tau$ (drop singular values $s_i<\tau$).
  \item Define the matrix-free operator $\mathcal{T}_\sigma$ via $\bm{v}\mapsto \mathbf{V}(\mathbf{S}^{-1}(\mathbf{U}^{\mathrm{H}}(\widetilde{\mathbf{B}}_{\mathbf{\Psi}} \bm v)))$.
  \item Use Arnoldi (e.g., \texttt{eigs} in \texttt{MATLAB} with function handle) to compute the leading eigenpairs $(\theta_i,v_i)$ of $\mathcal{T}_\sigma$.
  \item Recover original eigenvalues $\lambda_i = \sigma + 1/\theta_i$ and physical eigenfunctions $\bm{\phi}_i = \mathbf{\Psi} \bm v_i$.
\end{enumerate}

\begin{remark}
One might consider exploiting the factorization $\widetilde{\mathbf{J}}_u\bm{\phi} = \mathbf{J}_w \Psi^{\dagger}\bm{\phi}$ and setting $\phi = \Psi \bm{v}$ to obtain the rectangular generalized eigenproblem $\mathbf{J}_w \bm{v}  = \lambda \widetilde{\mathbf{B}}_{\Psi} \bm{v}$. However, this system is rectangular ($M \times N$) and cannot be solved directly with standard eigensolvers. Our shift-invert Arnoldi method circumvents this by working directly in the weight space with a square operator $\mathcal{T}_\sigma \in \mathbb{R}^{N \times N}$.
\end{remark}

\subsection{Implementation Details: Selection of Random Weights and Biases}
\label{sec:random_agnostic_selection}
In this section, we examine the \emph{a priori} selection of the internal weights and biases for the RPNN approximator. Although the full-rank result applies under continuous sampling distributions, practical implementations require careful parameter selection. Random projection bases are typically non-orthogonal and often redundant; even with parsimonious sampling, this can produce a severely {numerically rank-deficient and ill-conditioned} collocation matrix $\mathbf{\Psi}$ in Eq.~\eqref{eq:RPNN_solve}.

For one-dimensional intervals $[a,b]$, as discussed in previous works~\cite{fabiani2023parsimonious,fabiani2021numerical, calabro2021extreme, galaris2022numerical}, an effective strategy for selecting basis functions is to ensure that the activation functions $\psi_j$ have centers $c_j=-\beta_j/\alpha_j$ located within the domain of interest. The optimal bounds of the parameter distribution depend on both the choice of activation function and the specific problem, and may vary for tasks such as simple regression, problems involving spatial derivatives, or time-dependent dynamics with hidden or challenging stability constraints.

For the remainder of this work, we focus on the logistic sigmoid as a representative smooth activation function for shallow, single-hidden-layer randomized networks:
\begin{equation}
    \psi_j(x) =\frac{1}{1+\text{exp}(- \xb\cdot\alphab_j -\beta_j)}.
    \label{eq:logistic_sigmoid}
\end{equation}

A practical feature of the shallow single-hidden-layer architecture is that all spatial derivatives required by the PDE residuals are available in closed form and can be evaluated analytically. 
Specifically, the first and second derivatives with respect to the $k$-th component of $\xb$, $x_k$ are given by:
\begin{equation}
\label{eq:der:SF}
\begin{split}
    \frac{\partial}{\partial x_{k}}\psi_j(\xb)&= \alpha_{j,k}\frac{\text{exp}(z_j)}{(1+\text{exp}(z_j))^2},\\
    \frac{\partial^2}{\partial x_{k}^2}\psi_j(\xb)&= \alpha_{j,k}^2\frac{\text{exp}(z_j) \cdot (1-\text{exp}(z_j))}{(1+\text{exp}(z_j))^3},
\end{split}
\end{equation}
where $z_j=\xb \cdot \alphab_j +\beta_j$.

For the one-dimensional case ($d=1$), the logistic sigmoid $\psi$ is a monotonic function such that $\lim_{x\rightarrow +\infty} \psi(x)=1$ and $\lim_{x\rightarrow -\infty} \psi(x)=0$.
This function has an inflection point, that we call \emph{center} $c_j$, such that $\psi(\alpha_{j} \cdot c_j + \beta_j)=\frac{1}{2}$.
The slope $\alpha_{j}$ controls the sharpness of the transition: large $|\alpha_{j}|\gg1$ yields a step-like function, whereas small $|\alpha_{j}|\ll1$ produces an almost constant function.
To avoid ``inactive" neurons over the domain $I=[a,b]$, $\alpha_j$ can be sampled from a uniform distribution within modest bounds relative to $|I|$ and the number of neurons $N$, and $\beta_j$ set to place centers at selected points $c_j$. While exact upper bounds $\alpha_U$ can be tuned for a given problem, for example in previous works~\cite{fabiani2023parsimonious,fabiani2021numerical, fabiani2024stability, calabro2021extreme}, values such as 
\begin{equation}
\alpha_j \sim \mathcal{U}\Bigl(-\alpha_U, \alpha_U\Bigr), \qquad \text{with } \alpha_U=\frac{\min(N,M)/5+4}{|I|},
\label{eq:upper_bound}
\end{equation}
have been used successfully. These choices are heuristically optimized for stationary nonlinear PDEs but are not meant as universal prescriptions.

For two-dimensional inputs ($d=2$), the sigmoid’s inflection forms a line of centers (the \emph{central direction} of the plane wave), and for a given center $\bm{c}_j=(c_{j,1},c_{j,2})$ the bias is set as
$ \beta_j = -\alpha_{j,1} c_{j,1} - \alpha_{j,2} c_{j,2}. $
Centers can be chosen systematically (e.g., equally spaced) or randomly sampled within the domain, with biases computed accordingly. These heuristically optimized ranges and placements have proven effective in applications to stationary nonlinear PDEs~\cite{fabiani2023parsimonious, fabiani2021numerical, calabro2021extreme, galaris2022numerical}.

In higher-dimensional settings, which we do not address here, one may consider data-driven geometrical sampling techniques~\cite{galaris2022numerical,bolager2024sampling} or randomized Fourier feature methods for kernel approximation~\cite{rahimi2007random}. Alternative approaches for optimizing hidden-layer weights have also been explored~\cite{dong2022computing}.

\section{Numerical Analysis Results: the Case Studies}
\label{sec:numerical_results}
The efficiency of the proposed numerical scheme is demonstrated through three benchmark nonlinear PDEs, namely (a) the one- and two-dimensional Liouville-Bratu-Gelfand problem; (b) The FitzHugh-Nagumo PDE; (c) The Allen-Cahn phase-field PDE.
These problems have been widely studied and used to model and analyze the behavior of many physical and chemical systems (e.g.,see refs.~\cite{allen1972ground, boyd1986analytical, fabiani2025enabling, fabiani2021numerical, galaris2022numerical}).

In all the computations with the proposed machine learning (PI-RPNN) scheme, the convergence criterion for Newton's iterations was the $L_2$ norm\footnote{The relative error is the $L_2$--norm of the difference between two successive solutions $||u(\bm{w})_{-1}-u(\bm{w})_{-2}||_2$.}.
Here, we do not perform comparisons with FEM, and use finite differences (FD) solely as a reference; detailed analyses of computational cost, network size, or collocation point convergence are omitted here, as they have been thoroughly addressed in Refs.~\cite{fabiani2021numerical, calabro2021extreme}, where PI-RPNN schemes have been shown to have comparable computational cost to FD and FEM  discretizations.

\begin{figure}[ht!]
    \centering
    \begin{tabular}{cc}
      \includegraphics[width=0.45\linewidth]{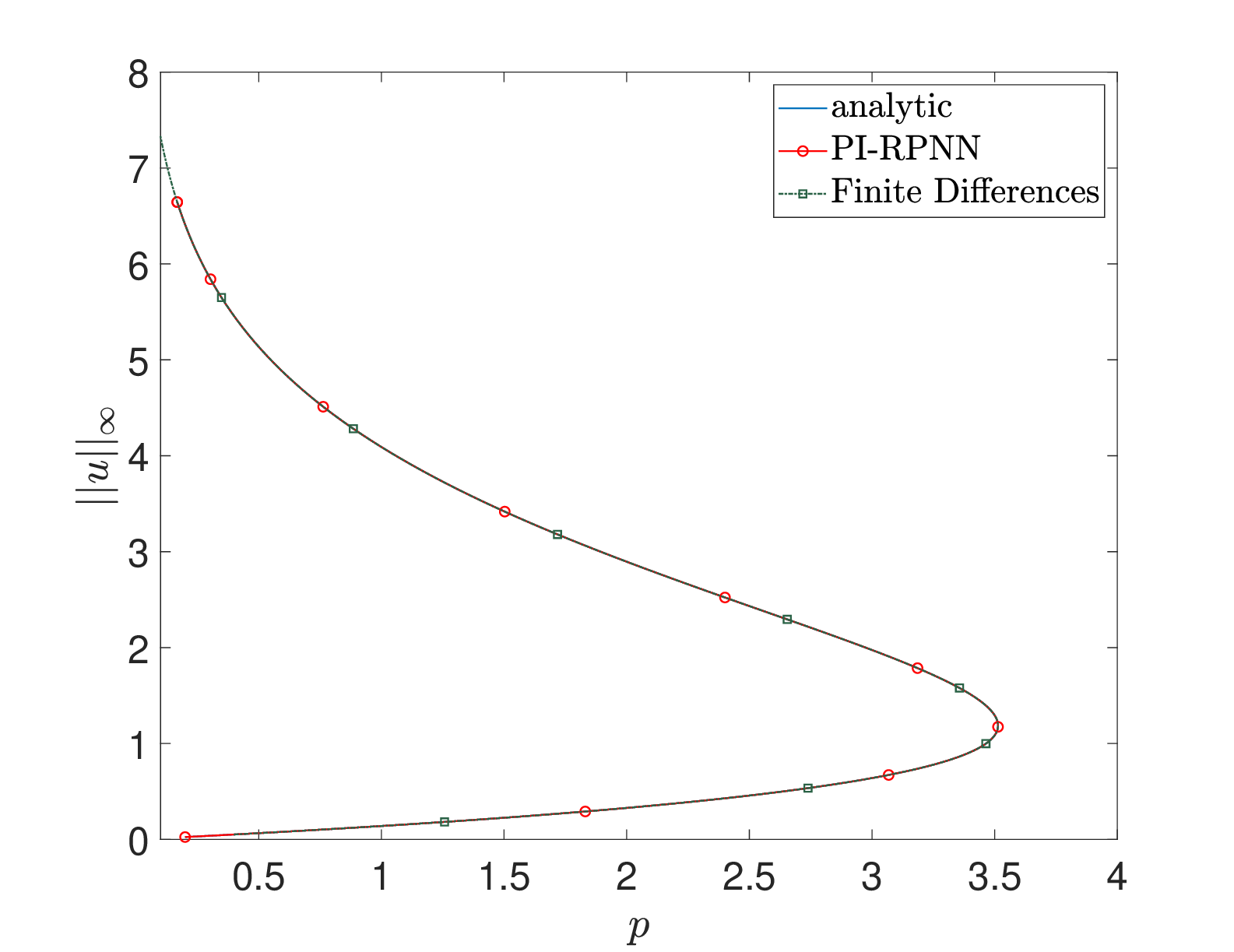}   &  \includegraphics[width=0.45\linewidth]{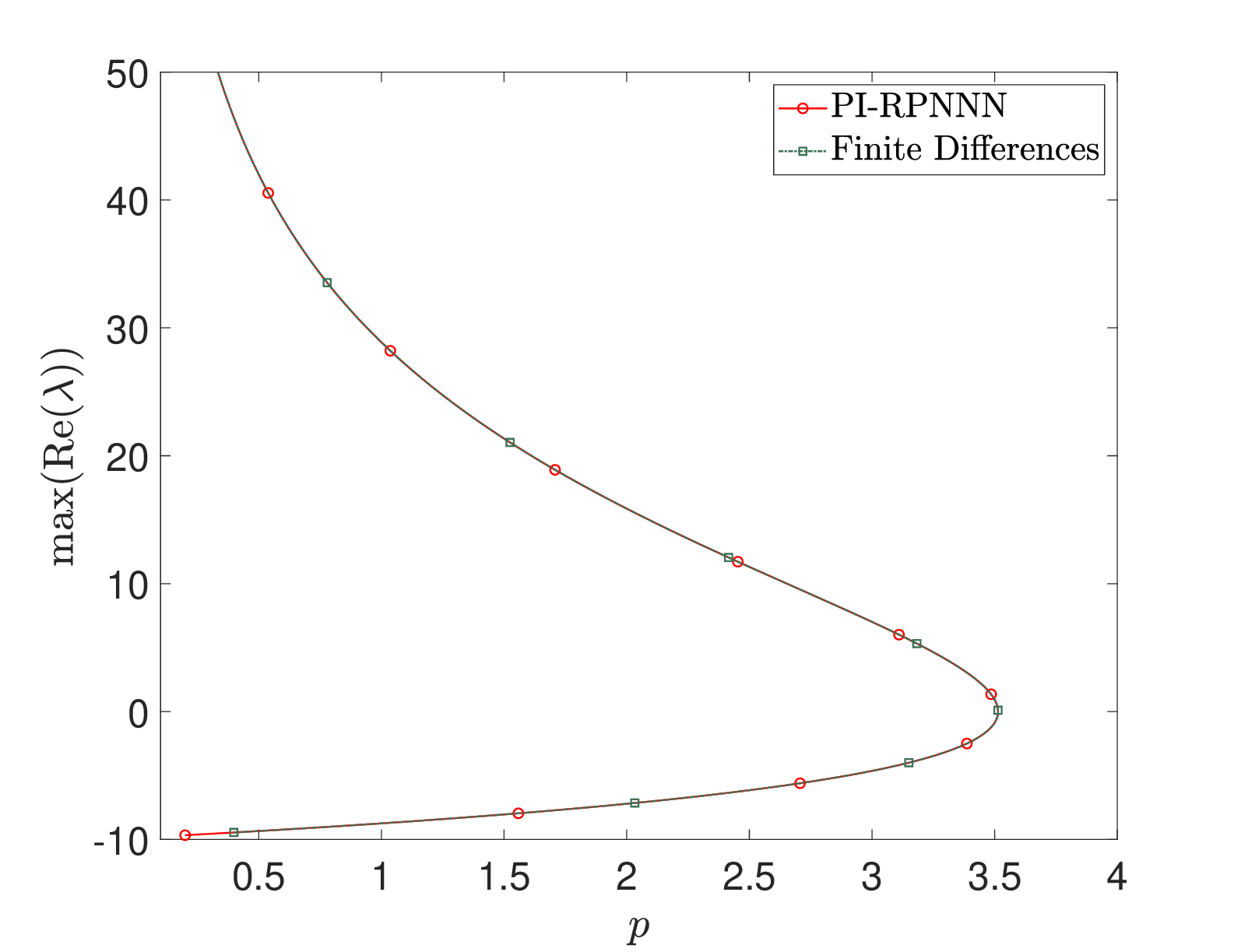} \\
        (a) & (b) \\
       \includegraphics[width=0.45\linewidth]{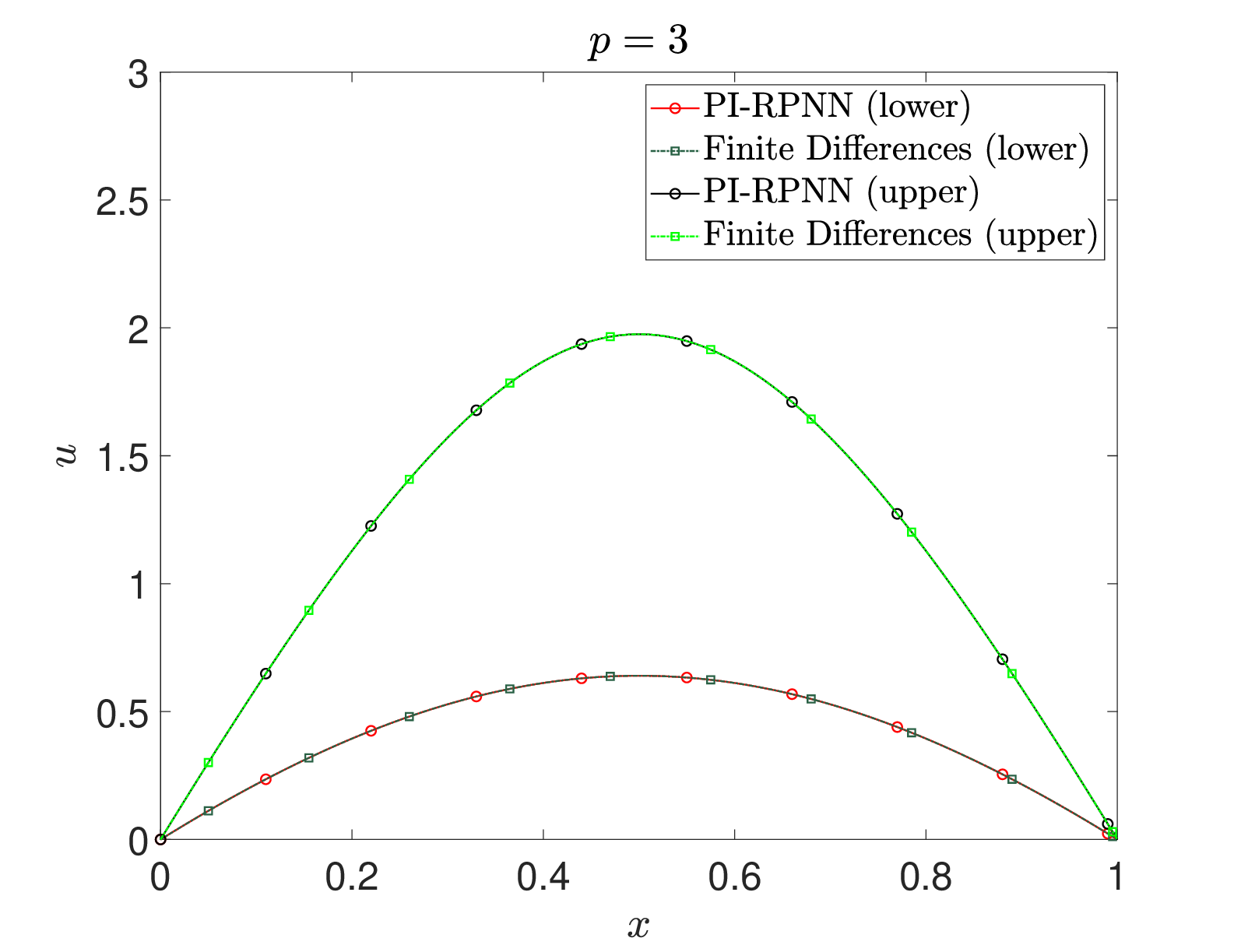}   &  \includegraphics[width=0.45\linewidth]{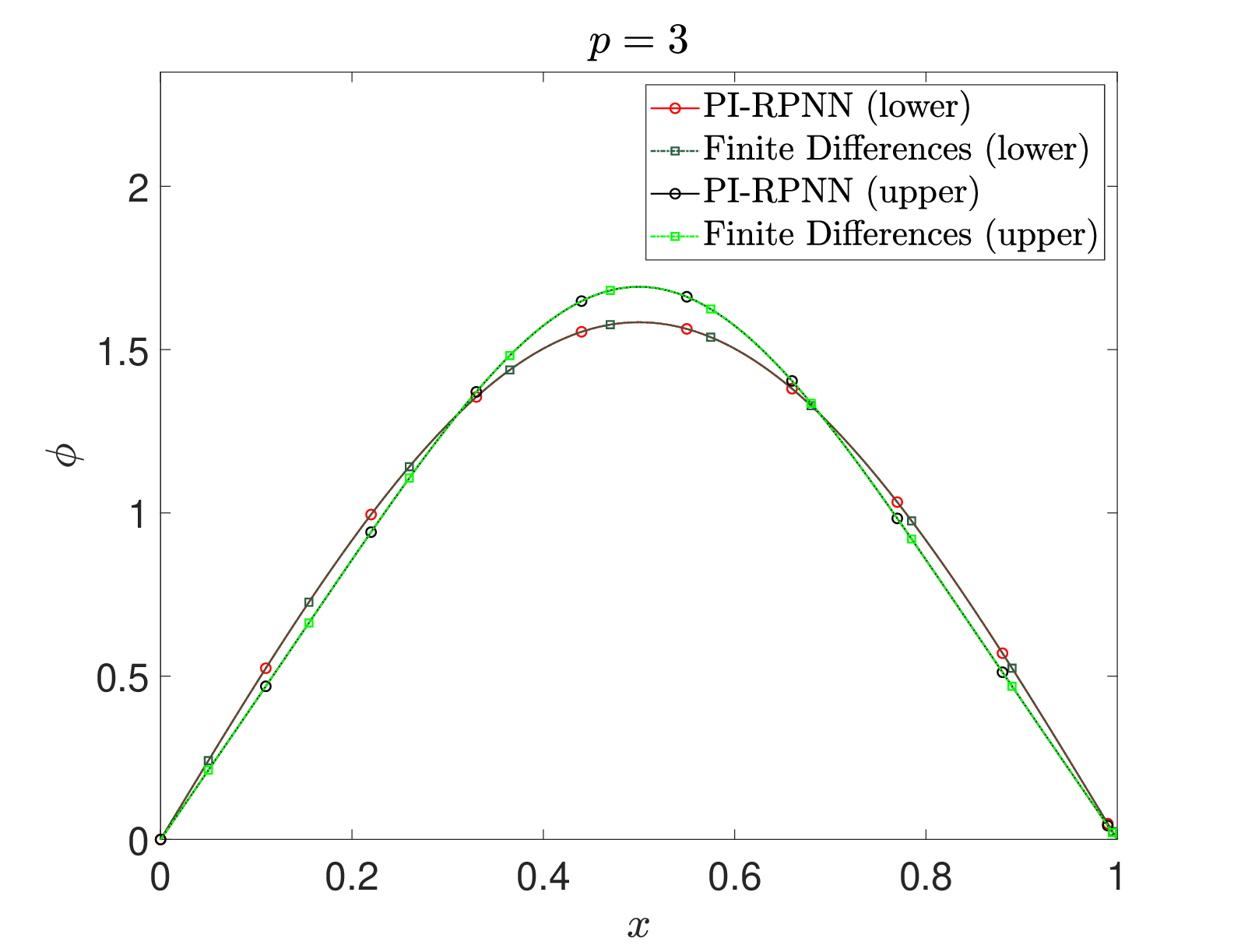} \\
        
        (c) & (d)
    \end{tabular}
    \caption{Numerical results for the one-dimensional Liouville-Bratu-Gelfand problem (\eqref{eq:Bratu}). (a) Bifurcation diagram;
    (b) Dominant eigenvalue of the corresponding generalized eigenproblem. 
    (c) Co-existing solution profiles at $p=3$ corresponding to the lower-stable branch and to the upper-unstable branch.
    (d) Eigenfunction $\phi$ associated with the dominant eigenvalue for the two coexisting solutions at $p=3$.
    The lower (stable) branch has dominant eigenvalue $\lambda_1 \approx -4.64$, whereas the upper (unstable) branch has dominant eigenvalue $\lambda_1 \approx 7.01$.
    }
\label{fig:bifdiagram_bratu1d}
\end{figure}

\subsection{Case study 1: (saddle-node) The one- and two-dimensional Liouville-Bratu-Gelfand Problem}\label{Sec:Bratu}
Our first illustrative example is a nonlinear diffusion-reaction problem, Liouville-Bratu-Gelfand PDE~\cite{boyd1986analytical}, that arises in modeling several physical and chemical systems, given by:
\begin{equation}
    \Delta u+p \exp(u)=0, \qquad \xb \in \Omega=[0,1]^d,
    \label{eq:Bratu}
\end{equation}
with homogeneous Dirichlet BCs $u(\xb,t)=0$, for $\xb \in \partial\Omega$.
Here we consider both the cases, $d=1$ and $d=2$.
The one-dimensional steady state problem ($d=1$, $\Omega=[0,1]$) admits an analytical solution~\cite{mohsen2014simple}, reading
\begin{equation}
\begin{split}
    u(x)=2\ln\frac{\cosh{\theta}}{\cosh{\theta (1-2x)}},
    \text{ such that } \cosh{\theta} = \frac{4\theta}{\sqrt{2 p}}.
\end{split}
\label{eq:sys1}
\end{equation}
It can be shown that when $0<p <p_c$, the problem admits two solution branches that meet in the saddle-node point $p_c \sim 3.513830719$ where the stability changes, whereas beyond $p_c$ no solutions exist~\cite{boyd1986analytical}. \par
For the two-dimensional case ($d=2$, $\Omega=[0,1]^2$), no exact analytical solution is known. Numerical studies, such as~\cite{hajipour2018accurate}, report a turning point near $p_c \approx 6.808124$.

\begin{figure}[ht!]
    \centering
    \begin{tabular}{cc}
    \includegraphics[width=0.45\linewidth]{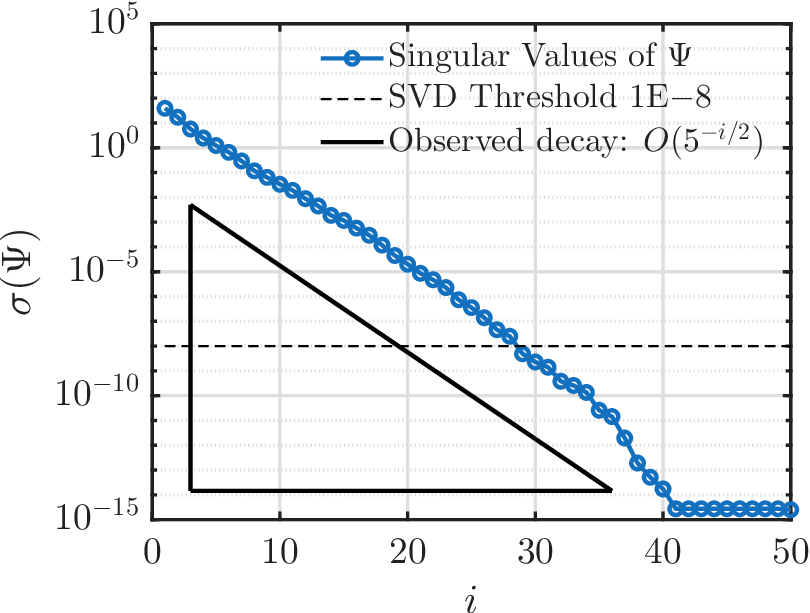}  & 
       \includegraphics[width=0.45\linewidth]{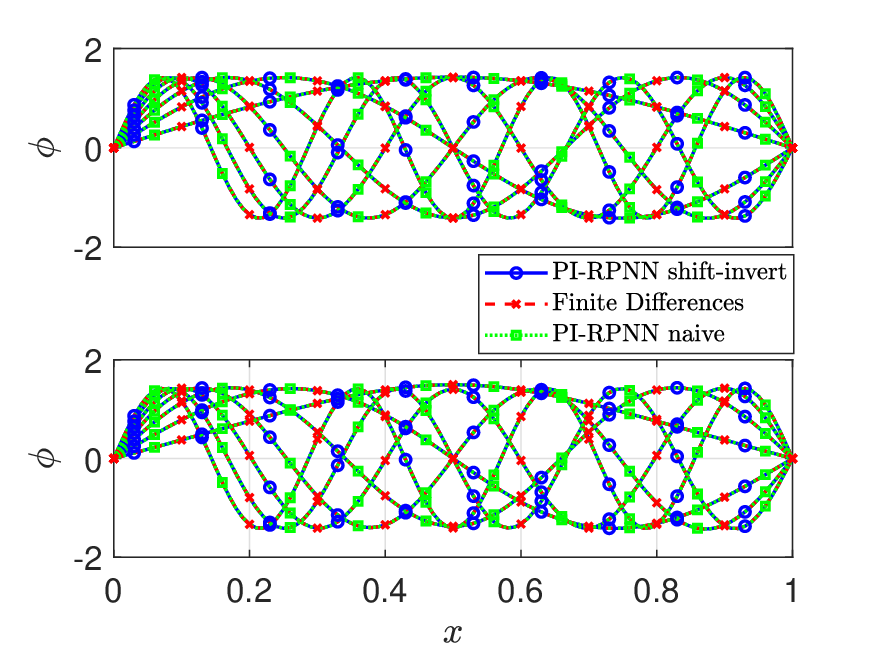}
       \\
       (a)  & (b) \\
       \includegraphics[width=0.45\linewidth]{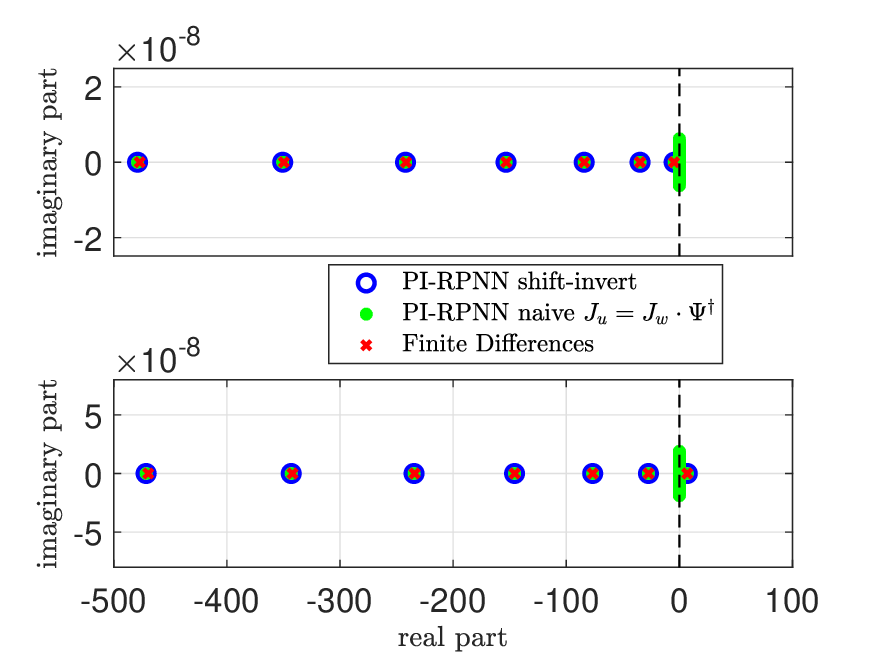} &
       \includegraphics[width=0.45\linewidth]{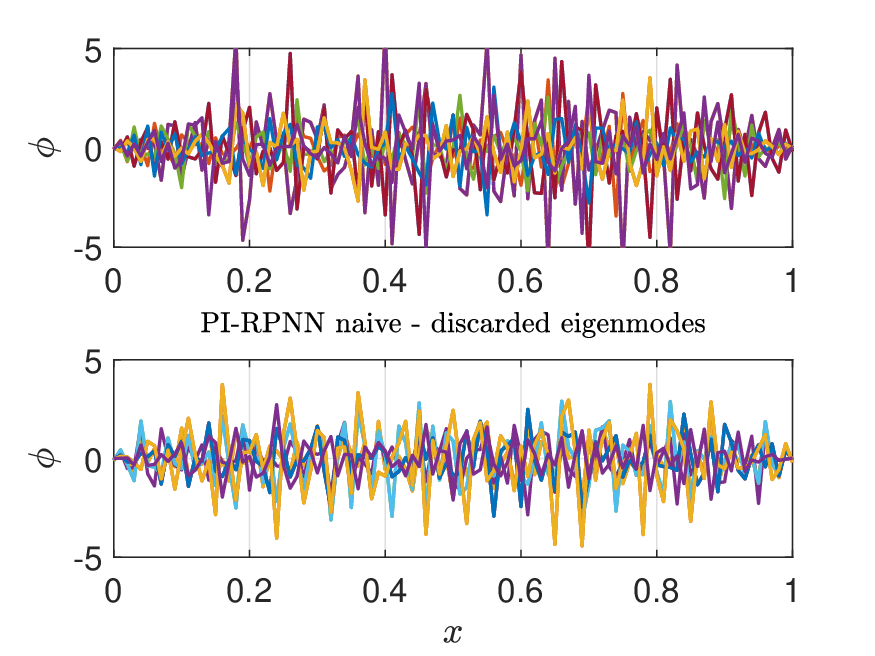} \\
       (c) & (d) 
    \end{tabular}
    \caption{Demonstrating the impact of collocation matrix rank-deficiency and the efficacy of the shift-invert formulation. Results {for both the stable and unstable solutions} of the 1D Liouville-Bratu-Gelfand problem at $p=3$. (a) Singular values of the random projection basis matrix $\mathbf{\Psi} \in {\R^{101 \times 50}}$, exhibiting an exponential decay with rate $\sigma_n \approx O(5^{-n/2})$, consistent with the theoretical prediction of Proposition~\ref{prop:Psi_decay}.
    (b) The dominant eigenvectors computed via finite differences (FD, reference), the proposed shift-invert PI-RPNN method, and a ``naive" PI-RPNN approach (using the approximation $\widetilde{\mathbf{J}}_u=\mathbf{J}_w\cdot \mathbf{\Psi}^{\dagger}$), are visually indistinguishable. (c) Corresponding leading eigenvalues. 
    (d) Spurious eigenmodes from the ``naive'' approach. Shown are ten representative eigenvectors corresponding to near-zero eigenvalues from the cluster in (c).
    For all panels (b), (c), and (d), the upper subplots correspond to the stable solution branch, whereas the lower ones correspond to the unstable solution branch.
    \label{fig:bratu_illcond}
    }
\end{figure}

Considering a set of collocation points $\xb_i\in\Omega$, $i=1,\dots,M_{\mathrm{int}}$, and using the PI-RPNN network function (~\ref{eq:RPNN}) as an ansatz in the Liouville-Bratu-Gelfand problem in Eq.~\eqref{eq:Bratu}, leads to the following system of equations:
\begin{equation}
    F_i(\bm{w},p)=\sum_{j=1}^N w_j \sum_{q=1}^d\frac{\partial^2\psi_j(\xb_i)}{\partial x_{q}^2}+p \, \text{exp}\left(\sum_{j=1}^N w_j \psi_j(\xb_i)\right)=0, \quad i=1,\dots,M_{\mathrm{int}}.
\end{equation}
Then considering points $\xb_k \in \partial \Omega$, $k=1,\dots,M_{\mathrm{bc}}$ we obtain for the BCs:
\begin{equation}
    F_k(\bm{w},p)=\sum_{j=1}^N w_j \psi_j(\xb_k)=0, \qquad k=1,\dots,M_{\mathrm{bc}}.
\end{equation}
Thus, the elements of the Jacobian matrix $\nabla_{\bm{w}} \Fb$ are given by:
\begin{equation}
\frac{\partial F_k}{\partial w_j}(\bm{w},p) =
\begin{cases}
\displaystyle \sum_{q=1}^d \frac{\partial^2 \psi_j(\xb_k)}{\partial x_{q}^2} + p \, \psi_j(\xb_k) \, \exp\Big(\sum_{{\ell}=1}^N w_{\ell} \psi_{\ell}(\xb_k)\Big), & \xb_k \in \Omega, \\
\displaystyle \psi_j(\xb_k), & \xb_k \in \partial\Omega.
\end{cases}
\end{equation}
The application of Newton's methodin Eq.~\eqref{eq:Newton_iter} is straightforward, as all derivatives of the basis functions required by the Jacobian $\nabla_{\bm{w}} \Fb$ can be computed analytically in closed form (see Eq.~\eqref{eq:der:SF}). %

\begin{figure}[ht!]
    \centering
    \begin{tabular}{cc}
      \includegraphics[width=0.45\linewidth]{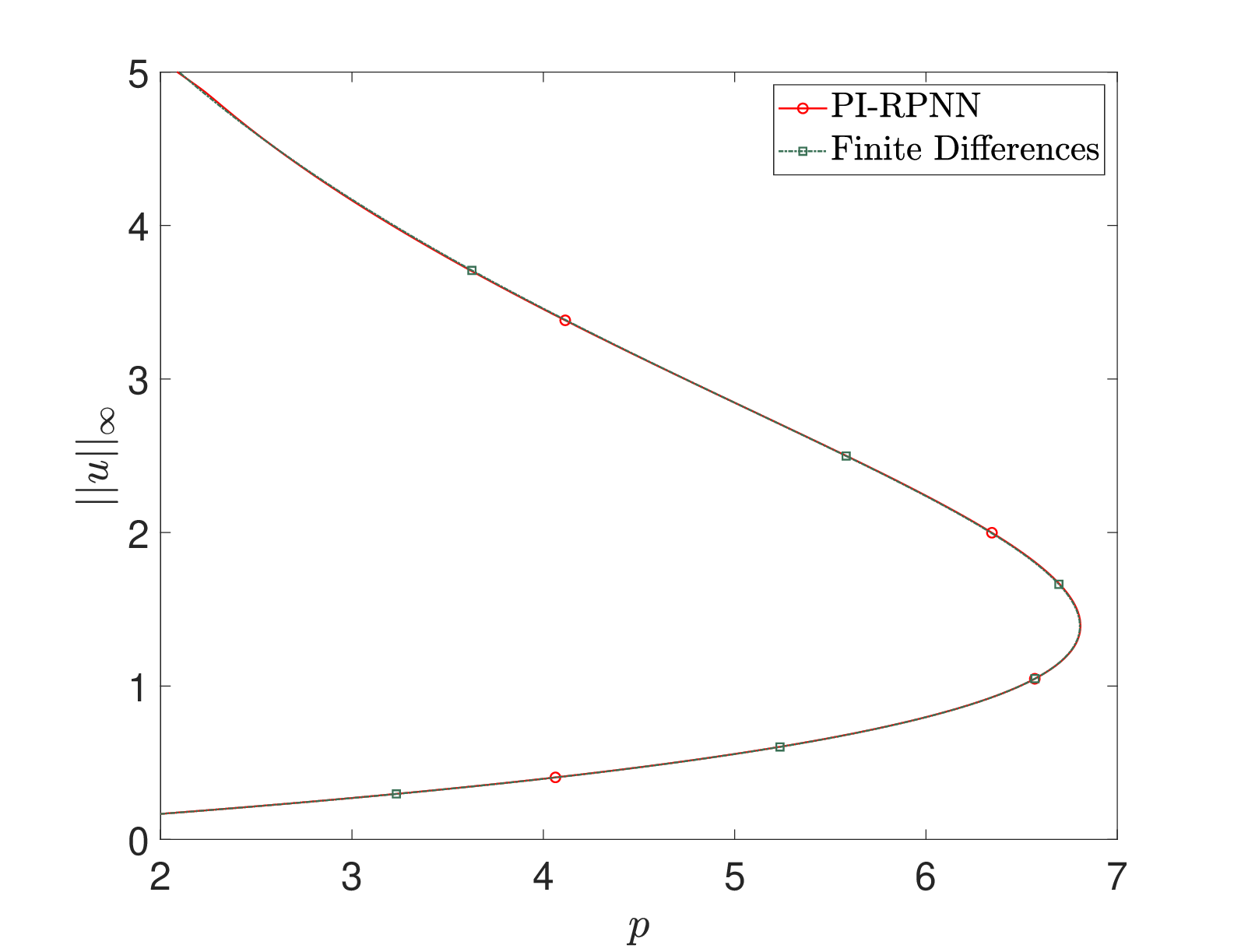}   &  \includegraphics[width=0.45\linewidth]{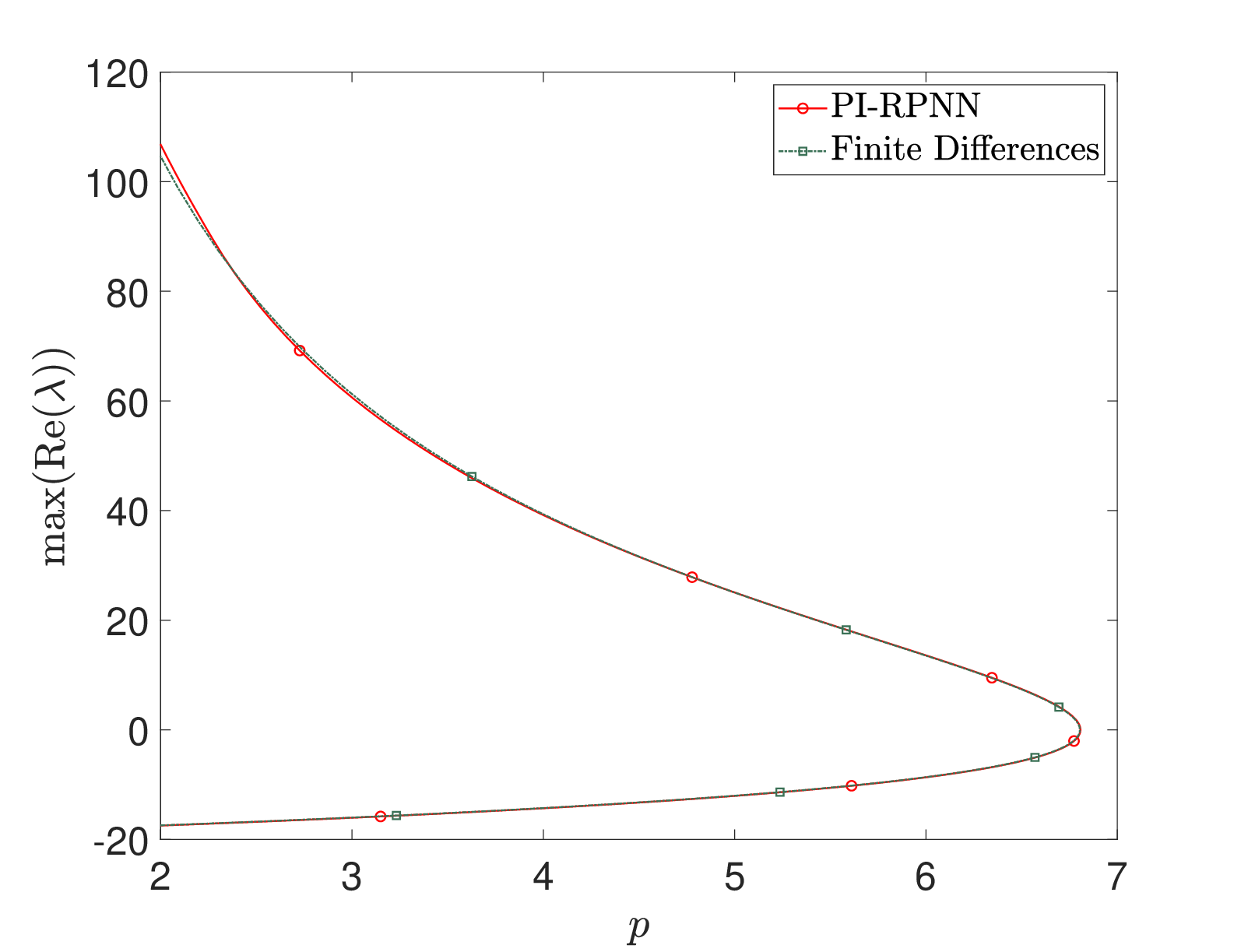} \\
        (a) & (b) \\
        \includegraphics[width=0.45\linewidth]{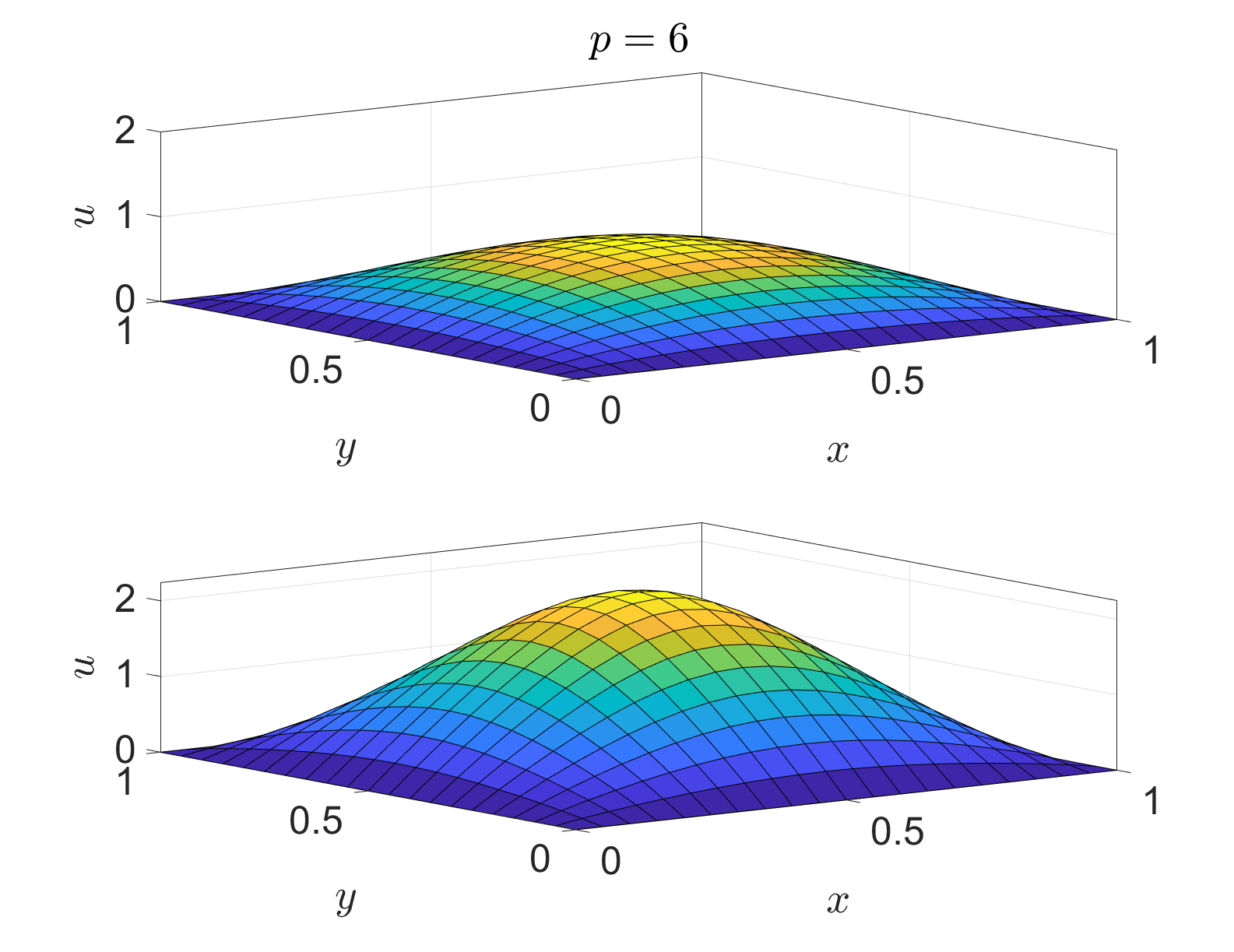}   &  \includegraphics[width=0.45\linewidth]{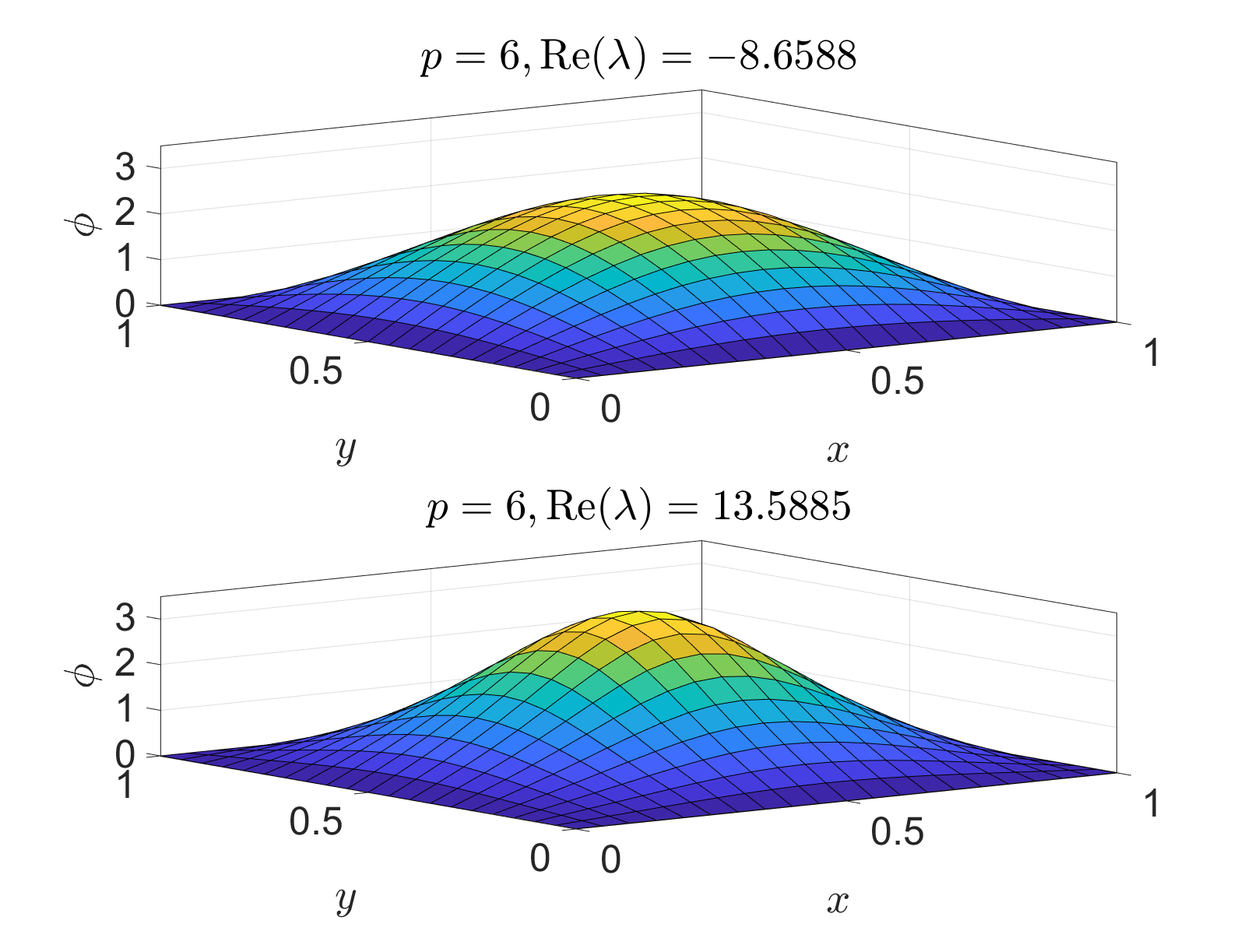} \\
        (c) & (d)
    \end{tabular}
    \caption{Numerical results for the two-dimensional Liouville-Bratu-Gelfand problem (\ref{eq:Bratu}). (a) Bifurcation diagram.
    (b) Dominant eigenvalue of the corresponding generalized eigenproblem. 
    The saddle-node, $p_c \approx  6.80736 $, marks the onset of instability for the upper solution branch with positive dominant eigenvalues.
    (c) Co-existing solution profiles at $p=6$ corresponding to the lower-stable branch and to the upper-unstable branch.
    (d) Eigenvector, $\phi$, corresponding to the dominant eigenvalue for the two co-existing solutions at $p=6$.
    }
\label{fig:bifdiagram_bratu2d}
\end{figure}

\subsubsection{One-dimensional Liouville-Bratu-Gelfand Results}
For the one-dimensional case, we use {$M=101$} equidistant collocation points and construct a PI-RPNN with {$N=50$} neurons.
Fig.~\ref{fig:bifdiagram_bratu1d}(a) shows the bifurcation diagram obtained from the analytical solution, as well as the numerical finite difference, and PI-RPNN solutions.
To quantify the stability of the computed solutions, we solve the generalized eigenvalue problem (\ref{eq:JW_weight_eig}) using a fixed shift $\sigma=1$, and MATLAB's \texttt{eigs} function.
In Fig.~\ref{fig:bifdiagram_bratu1d}(b), we display the dependence of the dominant eigenvalue on parameter $p$ and compare it with the finite-difference results.
All eigenvalues are real;
the largest eigenvalue is negative along the lower (stable) branch of the bifurcation diagram, whereas at the saddle-node $p_c$, the dominant eigenvalue changes sign, indicating a change in stability.
Fig.~\ref{fig:bifdiagram_bratu1d}(d) shows the corresponding dominant eigenvector for both the stable (lower) and unstable (upper) branch solution at $p=3$.
To validate our findings, we solve the one-dimensional (\ref{eq:Bratu}) using finite-differences and perform the corresponding stability analysis with MATLAB's \texttt{eigs} function.
Both numerical approaches produce indistinguishable results.\par

\subsubsection{Addressing Ill-Conditioning: A Comparative Analysis}
\label{sec:illcond_analysis}

To underscore the necessity of the proposed shift-invert formulation, we perform a comparative analysis on the 1D Liouville-Bratu-Gelfand problem at $p=3$. As discussed in Section~\ref{sec:PI-RPNN_stability}, the ``naive" approach -- explicitly reconstructing $\widetilde{\mathbf{J}}_u$ via $\mathbf{J}_w \mathbf{\Psi}^{\dagger}$ -- is critically hampered by the severe rank-deficiency of the collocation matrix $\mathbf{\Psi}$, and we now illustrate this concretely through a numerical comparison.

Fig.~\ref{fig:bratu_illcond}(a) displays the singular values of $\mathbf{\Psi}$ for our standard discretization ($M=101$ collocation points, $N=50$ neurons). The rapid decay confirms a condition number of $\mathcal{O}(10^{17})$, with a numerical rank significantly lower than $\min(M, N)$ when a practical pseudo-inverse tolerance (e.g., $\tau = 10^{-8}$, dashed line) is applied. This rank deficiency has direct consequences for the computed eigenvalue spectrum.
The observed empirical decay rate is approximately $O(5^{-n/2})$ in Fig.~\ref{fig:bratu_illcond}(a), in agreement with Proposition~\ref{prop:Psi_decay}. Note that the effective rate $R$ is sensitive to the sharpness (upper bound) of the sigmoid activation used, and the reported behavior is specific to the present discretization with $M=101$ and $N=50$, and choice of upper-bound as in Eq.~\eqref{eq:upper_bound}.
As shown in Fig.~\ref{fig:bratu_illcond}(c), solving the full generalized eigenproblem for the naively constructed $\widetilde{\mathbf{J}}_u$ produces, in addition to the correct leading eigenvalues, a dense cluster of spurious near-zero eigenvalues. The nature of these spurious modes is illustrated in Fig.~\ref{fig:bratu_illcond}(d), which displays ten representative eigenfunctions from this cluster for both the stable and unstable solution branches. These discarded eigenfunctions exhibit high-frequency, unstructured oscillations devoid of physical meaning.
These artifacts correspond directly to the truncated singular modes of $\mathbf{\Psi}$. Although the associated dominant eigenvectors remain accurate and visually indistinguishable from the reference solution (see Fig.~\ref{fig:bratu_illcond}(b)), the contaminated eigenvalue spectrum renders the approach unreliable for automated bifurcation detection, where distinguishing true critical eigenvalues from numerical artifacts is essential.
Additionally, a systematic sensitivity analysis with respect to the truncation tolerance 
$\tau$ is provided in Appendix~\ref{app:sensitivity}, confirming the 
robustness of the proposed method and the choice $\tau = 10^{-8}$ 
used throughout this work.

Beyond numerical stability, the shift-invert Arnoldi method offers a decisive computational advantage. The naive approach requires solving a dense, full generalized eigenvalue problem of size $M \times M$. In contrast, our matrix-free method targets only a few eigenvalues near a specified shift $\sigma$ using the Arnoldi iteration, which relies solely on matrix-vector products with the operator $\mathcal{T}_{\sigma}$ defined in Eq.~\eqref{eq:shift_invert_operator}. This avoids the $\mathcal{O}(M^3)$ cost of a full diagonalization. In our experiments for this configuration, solving the full generalized eigenproblem for the naive approach required an average of $0.0058$ seconds, whereas computing the first seven eigenvalues using our shift-invert implementation required only $0.0025$ seconds—effectively halving the computational time. Although these timings are indicative for this moderate-scale problem, the advantage of the matrix-free, targeted approach becomes increasingly significant for finer discretizations and higher-dimensional problems, where forming and factoring large, dense matrices becomes prohibitive. Thus, the proposed method not only ensures spectral fidelity by circumventing the rank-deficient collocation matrix inversion but also provides a scalable pathway for stability analysis in larger-scale settings.

\subsubsection{Two-dimensional Liouville-Bratu-Gelfand Results}
For the two-dimensional case, we discretize the domain using a grid of $21 \times 21$ equidistant collocation points. 
The PI-RPNN is constructed with $N=800$ neurons, and Fig.~\ref{fig:bifdiagram_bratu2d}(a) shows the resulting bifurcation diagram, featuring a saddle-node at $p_c \approx 6.80736$.
For validation, we compare the PI-RPNN solutions against a finite-difference reference solution, observing very good agreement. 
To assess the stability of the computed solutions, we solve the generalized eigenvalue problem (\ref{eq:JW_weight_eig}) using a shift $\sigma=1$.
Fig.~\ref{fig:bifdiagram_bratu2d}(b) displays the dominant eigenvalue along the solution branches;
negative eigenvalues identify the lower stable branch, with the stability switch occurring at $p_c \approx 6.80736$.
Figs.~(\ref{fig:bifdiagram_bratu2d})(c)-(d) present the two co-existing two-dimensional Liouville-Bratu-Gelfand solutions at $p=6$.
In panel (c), the top subfigure corresponds to the stable lower branch.
Panel (d) shows the dominant eigenfunctions computed using MATLAB's \texttt{eigs}.
The top subfigure illustrates the eigenvector $\phi$ associated with the stable dominant eigenvalue $\lambda=-8.6588$, whereas the bottom subfigure displays the dominant unstable eigenvector corresponding to $\lambda=13.5885$.

\subsection{Case study 2: FitzHugh-Nagumo (FHN)}
\label{sec:FHN}
Our second test problem is the FitzHugh--Nagumo (FHN) system. Originally introduced in Ref.~\cite{fitzhugh1961impulses} as a reduced model of the Hodgkin--Huxley equations, it describes the evolution of an activator $u$ and an inhibitor $v$ along an unmyelinated axon. 
\begin{figure}[ht!]
    \centering
    \begin{tabular}{cc}
      \includegraphics[width=0.45\linewidth]{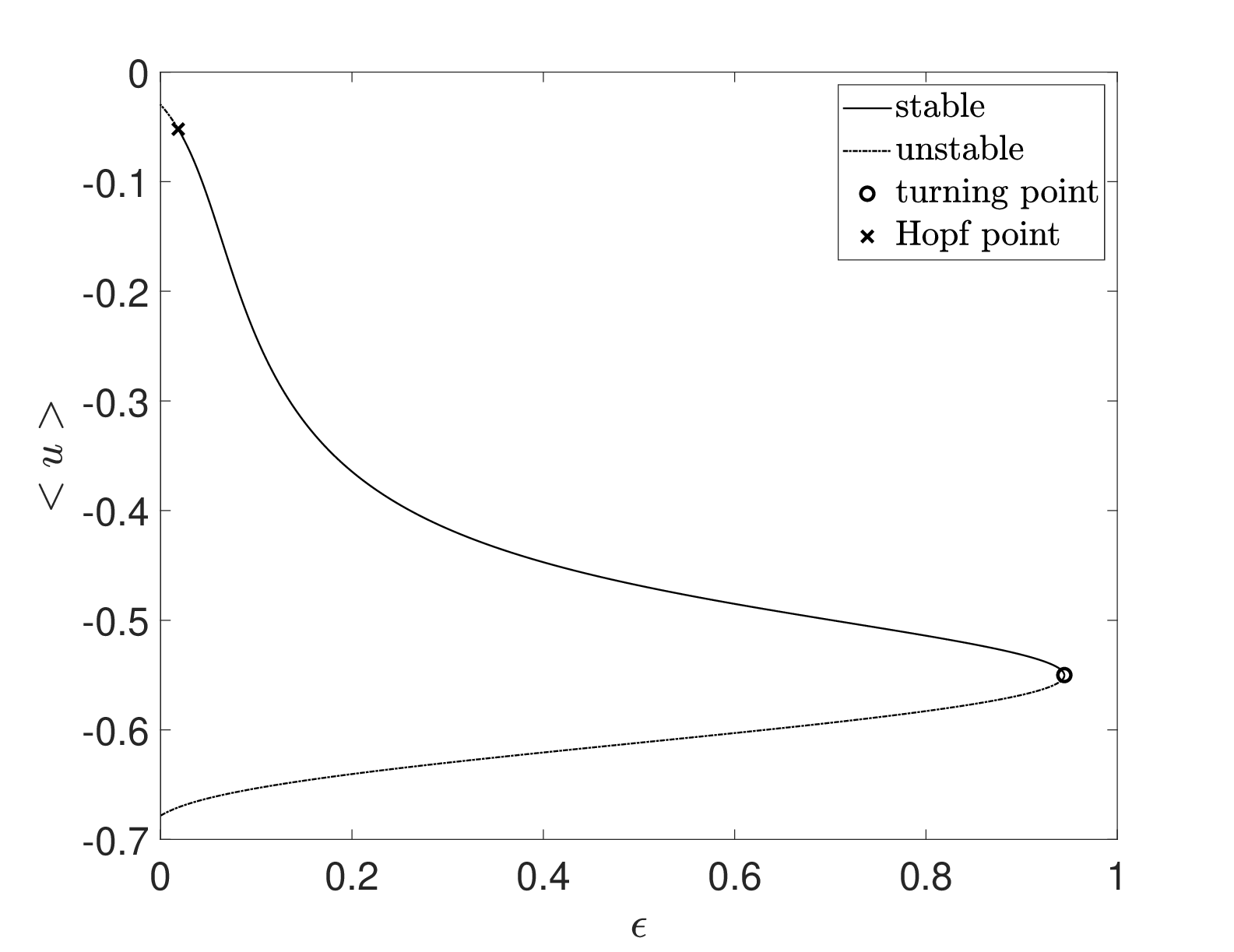}   &  \includegraphics[width=0.45\linewidth]{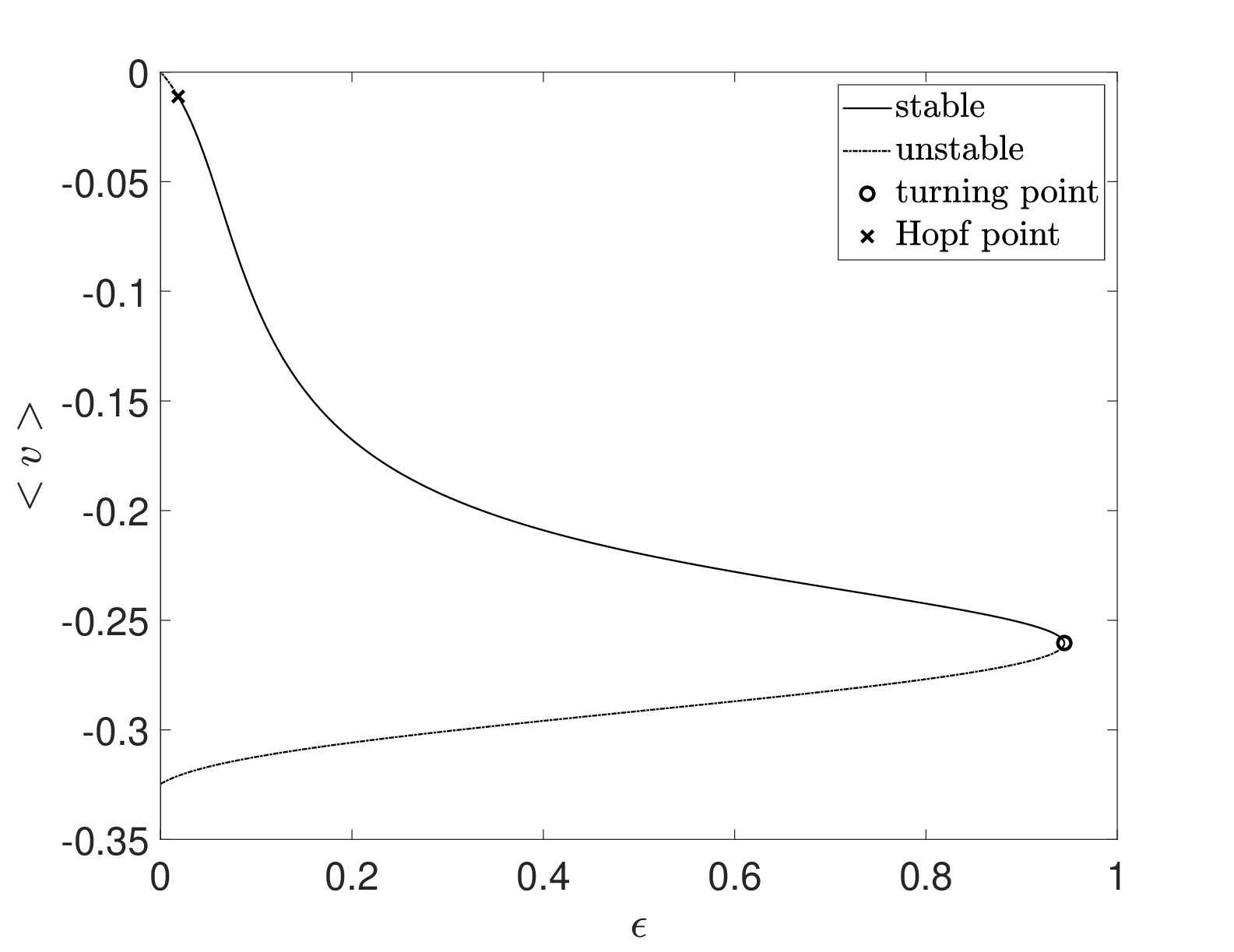} \\
        (a) & (b) \\
        \includegraphics[width=0.45\linewidth]{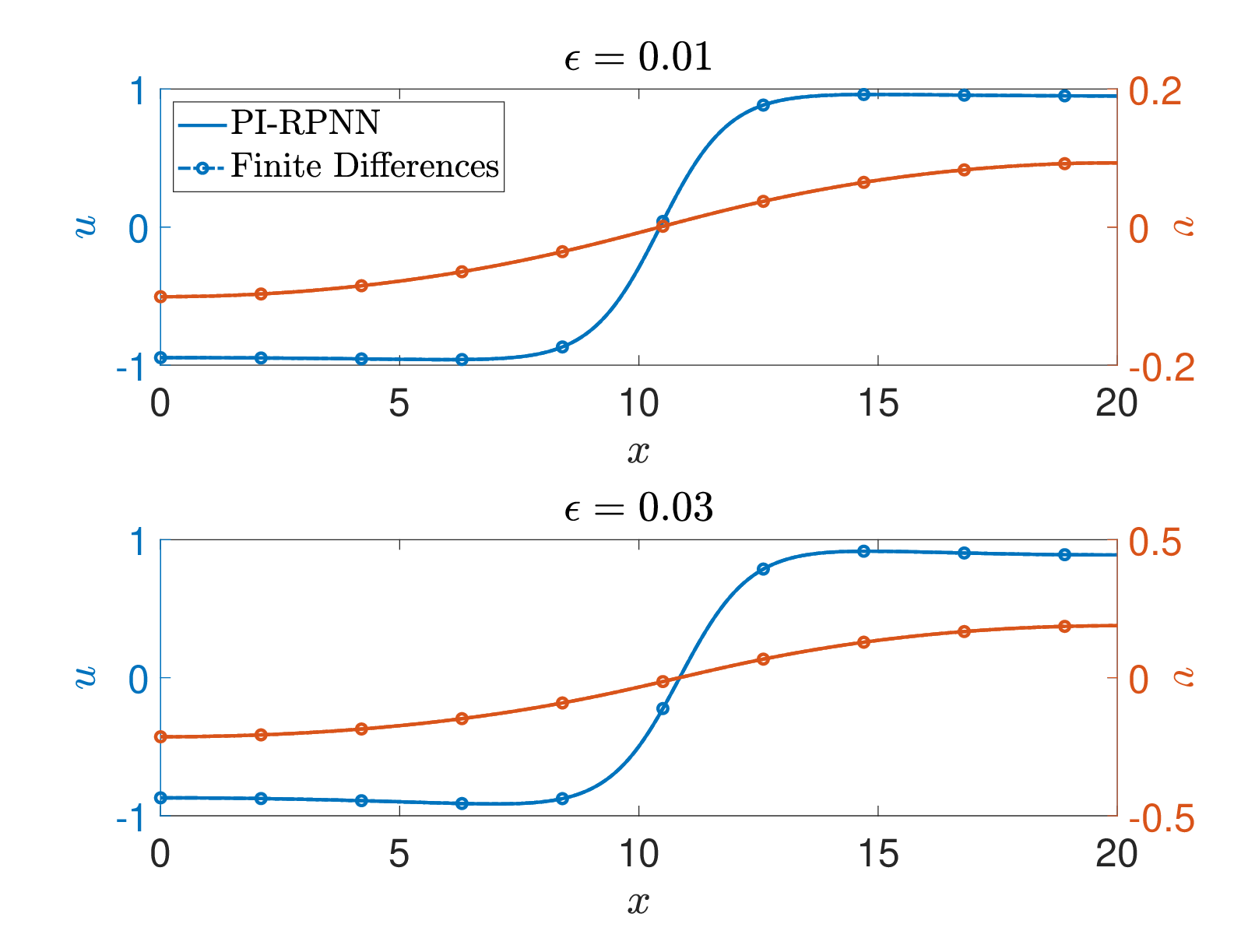}   &  \includegraphics[width=0.45\linewidth]{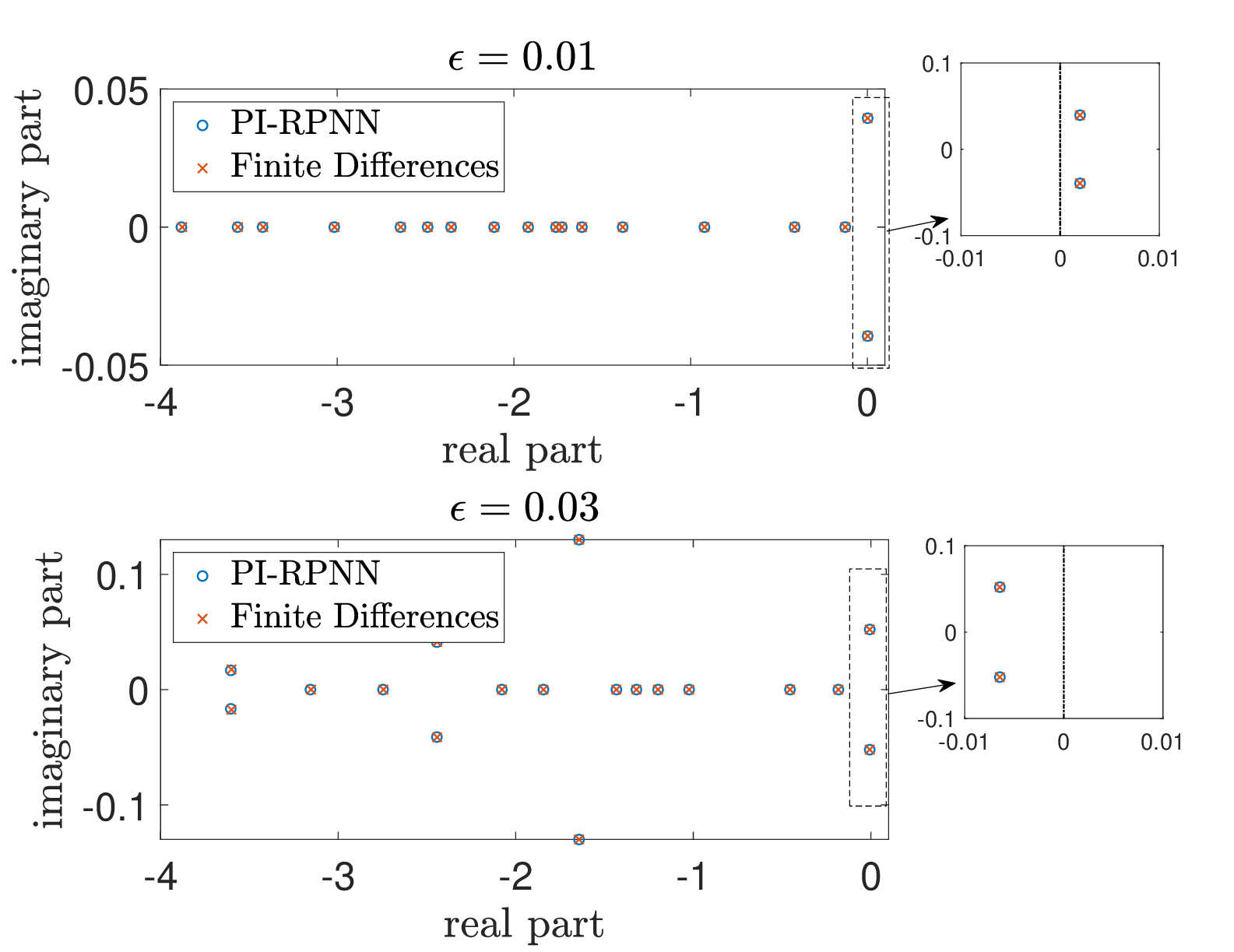} \\
        (c) & (d)
    \end{tabular}
    \caption{Numerical results for the FitzHugh-Nagumo problem (\ref{eq:FHN_PDE}). (a)-(b) Bifurcation diagrams, showing the dependence of $<u>$ and $<v>$ on the parameter $\varepsilon$.
    The open circle at $\varepsilon_c \approx 0.9446$ marks the saddle-node bifurcation at which stability switches from the upper stable branch to the lower unstable branch. 
    The black cross at $\varepsilon_H \approx 0.0184$ denotes a Hopf point.
    (c) Solution profiles of components $u$ and $v$ for $\varepsilon$ values on both sides of the Hopf point:
    the top subfigure depicts the unstable solution at $\varepsilon=0.01$, and the bottom subfigure displays the stable solution profiles at $\varepsilon=0.03$.
    The circled lines correspond to numerical solutions computed using finite differences with the same number of equidistant collocation points; 
    (d) Dominant eigenvalues at $\varepsilon=0.01$ and $0.03$ (top and bottom subfigures). 
    At $\varepsilon=0.01$, a pair of complex eigenvalues with positive real parts quantifies the instability of the steady-state solution.
    At $\varepsilon=0.03$, the corresponding complex pair has negative real parts. 
    \label{fig:bifdiagram_FHN}
    }
\end{figure}

We study the steady-state version of the one-dimensional reaction--diffusion system, obtained by setting the time derivatives to zero:
\begin{equation}
\label{eq:FHN_PDE}
\begin{aligned}
D^{u}\,\frac{\partial^{2} u(x)}{\partial x^{2}}
+ u(x) - u(x)^{3} - v(x) &=0,  \\
D^{v}\,\frac{\partial^{2} v(x)}{\partial x^{2}}
+ \varepsilon\big(u(x) - a_{1} v(x) - a_{0}\big) &= 0,
\end{aligned}
\end{equation}
posed on $\Omega=[0,L]$ with $L=20$. The constants $D^u,D^v$ and $a_0,a_1$ specify the diffusion and reaction strengths, whereas $\varepsilon$ is treated as the bifurcation parameter.  
Homogeneous Neumann BCs are imposed on both fields. The corresponding steady-state branch exhibits a turning point and a supercritical Hopf bifurcation, as reported in Refs.~\cite{galaris2022numerical, fabiani2025enabling}.\par

The use of collocation points combined with the PI--RPNN formulation yields a finite-dimensional system of nonlinear algebraic equations.  
As illustrated pedagogically for the Liouville-Bratu-Gelfand problem, the residuals arise by enforcing the PDE and BCs at all collocation points. %
For the FitzHugh--Nagumo system, the procedure is identical and omitted here for brevity. 
The only difference is that one introduces two residual blocks, associated with the fields $u$ and $v$, and uses output weights  
$\bm w \in \R^{N \times 2}$, where the first column parameterizes the approximation of the field $u$ and the second column parameterizes the approximation of the field $v$.\par

\subsubsection{FitzHugh-Nagumo results}
We solve the steady-state of (\ref{eq:FHN_PDE}) and perform a parametric analysis with respect to parameter $\varepsilon$.
In our computations we set: $D^u=1$, $D^v=4$, $a_0=-0.03$, and $a_1=2$.
The domain $\Omega = [0,L]$ is discretized with $201$ equidistant collocation points, and the PI--RPNN is constructed with $N=200$ neurons for the approximation of $u$ and $N=200$ neurons for the approximation of $v$.
\begin{figure}[ht!]
    \centering
    \begin{tabular}{cc}
      \includegraphics[width=0.45\linewidth]{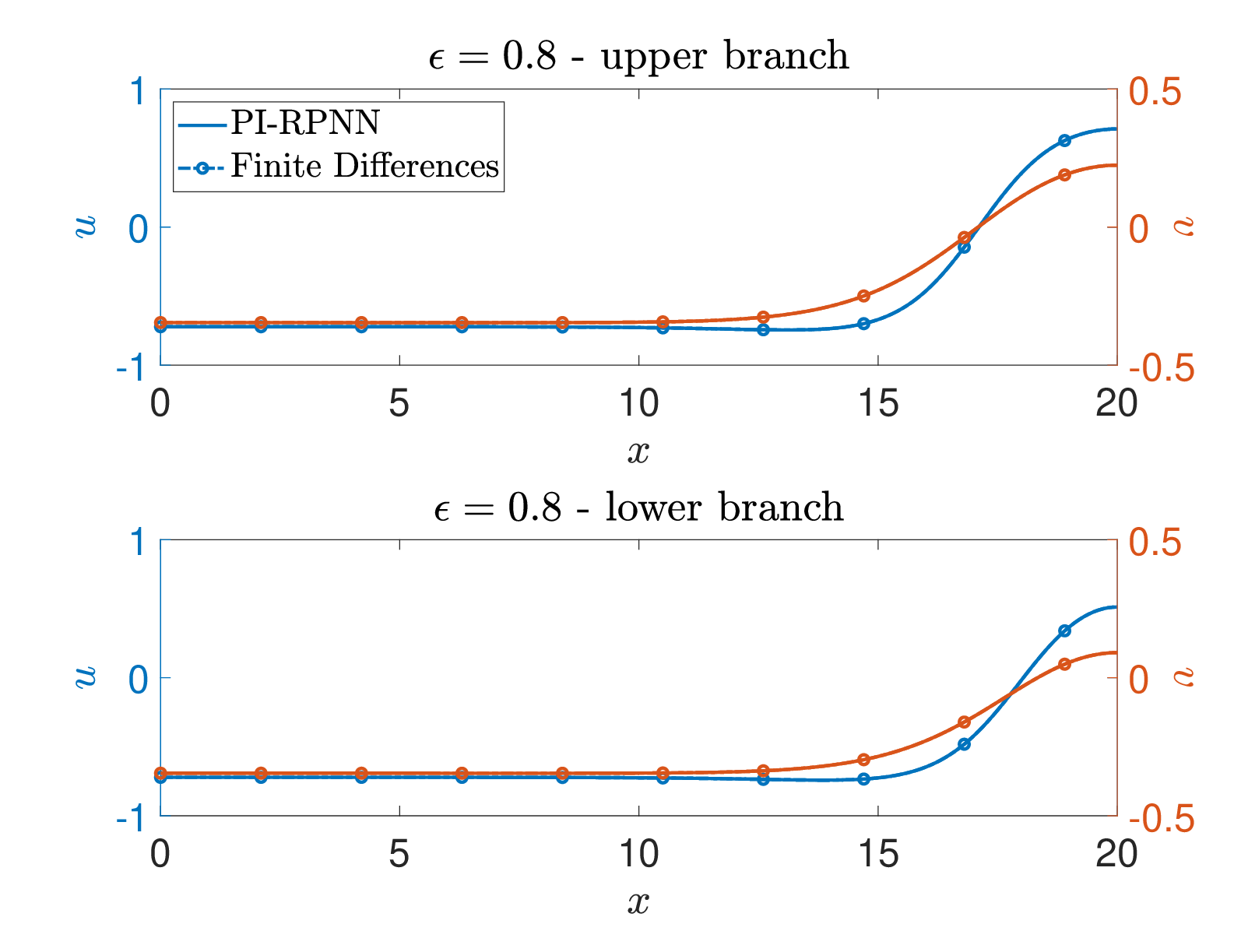}   &  \includegraphics[width=0.45\linewidth]{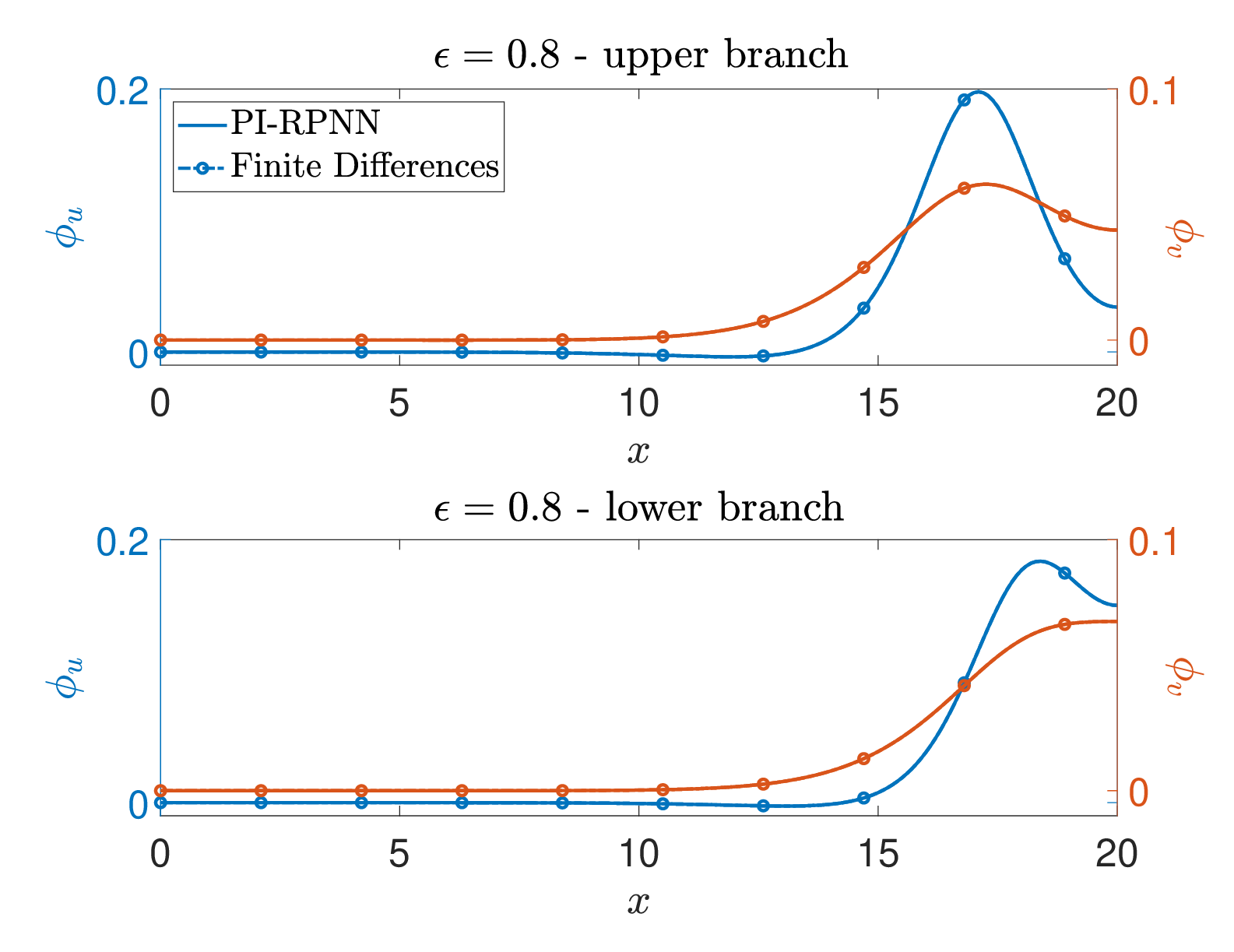} \\
        (a) & (b) 
    \end{tabular}
    \caption{Coexisting steady-state solutions of the FitzHugh-Nagumo problem (\ref{eq:FHN_PDE}) at $\varepsilon=0.8$. (a)
    The top subfigure shows the solution on the upper stable branch, with negative dominant eigenvalue ($\lambda \approx -0.0495$). 
    The bottom subfigure displays the solution on the lower unstable branch, characterized by a positive dominant eigenvalue ($\lambda \approx 0.1266$). 
    Panel (b) depicts the corresponding dominant eigenvectors for the $u$ and $v$ components (top: stable branch, bottom: unstable branch). 
    The circled curves in panels (a) and (b) correspond to the finite-difference reference solutions.
    }
\label{fig:FHN_saddle}
\end{figure}
Figures~(\ref{fig:bifdiagram_FHN})(a)-(b) depict the resulting bifurcation diagrams for $u$ and $v$, respectively, featuring a turning point at $\varepsilon_C \approx 0.9446$ (marked with a circle), and a Hopf point at  $\varepsilon_H \approx 0.0184$.
The existence of Hopf point is confirmed through our generalized eigenvalue analysis: for $\varepsilon>\varepsilon_H$, the dominant eigenvalues form a complex conjugate pair with negative real parts (Fig.~\ref{fig:bifdiagram_FHN}(d), bottom subfigure shows the eigenvalue spectrum at $\varepsilon=0.03$), indicating stability. 
When $\varepsilon<\varepsilon_H$, this pair crosses into the right half-plane and the equilibrium loses stability. 
For illustration, we show the dominant eigenvalues at $\varepsilon=0.01$ (Fig.~\ref{fig:bifdiagram_FHN}(d), top subfigure), where the leading complex pair has positive real parts.

Fig.~\ref{fig:FHN_saddle}(a) displays the coexisting solutions on the upper stable and lower unstable branches at $\varepsilon=0.8$.
The dominant eigenvalue is negative on the upper branch ($\lambda \approx -0.0495$), and positive on the lower branch ($\lambda \approx 0.1266$). 
The right panel of Fig.~\ref{fig:FHN_saddle} depicts the corresponding dominant eigenvectors computed using our generalized eigensolver (\ref{eq:gen_eigRPNN}).
Our findings are validated against the reference finite-difference solutions, illustrated as circled curves, demonstrating excellent agreement. 

\subsection{Case study 3: Allen-Cahn 1D (pitchfork)}
\label{sec:Allen_Cahn}
Our third illustrative case is the one-dimensional Allen-Cahn PDE~\cite{allen1972ground,zhao2024bifurcation}, a model widely used to describe phase transitions and interface dynamics.
\begin{figure}[ht!]
    \centering
    \begin{tabular}{cc}
      \includegraphics[width=0.45\linewidth]{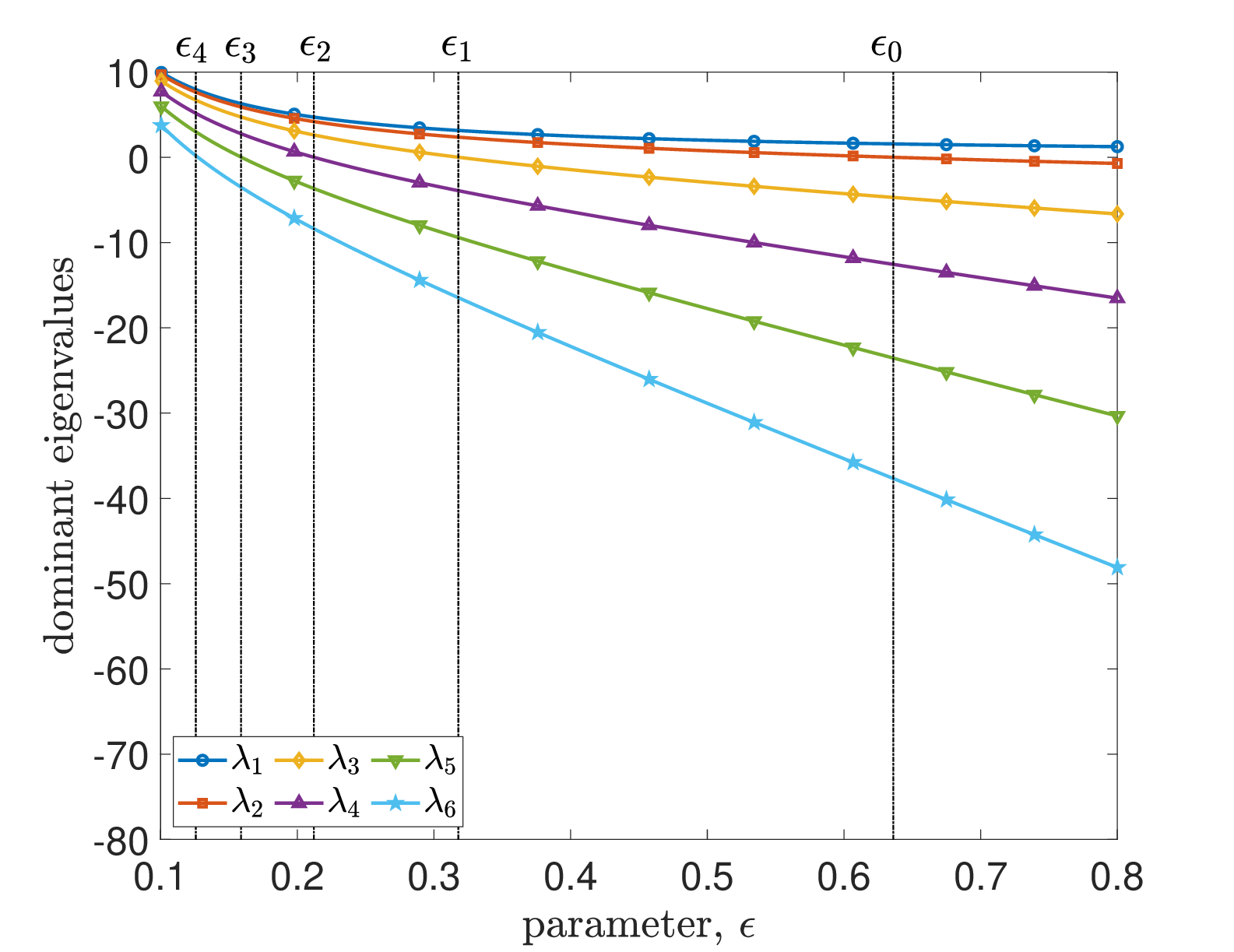}   &  \includegraphics[width=0.45\linewidth]{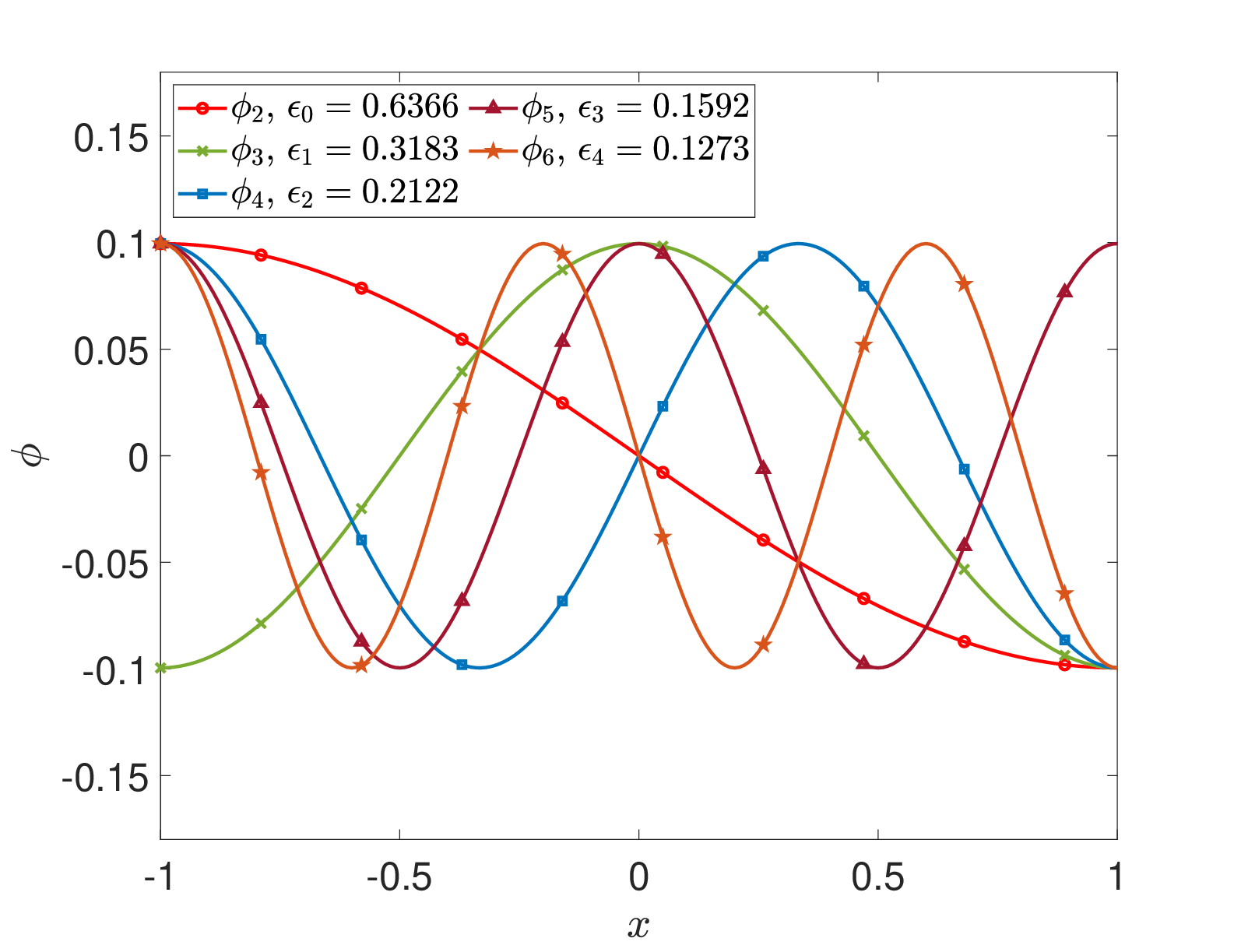} \\
        (a) & (b) \\
    \includegraphics[width=0.45\linewidth]{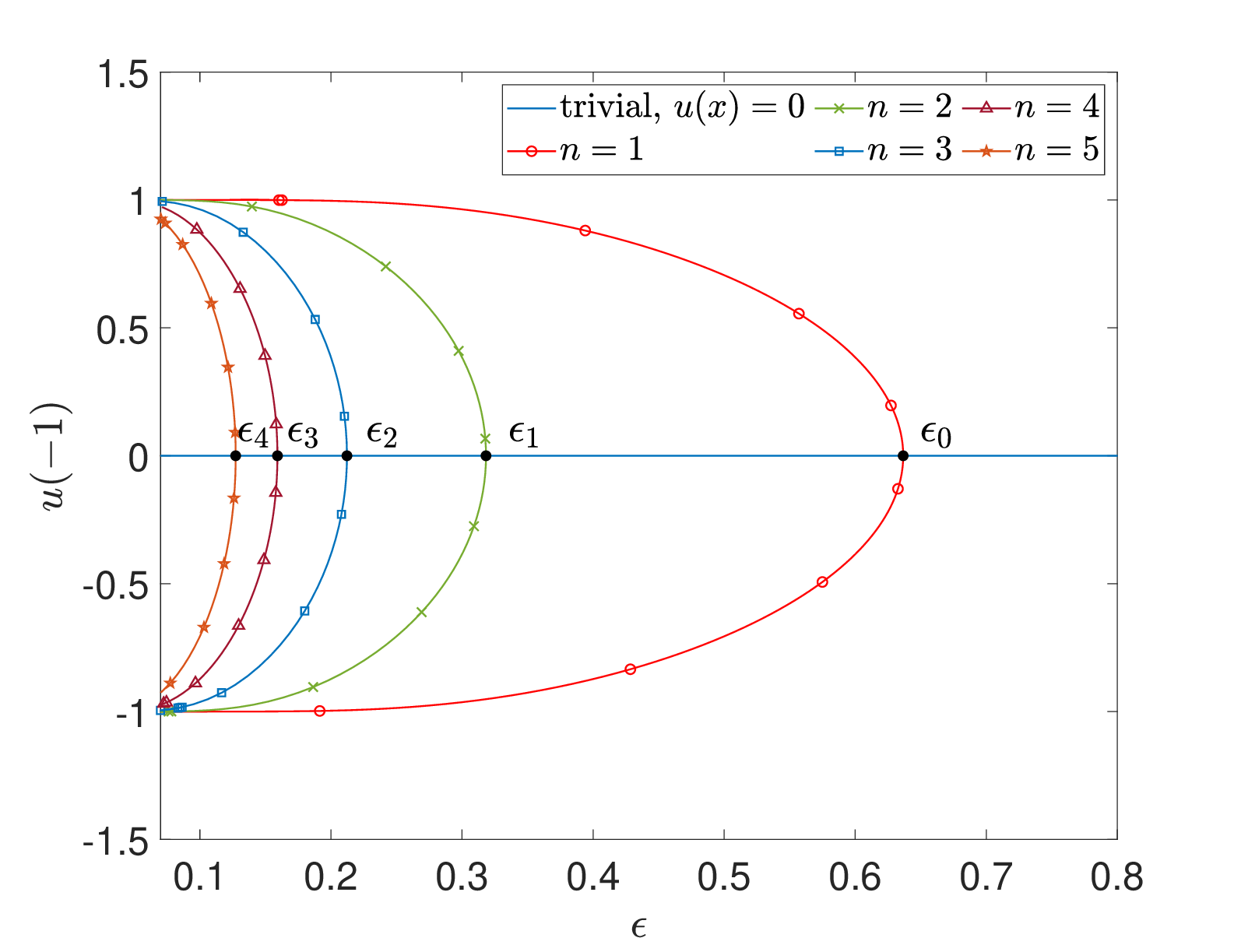}   &  \includegraphics[width=0.45\linewidth]{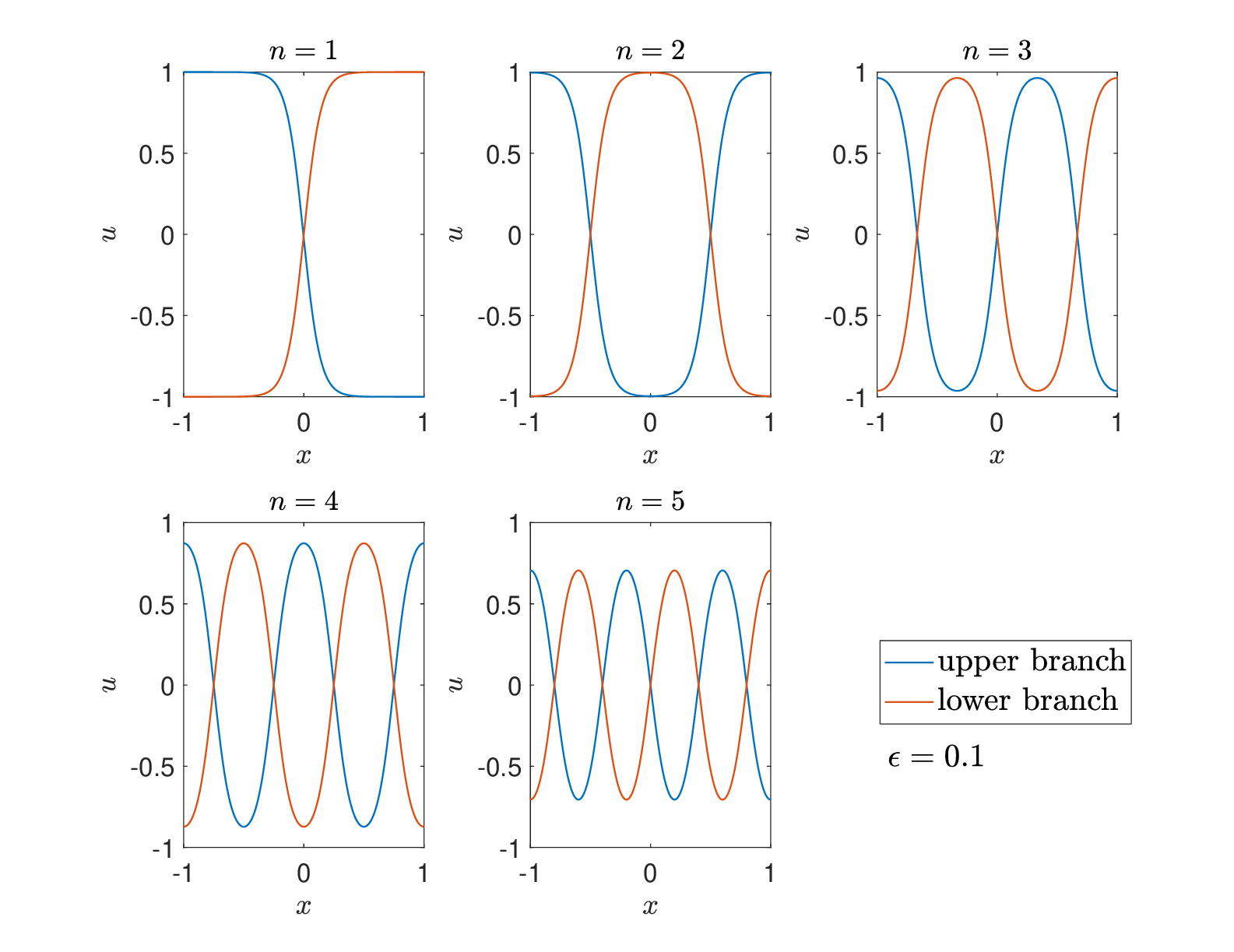} \\
        (c) & (d)    
    \end{tabular}
    \caption{Numerical results for the Allen-Cahn PDE in Eq.~\eqref{eq:ac_ss}. (a) Dependence of the six dominant eigenvalues of the Jacobian matrix along the trivial solution $u(x)=0$ on the parameter $\varepsilon$.
    The points at which $\lambda_2,\lambda_3,...,\lambda_6$ change sign indicate the bifurcation points $\varepsilon_0, \varepsilon_1,...,\varepsilon_4$.
    (b) Eigenvectors corresponding to dominant eigenvalues that change sign at bifurcation points.
    (c) Bifurcation diagram of the 1D steady-state Allen-Cahn equation, illustrating branches $n=1,2,..,5$ that emanate from the corresponding bifurcation points. 
    (d) Ten non-trivial solutions of the 1D Allen-Cahn equation for $\varepsilon=0.1$. 
    The blue and red curves correspond to solutions of the upper and lower parts, respectively, of branches $n=1,2,3,4$ and $5$ (shown in panel (c)). 
    }
\label{fig:AC_bifurcation}
\end{figure}
The steady-state Allen-Cahn equation is given by
\begin{equation}
-\varepsilon u_{xx} + \frac{1}{\varepsilon} \left(u^3 - u\right) = 0, \quad -1 < x < 1,
\label{eq:ac_ss}
\end{equation}
with Neumann boundary conditions $u_x(-1) = u_x(1) = 0$, where $0 < \varepsilon \ll 1$ is a parameter controlling the interface width. The phase field variable $u(x)$ represents two distinct phases (a value of $+1$ representing one phase, and $-1$ representing the other).

This equation arises as the Euler-Lagrange equation of the Ginzburg-Landau free energy functional
\begin{equation}
E(u) = \int_{-1}^1 \left( \frac{\varepsilon}{2} u_x^2 + \frac{1}{4\varepsilon} \left(u^2 - 1\right)^2 \right) \dd x.
\end{equation} 

The bifurcation diagram of the 1D Allen-Cahn PDE has been extensively studied in Ref.~\cite{zhao2024bifurcation}. In particular, it has been shown that the bifurcations from the trivial steady state can be explained by the Crandall-Rabinowitz theorem~\cite{crandall1971bifurcation}. For completeness of the presentation, we summarize these results here. It is straightforward to verify that $u(x) = -1, 0, 1$ are three trivial solutions of Eq.~\eqref{eq:ac_ss}. Beyond these, there are a number of bifurcation points along $u(x)=0$ at specific values of $\varepsilon$.
\begin{itemize}
    \item For each integer $n \geq 0$, the parameter values $\varepsilon^{(1)}_n = \frac{1}{\pi/2 + n\pi}$ are bifurcation points. Local solutions $ (u_n(x, s), \varepsilon_n(s)) $ around these points are well approximated by 
    \begin{equation} \label{eq:sin_bf}
    \begin{split}
    \varepsilon_n(s) = \varepsilon^{(1)}_n + s, \quad \quad |s| \leq 1, \\
    u_n(x, s) = s \sin\left(\biggl(\frac{\pi}{2} + n\pi\biggr) x\right) + O(s^2).
    \end{split}
    \end{equation}
    \item For each integer $n \geq 1$, the parameter values $\varepsilon^{(2)}_n = \frac{1}{n\pi}$ are bifurcation points. Local solutions $ (u_n(x, s), \varepsilon_n(s)) $ around these points are well approximated by
\begin{equation} \label{eq:cos_bf}
\begin{split}
\varepsilon_n(s) = \varepsilon^{(2)}_n + s, \quad |s| \leq 1\\
u_n(x, s) = s \cos(n\pi x) + O(s^2);   
\end{split}
\end{equation}
\end{itemize}
where $s$ is the arc-length parameterization.\par

Analogously to the Liouville-Bratu-Gelfand formulation, the PI--RPNN approximation for the Allen--Cahn equation converts the PDE into a nonlinear algebraic system obtained by enforcing the residual at the chosen collocation points.

\subsubsection{Allen-Cahn Results}
The one-dimensional domain $\Omega = [-1,1]$ is discretized with $201$ equidistant collocation points. 
Our PI-RPNN approximation uses $N=200$ neurons, and we solve the generalized eigenvalue problem (\ref{eq:gen_eigRPNN}) to locate the bifurcation points along $u(x)=0$.
The largest dominant eigenvalue is positive for the examined $\varepsilon$ value in $[0,0.8]$.
The second largest real eigenvalue changes sign at $\varepsilon_0 \approx 0.6366$, the third at $\varepsilon_1 \approx 0.3183$, the fourth at $\varepsilon_2 \approx 0.2122$, the fifth at $\varepsilon_3=0.1592$, and the sixth at $\varepsilon_4=0.1273$ (see Fig.~\ref{fig:AC_bifurcation}(a)).
The corresponding analytic values from (\ref{eq:sin_bf})-(\ref{eq:cos_bf}) match perfectly with the numerically obtained ones: $0.6366$, $0.3183$, $0.2122$, $0.1592$ and $0.1273$.
To switch from the trivial solution $u(x)=0$ branch, we perturb along the eigendirection corresponding to the dominant eigenvalue changing sign at each bifurcation point. 
We depict these eigenfunctions in Fig.~\ref{fig:AC_bifurcation}(b).
The resulting bifurcation diagram for the 1D Allen-Cahn equation is shown in Fig.~\ref{fig:AC_bifurcation}(c). 
Finally, Fig.~\ref{fig:AC_bifurcation}(d) shows non-trivial solutions for $\varepsilon=0.1$.

\section{Conclusion}
\label{sec:conclusion}
Building upon previous work~\cite{fabiani2021numerical}, we demonstrated that physics-informed random projection neural networks (PI-RPNNs) can be effectively employed for the stability and bifurcation analysis of nonlinear PDEs, despite the ill-conditioning of the random collocation matrix. In that work, we proposed a machine learning-aided numerical scheme based on PI-RPNNs and collocation to approximate steady-state solutions of nonlinear PDEs. The method leverages the property of PI-RPNNs as universal function approximators, where the solution subspace is spanned by randomized hidden-layer transfer functions and only the hidden-to-output weights are trained, thus avoiding the computationally expensive non-convex optimization of fully trainable neural networks. By coupling this approach with numerical continuation methods, we showed how bifurcation branches of nonlinear PDEs can be traced past turning points efficiently.

The key advance presented here is the use of a matrix-free shift-invert Krylov–Arnoldi method, which allows the computation of the leading eigenvalues of the physical Jacobian without explicitly inverting the ill-conditioned random collocation matrix $\mathbf{\Psi}$. This pseudo-inversion would otherwise propagate the numerical rank-deficiency into the physical Jacobian, introducing a dense cluster of spurious near-zero eigenvalues and contaminating the true eigenvalue spectrum (see Fig.~\ref{fig:bratu_illcond}). By solving a generalized eigenproblem for the pencil $ \mathbf{J}_w -\lambda \widetilde{\mathbf{B}}_{\mathbf{\Psi}}$ via a truncated SVD and targeted Arnoldi iteration, we obtain a clean, reliable set of leading eigenvalues and eigenvectors that correspond directly to the linearization of the physical PDE operator.
This addresses an aspect that was not considered in previous PI-RPNN approaches, enabling the reliable computation of eigenvalues required for stability and bifurcation analysis.

From a theoretical standpoint, we prove two complementary results: (i) the generalized eigenproblem in the PI-RPNN basis is almost surely regular, guaranteeing solvability with standard eigensolvers, and (ii) the collocation matrix exhibits exponentially decaying singular values for analytic activation functions, explaining the numerical rank deficiency and the cluster of near-zero eigenvalues, whereas the remaining finite eigenvalues accurately capture the physical dynamics.

We applied the method to benchmark problems, including the Liouville-Bratu-Gelfand, FitzHugh–Nagumo, and Allen–Cahn PDEs, showing that steady states, bifurcation branches, and the associated leading eigenvalues and eigenvectors can be accurately computed across parameter ranges. In particular, computing multiple top eigenvalues confirms that the method remains robust even in the presence of an ill-conditioned basis collocation matrix.  

This study intentionally focuses on establishing the \emph{numerical feasibility and stability} of the proposed eigen-analysis. Aspects such as detailed comparisons with finite-element or spectral methods, rigorous convergence studies with respect to network width, and optimization of the random sampling distribution were not our primary aim; these topics have been addressed comprehensively in earlier works on PI-RPNNs for forward problems~\cite{fabiani2021numerical,calabro2021extreme}. Nevertheless, several natural limitations and extensions warrant mention. The current implementation relies on heuristic bounds for the random parameters, which, although effective for the problems considered, may require adjustment for equations with very steep gradients or widely separated scales.
A promising direction to mitigate this sensitivity is to extend the framework to fully-trained physics-informed neural networks (PINNs). Here, one could first train a standard PINN (with multiple hidden layers) and then, after training, fix the inner layers to define a \emph{trained} feature basis. The subsequent stability analysis could then be performed in the reduced weight space of the final readout layer, constructing an analogous Jacobian in this adapted basis. Such an approach would blend the efficiency and theoretical grounding of the randomized projection with the adaptive, problem-specific features learned by gradient-based training, potentially bridging the gap between fixed randomized approaches and fully-trainable PINNs for high-fidelity stability analysis.

The computational cost is dominated by the truncated SVD of the $M \times N$ matrix $\mathbf{A}_\sigma$; for extremely large-scale 3D problems, randomized SVD techniques or iterative Jacobian-free approaches could be integrated to maintain scalability. Furthermore, the method is presently designed for steady-state bifurcations; extending it to directly compute Floquet multipliers for periodic orbits or to analyze stability of time-dependent PDEs would be a valuable direction.  

In summary, this work bridges a gap between scalable, physics-informed machine learning discretizations and rigorous numerical bifurcation analysis. By providing a stable pathway to compute the true eigenvalue spectrum of the PDE Jacobian within the PI-RPNN framework, we enable a new class of data-efficient, meshless tools for exploring parameter-dependent dynamics in nonlinear systems. The combination of random projection efficiency with robust linear algebra offers a practical and reliable approach for scientists and engineers who wish to perform stability and bifurcation analysis directly from a learned neural network representation, without resorting to traditional spatial discretizations.

\section*{Data and code availability}
The data and code supporting the findings of this study will be made available upon publication or upon reasonable request.

\normalsize
\section*{Acknowledgments}
G.F. and I.G.K. acknowledge the Department of Energy (DOE) support under Grant No. DE-SC0024162; I.G.K. also acknowledges partial support by the National Science Foundation under Grants No. CPS2223987 and FDT2436738.
C.S. acknowledges partial support from Gruppo Nazionale Calcolo Scientifico-Istituto Nazionale di Alta Matematica (GNCS-INdAM).

\section*{Author contributions statement}
\noindent
\textbf{G.F.}: Conceptualization, Methodology, Software, Validation, Formal analysis, Investigation, Writing - Original Draft, Writing - Review \& Editing, Visualization.\\
\textbf{M.E.K.}: Conceptualization, Methodology, Software, Validation, Formal analysis, Investigation, Writing - Original Draft, Writing - Review \& Editing, Visualization.\\
\textbf{C.S.}: Validation, Formal analysis, Resources, Writing - Review \& Editing, Supervision.\\
\textbf{I.G.K.}: Validation, Formal analysis, Resources, Writing - Review \& Editing, Supervision.

\section*{Competing interest statement}
The authors declare that there are no known competing financial interests or personal relationships that could have appeared to influence the work reported in this manuscript.

\vspace{2cm}
\textbf{\LARGE Appendix}
\appendix
\renewcommand{\thefigure}{S\arabic{figure}}
\setcounter{figure}{0}
\makeatletter
\renewcommand{\fnum@figure}{Supplementary Figure \thefigure}
\makeatother

\section{Proof of Proposition~\ref{prop:regular_pencil}}
\label{app:proof:prop_4}
\color{black}
To establish the regularity of the pencil $(\widetilde{\mathbf{J}}_u, \mathbf{B})$ as stated in Proposition~\ref{prop:regular_pencil}, we first recall a determinant decomposition formula based on the Schur complement.

\begin{lemma}[Schur complement determinant~\cite{zhang2006schur}]
\label{lem:schur_complement}
Let 
\begin{equation}
\mathbf{A} = \begin{pmatrix} \mathbf{A}_{11} & \mathbf{A}_{12} \\ \mathbf{A}_{21} & \mathbf{A}_{22} \end{pmatrix}
\label{eq:block_matrix}
\end{equation}
be a block matrix with $\mathbf{A}_{11}$ and $\mathbf{A}_{22}$ square. If $\mathbf{A}_{11}$ is invertible, then
\begin{equation}
\det(\mathbf{A}) = \det(\mathbf{A}_{11}) \cdot 
\det\left( \mathbf{A}_{22} - \mathbf{A}_{21} \mathbf{A}_{11}^{-1} \mathbf{A}_{12} \right).
\label{eq:schur_upper}
\end{equation}
If $\mathbf{A}_{22}$ is invertible, then
\begin{equation}
\det(\mathbf{A}) = \det(\mathbf{A}_{22}) \cdot 
\det\left( \mathbf{A}_{11} - \mathbf{A}_{12} \mathbf{A}_{22}^{-1} \mathbf{A}_{21} \right).
\label{eq:schur_lower}
\end{equation}
\end{lemma}

We now observe this simple consequence.

\begin{lemma}[A sufficient condition for pencil regularity]
\label{lem:block_pencil}
Let 
\begin{equation}
\mathbf{A} = \begin{pmatrix} \mathbf{A}_{11} & \mathbf{A}_{12} \\ \mathbf{A}_{21} & \mathbf{A}_{22} \end{pmatrix}, \qquad
\mathbf{B}_0 = \begin{pmatrix} \mathbf{I} & \mathbf{0} \\ \mathbf{0} & \mathbf{0} \end{pmatrix},
\end{equation}
with $\mathbf{A}_{11}, \mathbf{A}_{22}$ square and $\mathbf{I}$ the identity of the same size as $\mathbf{A}_{11}$. If $\mathbf{A}_{22}$ is invertible, then the pencil $(\mathbf{A}, \mathbf{B}_0)$ is regular.
\end{lemma}

\begin{proof}
Assume $\mathbf{A}_{22}$ is invertible. Applying Schur complement Lemma~\ref{lem:schur_complement} with respect to the lower-right block $\mathbf{A}_{22}$ gives
\begin{equation}
\det(\mathbf{A} - \lambda \mathbf{B}_0)
= \det
\begin{pmatrix}
\mathbf{A}_{11} - \lambda \mathbf{I} & \mathbf{A}_{12} \\
\mathbf{A}_{21} & \mathbf{A}_{22}
\end{pmatrix}
= \det(\mathbf{A}_{22}) \cdot 
\det\left( \mathbf{A}_{11} - \mathbf{A}_{12} \mathbf{A}_{22}^{-1} \mathbf{A}_{21} - \lambda \mathbf{I} \right).
\label{eq:lemma5eq1}
\end{equation}
The factor $\det(\mathbf{A}_{11} - \mathbf{A}_{12} \mathbf{A}_{22}^{-1} \mathbf{A}_{21} - \lambda \mathbf{I})=\det(\widetilde{\mathbf{A}}-\lambda \mathbf{I})$ is the characteristic polynomial $p(\lambda)$ of matrix $\widetilde{\mathbf{A}}=\mathbf{A}_{11} - \mathbf{A}_{12} \mathbf{A}_{22}^{-1} \mathbf{A}_{21}$. It has at most a finite number of roots in $\mathbb{C}$. Hence, it cannot vanish for all $\lambda \in \mathbb{C}$. Since $\det(\mathbf{A}_{22}) \neq 0$, the product in Eq.~\eqref{eq:lemma5eq1} is not identically zero. Hence the pencil $(\mathbf{A}, \mathbf{B}_0)$ is regular. \end{proof}

We now establish three auxiliary lemmas that are used in the proof of Case~(c) of Proposition~\ref{prop:regular_pencil}.

\begin{lemma}[Linear spanning and almost-sure full rank]
\label{lem:analytic_features_full_rank}
For distinct points $\bm x_1,\ldots,\bm x_M$, a real-analytic
nonpolynomial activation, and a nonempty open set
$\mathcal H\subset\mathbb R^{d+1}$, the random features vectors in Eq.~\eqref{eq:regular_feature_vectors} span $\mathbb R^M$.
If the joint neuron parameters have a Lebesgue density on a connected
open product domain $\mathcal H^N$, then
$\operatorname{rank}\mathbf\Psi=\min(M,N)$ almost surely.
The analogous assertion holds after restricting to any subset of rows.
\end{lemma}
\begin{proof}
Suppose $\bm c^\top\bm z(\bm\eta)=0$ for every
$\bm\eta\in\mathcal H$. This is an identity between real-analytic
functions of $(\bm\alpha,\beta)$, and hence extends to
$\mathbb R^{d+1}$. Choose $\bm v\in\mathbb R^d$ such that the scalars
$t_i=\bm v\cdot\bm x_i$ are distinct. Such a direction exists because
the forbidden directions form a finite union of proper hyperplanes.
Setting $\bm\alpha=a\bm v$ gives
\begin{equation}
\sum_{i=1}^M c_i\psi(at_i+b)=0
\qquad\text{for every }a,b\in\mathbb R.
\end{equation}
Differentiating $k$ times with respect to $a$ at $a=0$ yields
\begin{equation}
\psi^{(k)}(b)\sum_{i=1}^M c_i t_i^k=0.
\end{equation}
For each $k=0,\ldots,M-1$, the function $\psi^{(k)}$ is not identically
zero, since $\psi$ is nonpolynomial. Thus all these moments vanish.
The resulting Vandermonde system is nonsingular, so $\bm c=0$.
This proves the spanning assertion.

Let $r=\min(M,N)$. The spanning assertion provides $r$ feature vectors
that are linearly independent, and therefore an $r\times r$ minor
that is not identically zero as a function of the neuron parameters.
This minor is real analytic. Its zero set has Lebesgue measure zero
on the connected parameter domain~\cite{mityagin2020zero}.
Absolute continuity of the sampling law proves the rank assertion.
Applying the same argument to a subset of the distinct collocation
points proves the final assertion.
\end{proof}

\begin{lemma}
\label{lem:quadratic_span_equivalence}
Let $\mathcal Z\subset\mathbb R^M$, and define
\(
\mathcal W=\operatorname{span}\{\bm z\bm z^\top:\bm z\in\mathcal Z\}
\subseteq\operatorname{Sym}_M.
\)
The following are equivalent:
\begin{enumerate}
\item $\mathcal W=\operatorname{Sym}_M$.
\item If $\mathbf Q\in\operatorname{Sym}_M$ satisfies
$\bm z^\top\mathbf Q\bm z=0$ for every $\bm z\in\mathcal Z$,
then $\mathbf Q=\mathbf0$.
\end{enumerate}
\end{lemma}
\begin{proof}
Equip $\operatorname{Sym}_M$ with the Frobenius inner product, for
which
\begin{equation}
\langle\mathbf Q,\bm z\bm z^\top\rangle_F
=\operatorname{tr}(\mathbf Q\,\bm z\bm z^\top)
=\bm z^\top\mathbf Q\bm z .
\end{equation}
Since $\mathcal W$ is spanned by the generators $\bm z\bm z^\top$, a
matrix $\mathbf Q$ is orthogonal to $\mathcal W$ if and only if
$\bm z^\top\mathbf Q\bm z=0$ for every $\bm z\in\mathcal Z$. Hence 
\(
\mathcal W^{\perp}
=\{\mathbf Q\in\operatorname{Sym}_M:
\bm z^\top\mathbf Q\bm z=0\ \text{for every }\bm z\in\mathcal Z\}
\)
is the orthogonal complement of $\mathcal W$.
In a finite-dimensional inner-product space, a subspace is the whole space if and only if its orthogonal complement is $\{\mathbf0\}$.
This proves both implications.
\end{proof}

\begin{lemma}
\label{lem:separately_quadratic_extension}
Suppose $\mathcal Z\subset\mathbb R^M$ satisfies the equivalent
conditions of Lemma~\ref{lem:quadratic_span_equivalence}.
Let $p:(\mathbb R^M)^N\to\mathbb R$ be a polynomial that is
homogeneous of degree two in each vector argument separately.
If
\begin{equation}
p(\bm z_1,\ldots,\bm z_N)=0
\qquad\text{for every }(\bm z_1,\ldots,\bm z_N)\in\mathcal Z^N,
\end{equation}
then $p$ is identically zero on $(\mathbb R^M)^N$.
\end{lemma}
\begin{proof}
Fix $\bm z_2,\ldots,\bm z_N\in\mathcal Z$.
Since $p$ is homogeneous quadratic in its first argument, there exists
a symmetric matrix $\mathbf Q$, depending on these fixed arguments, such
that
\begin{equation}
p(\bm x_1,\bm z_2,\ldots,\bm z_N)
=\bm x_1^\top\mathbf Q\bm x_1.
\end{equation}
For example, the diagonal entries of $\mathbf Q$ are the coefficients
of $x_{1,i}^2$, and its off-diagonal entries are half the coefficients
of $x_{1,i}x_{1,j}$. By hypothesis this quadratic form vanishes on
$\mathcal Z$, so Lemma~\ref{lem:quadratic_span_equivalence} gives
$\mathbf Q=\mathbf0$. Therefore $p$ vanishes on
$\mathbb R^M\times\mathcal Z^{N-1}$.

Inductively, suppose it vanishes on
$(\mathbb R^M)^k\times\mathcal Z^{N-k}$.
Fix the first $k$ arguments arbitrarily in $\mathbb R^M$ and the last
$N-k-1$ arguments in $\mathcal Z$. Applying the same quadratic-form
argument to argument $k+1$ extends the identity to
$(\mathbb R^M)^{k+1}\times\mathcal Z^{N-k-1}$.
After $N$ steps, the identity holds on all of $(\mathbb R^M)^N$.
\end{proof}

We can now proceed with the proof of Proposition~\ref{prop:regular_pencil}, demonstrating that the pencil $(\widetilde{\mathbf{J}}_u, \mathbf{B})$ is regular almost surely.

\begin{proof}[Proof of Proposition~\ref{prop:regular_pencil}]
Fix a realization of the distinct collocation points. All probabilities in the proof are conditional on these points.
For ease of notation, we denote by $\widetilde{\mathbf{J}}_u=\mathbf{J}_w\mathbf{\Psi}^\dagger$ the PI-RPNN approximation of the physical Jacobian $\mathbf{J}_u$.

Let $\mathcal{I}_{\mathrm{bc}} \subset \{1,\dots,M\}$ denote the indices of boundary collocation points, and $\mathcal{I}_{\mathrm{int}}\subset \{1,\dots,M\}$ the interior indices.
Without loss of generality, we reorder the indices so that all interior points are listed first, followed by all boundary points. 
We therefore partition the matrices accordingly:
\begin{equation}
\widetilde{\mathbf{J}}_u = \begin{pmatrix}
\widetilde{\mathbf{J}}_{\mathrm{int},\mathrm{int}} & \widetilde{\mathbf{J}}_{\mathrm{int},\mathrm{bc}} \\
\widetilde{\mathbf{J}}_{\mathrm{bc},\mathrm{int}} & \widetilde{\mathbf{J}}_{\mathrm{bc},\mathrm{bc}}
\end{pmatrix}, \qquad
\mathbf{B} = \begin{pmatrix}
\mathbf{I} & \mathbf{0} \\
\mathbf{0} & \mathbf{0}
\end{pmatrix},
\end{equation}
where $\widetilde{\mathbf{J}}_{\mathrm{int},\mathrm{int}} \in \R^{M_{\mathrm{int}}\times M_{\mathrm{int}}}$ and $\widetilde{\mathbf{J}}_{\mathrm{bc},\mathrm{bc}} \in \R^{M_{\mathrm{bc}}\times M_{\mathrm{bc}}}$.
By Lemma~\ref{lem:block_pencil}, it suffices to show that $\widetilde{\mathbf{J}}_{\mathrm{bc},\mathrm{bc}}$ is invertible almost surely.

Assume, for the sake of contradiction, that $\widetilde{\mathbf{J}}_{\mathrm{bc},\mathrm{bc}}$ is singular. Then there exists a \emph{nonzero} vector $\bm \phi_{\mathrm{bc}} \in \mathbb{R}^{M_{\mathrm{bc}}}$ such that
\begin{equation}
\widetilde{\mathbf{J}}_{\mathrm{bc},\mathrm{bc}} \, \bm{\phi}_{\mathrm{bc}} = \bm 0.
\label{eq:cond_singular_Jbc}
\end{equation}
Using the factorization $\widetilde{\mathbf{J}}_u = \mathbf{J}_w \mathbf{\Psi}^\dagger$ and restricting to the boundary rows, Eq.~\eqref{eq:cond_singular_Jbc} becomes
\begin{equation}
\mathbf{J}_w^{\mathrm{bc}} \left(\mathbf{\Psi}^{\dagger}\right)_{\mathrm{bc}} \bm{\phi}_{\mathrm{bc}} = 0,
\label{eq:cond_proposition_regularity}
\end{equation}
where $\mathbf{J}_w^{\mathrm{bc}}\in \mathbb{R}^{M_{\mathrm{bc}}\times N}$ denotes the restriction of $\mathbf{J}_w$ to the boundary rows, and $(\mathbf{\Psi}^{\dagger})_{\mathrm{bc}}\in\mathbb{R}^{N\times M_{\mathrm{bc}}}$ denotes the restriction of $\mathbf{\Psi}^{\dagger}$ to the boundary columns.

\color{black}
We now relate the boundary rows of the weight-space Jacobian to the collocation matrix, accounting for the error between the PI-RPNN field and the idealized discrete solution. We introduce the \emph{perturbed boundary Jacobian}
\begin{equation}
\widetilde{\mathbf{D}}(\bm\theta)
:=
\mathbf{D}
+
\left( \mathbf{H}_u \cdot \bm{e}(\bm\theta) \right)_{\mathrm{bc}}
+
O(\|\bm{e}(\bm\theta)\|^2),
\label{eq:tildeD_definition}
\end{equation}
where \( \mathbf{D} \in \mathbb{R}^{M_{\mathrm{bc}} \times M} \) is the Fréchet derivative matrix of the boundary operator \( \mathcal{B} \) with respect to the field \( u \), evaluated at the idealized discrete solution \( \bm{u}^{\ast} \); \( \bm{e}(\bm\theta) = \widetilde{\bm{u}}(\bm\theta) - \bm{u}^{\ast} \) is the error between the PI-RPNN field \( \widetilde{\bm{u}} \) and \( \bm{u}^{\ast} \); \( \mathbf{H}_u \) is the second-derivative tensor of the residual evaluated at \( \bm{u}^{\ast} \); and \( (\mathbf{H}_u \cdot \bm{e}(\bm\theta))_{\mathrm{bc}} \) denotes the contraction of \( \mathbf{H}_u \) with \( \bm{e}(\bm\theta) \) along one index, restricted to the boundary rows. Eq.~\eqref{eq:tildeD_definition} is the boundary analogue of the Taylor expansion in Eq.~\eqref{eq:taylor_jacobian}, and reduces to \( \widetilde{\mathbf{D}}(\bm\theta) = \mathbf{D} \) in the idealized limit \( \bm{e} \to \bm{0} \).

Since the boundary residual is \( \mathcal{B}[\tilde{u}] - g \), its derivative with respect to the output weights is
\begin{equation}
\frac{\partial}{\partial \bm{w}} \bigl( \mathcal{B}[\tilde{u}] - g \bigr)
= \left( \frac{\partial \mathcal{B}}{\partial u} \right) \frac{\partial \tilde{u}}{\partial \bm{w}}.
\end{equation}
The PI-RPNN approximation satisfies \( \tilde{u}(\bm{x}) = \mathbf{\Psi}(\bm{x}) \bm{w} \), hence \( \partial \tilde{u} / \partial \bm{w} = \mathbf{\Psi}(\bm{x}) \). Evaluating at the collocation points and using the perturbed boundary Jacobian \( \widetilde{\mathbf{D}}(\bm\theta) \) yields the matrix identity
\begin{equation}
\mathbf{J}_w^{\mathrm{bc}} = \widetilde{\mathbf{D}}(\bm\theta) \, \mathbf{\Psi}.
\label{eq:representation_BC}
\end{equation}

Substituting Eq.~\eqref{eq:representation_BC} into Eq.~\eqref{eq:cond_proposition_regularity} gives
\begin{equation}
\widetilde{\mathbf{D}}(\bm\theta) \, \mathbf{\Psi} \, (\mathbf{\Psi}^{\dagger})_{\mathrm{bc}} \bm \phi_{\mathrm{bc}} = \bm 0.
\label{eq:thm_lin_dep}
\end{equation}

\color{black}
Let $\mathbf{\Psi} = \mathbf{U} \boldsymbol{\Sigma} \mathbf{V}^\top$ be the singular value decomposition of $\mathbf{\Psi}$, with $\mathbf{U} \in \mathbb{R}^{M \times r}$, $r = \operatorname{rank}(\mathbf{\Psi})$, and partition
\begin{equation}
\mathbf{U} = \begin{pmatrix} \mathbf{U}_{\mathrm{int}} \\ \mathbf{U}_{\mathrm{bc}} \end{pmatrix},
\end{equation}
where $\mathbf{U}_{\mathrm{int}} \in \mathbb{R}^{M_{\mathrm{int}} \times r}$ and $\mathbf{U}_{\mathrm{bc}} \in \mathbb{R}^{M_{\mathrm{bc}} \times r}$. Then
\begin{equation}
\mathbf{\Psi} (\mathbf{\Psi}^{\dagger})_{\mathrm{bc}} = \mathbf{U} \mathbf{U}_{\mathrm{bc}}^\top = \begin{pmatrix} \mathbf{U}_{\mathrm{int}} \mathbf{U}_{\mathrm{bc}}^\top \\ \mathbf{U}_{\mathrm{bc}} \mathbf{U}_{\mathrm{bc}}^\top \end{pmatrix}.
\end{equation}
Define
\begin{equation}
\mathbf{A} = \mathbf{U}_{\mathrm{int}} \mathbf{U}_{\mathrm{bc}}^\top \in \mathbb{R}^{M_{\mathrm{int}} \times M_{\mathrm{bc}}}, \qquad
\mathbf{C} = \mathbf{U}_{\mathrm{bc}} \mathbf{U}_{\mathrm{bc}}^\top \in \mathbb{R}^{M_{\mathrm{bc}} \times M_{\mathrm{bc}}}.
\end{equation}
By Assumptions 1--2, and since in all cases (a)-(c) $N>M_{\mathrm{bc}}$, the boundary collocation submatrix $\mathbf{\Psi}^{\mathrm{bc}}$ has full row rank $M_{\mathrm{bc}}$ almost surely (as proven in Lemma~\ref{lem:analytic_features_full_rank}. See also, e.g.,~\cite{ito1996nonlinearity, huang2006extreme, fabiani2025random}). Then, $\mathbf{U}_{\mathrm{bc}}$ has full row rank $M_{\mathrm{bc}}$, so $\mathbf{C}$ is symmetric positive definite and hence invertible.

\color{black}
We now observe that the perturbed boundary Jacobian \( \widetilde{\mathbf{D}}_{\mathrm{bc}}(\bm \theta) \) is invertible for all \( \bm{\theta} \in \mathcal{H} \). Indeed, by Assumption~3 $\mathbf{D}_{\mathrm{bc}}$ is invertible and we can write
\begin{equation}
\widetilde{\mathbf{D}}_{\mathrm{bc}}(\bm \theta)
=
\mathbf{D}_{\mathrm{bc}}
\left(
\mathbf{I}
+
\mathbf{D}_{\mathrm{bc}}^{-1}
\left( \mathbf{H}_u \cdot \bm{e}(\bm \theta) \right)_{\mathrm{bc}}
+
O(\|\bm{e}(\bm \theta)\|^2)
\right),
\end{equation}
the Neumann series
\begin{equation}
\widetilde{\mathbf{D}}^{-1}_{\mathrm{bc}}(\bm \theta)
=
\left(
\mathbf{I}
+
\mathbf{D}_{\mathrm{bc}}^{-1}
\left( \mathbf{H}_u \cdot \bm{e}(\bm \theta) \right)_{\mathrm{bc}}
+
O(\|\bm{e}(\bm \theta)\|^2)
\right)^{-1}
\mathbf{D}^{-1}_{\mathrm{bc}}
\end{equation}
converges whenever
\begin{equation}
\left\|
\mathbf{D}^{-1}_{\mathrm{bc}}
\left( \mathbf{H}_u \cdot \bm{e}(\bm \theta) \right)_{\mathrm{bc}}
\right\|
\le
\|\mathbf{D}^{-1}_{\mathrm{bc}}\| \, \|(\mathbf{H}_u)_{\mathrm{bc}}\| \, \|\bm{e}(\bm \theta)\|
<
1,
\label{eq:Dinv_cond}
\end{equation}
which is exactly the condition of Assumption~4.
Hence \( \widetilde{\mathbf{D}}(\bm \theta) \) is invertible on all of \( \mathcal{H} \), and its inverse satisfies
\begin{equation}
\big\| \widetilde{\mathbf{D}}(\bm \theta)^{-1} \big\|
\le
\frac{\|\mathbf{D}^{-1}\|}{1 - \|\mathbf{D}^{-1}\| \, \|\mathbf{H}_u\| \, \|\bm{e}(\bm \theta)\|}.
\end{equation}

\color{black}
We now distinguish the three cases.

\medskip
\noindent\textbf{Case (a):}
\emph{Assume first $N\ge M$.}
Almost surely, $\mathbf{\Psi}$ has full row rank $M$ as proven in Lemma~\ref{lem:analytic_features_full_rank}. Then the Moore--Penrose
pseudoinverse is $\mathbf{\Psi}^{\dagger}
=\mathbf{\Psi}^{\top}(\mathbf{\Psi}\mathbf{\Psi}^{\top})^{-1}$,
and the orthogonal projector onto $\mathrm{range}(\mathbf{\Psi})$ is
\begin{equation}
\mathbf{P}=\mathbf{\Psi}\mathbf{\Psi}^{\dagger}=\mathbf{I}_M .
\end{equation}
Therefore
\begin{equation}
\mathbf{A}=\mathbf{P}_{\mathrm{int,bc}}=\mathbf{0},
\qquad
\mathbf{C}=\mathbf{P}_{\mathrm{bc,bc}}=\mathbf{I}_{M_{\mathrm{bc}}} .
\end{equation}
Consequently,
\begin{equation}
\widetilde{\mathbf{J}}_{\mathrm{bc,bc}}
={\widetilde{\mathbf{D}}}_{\mathrm{int}}\mathbf{A}
+{\widetilde{\mathbf{D}}}_{\mathrm{bc}}\mathbf{C}
={\widetilde{\mathbf{D}}}_{\mathrm{bc}}.
\end{equation}

which is invertible by Assumption~3-4 and Eq.~\eqref{eq:Dinv_cond}. By
Lemma~\ref{lem:block_pencil}, the pencil
$(\widetilde{\mathbf{J}}_u,\mathbf{B})$ is regular.

\medskip
\noindent\textbf{Case (b):} For Dirichlet-type conditions, the boundary operator acts only on the boundary degrees of freedom, so its Fréchet derivative has the block structure
\begin{equation}
\widetilde{\mathbf{D}} = \begin{pmatrix} \mathbf{0} & {\widetilde{\mathbf{D}}}_{\mathrm{bc}} \end{pmatrix},
\end{equation}
where ${\widetilde{\mathbf{D}}}_{\mathrm{bc}} \in \mathbb{R}^{M_{\mathrm{bc}} \times M_{\mathrm{bc}}}$ is invertible by Assumption~3-4 and Eq.~\eqref{eq:Dinv_cond}. 
Then Eq.~\eqref{eq:thm_lin_dep} becomes
\begin{equation}
{\widetilde{\mathbf{D}}}_{\mathrm{bc}} \, \mathbf{C} \bm \phi_{\mathrm{bc}} = \bm 0.
\end{equation}
Since both matrices are invertible, this implies $\bm \phi_{\mathrm{bc}} = \bm 0$, contradicting the assumption that $\bm{\phi}_{\mathrm{bc}} \neq \bm 0$. Hence $\widetilde{\mathbf{J}}_{\mathrm{bc},\mathrm{bc}}$ is invertible almost surely.

\medskip\noindent\textbf{Case (c): General linear boundary conditions.}
In this case, \(M_{\mathrm{bc}} < N < M\). Since the boundary operator is linear, the second-derivative tensor of the boundary residual vanishes, \((\mathbf{H}_u)_{\mathrm{bc}} = \mathbf{0}\). Consequently, the perturbed boundary Jacobian reduces to the idealized one,
\begin{equation}
\widetilde{\mathbf{D}}(\bm{\theta}) = \mathbf{D}_{\mathrm{bc}},
\end{equation}
independently of \(\bm{\theta}\). In particular, the boundary block is fixed once the collocation points are fixed, and no assumption on the error \(\bm{e}(\bm{\theta})\) is required in this case.
Let
\begin{equation}
\mathbf E_{\mathrm{bc}}=\begin{pmatrix}\mathbf0_{M_{\mathrm{int}}\times M_{\mathrm{bc}}}\\
\mathbf I_{M_{\mathrm{bc}}}\end{pmatrix}\in\mathbb R^{M\times M_{\mathrm{bc}}}
\end{equation}
select the boundary coordinates. Then
${\mathbf D}\mathbf E_{\mathrm{bc}}={\mathbf D}_{\mathrm{bc}}$.
Set $\mathcal Z=\{\bm z(\bm\theta):\bm\theta\in\mathcal H\}$.
By assumption, in case (c), $\mathcal Z$ satisfies Lemma~\ref{lem:quadratic_span_equivalence}.

\medskip
\noindent\emph{Step 1: A polynomial representation of the boundary determinant.}
Define the Gram matrix
\begin{equation}
\mathbf G(\bm\theta)
=
\mathbf\Psi(\bm\theta)^\top\mathbf\Psi(\bm\theta)
\in\mathbb R^{N\times N}.
\end{equation}
Whenever $\det\mathbf G(\bm\theta)\ne0$, the orthogonal projector
onto the column space of $\mathbf\Psi(\bm\theta)$ is
\begin{equation}
\mathbf P(\bm\theta)
=
\mathbf\Psi(\bm\theta)\mathbf\Psi(\bm\theta)^\dagger
=
\mathbf\Psi(\bm\theta)\mathbf G(\bm\theta)^{-1}
\mathbf\Psi(\bm\theta)^\top.
\end{equation}
Consequently, the boundary block whose invertibility we seek satisfies
\begin{equation}
\widetilde{\mathbf J}_{\mathrm{bc,bc}}(\bm\theta)
=
{\mathbf D}\mathbf P(\bm\theta)\mathbf E_{\mathrm{bc}}
=
{\mathbf D}\mathbf\Psi(\bm\theta)\mathbf G(\bm\theta)^{-1}
\mathbf\Psi(\bm\theta)^\top\mathbf E_{\mathrm{bc}}.
\end{equation}

This expression is real analytic on the full-column-rank set
\(
\Theta
=
\{\bm\theta\in\mathcal H^N:
\det\mathbf G(\bm\theta)\ne0\}.
\)
However, \(\mathcal{H}\) need not be connected: the condition \(\det \mathbf{G}(\bm{\theta}) = 0\) is codimension one, and its solutions may form hypersurfaces that disconnect \(\mathcal{H}\). Applying the analytic
zero-set theorem~\cite{mityagin2020zero} directly to
$\det\widetilde{\mathbf J}_{\mathrm{bc,bc}}$ would therefore require
establishing its non-identical vanishing on each connected component.
We avoid this difficulty by constructing an equivalent determinant condition that contains no matrix inverse and extends analytically to the entire parameter domain.
To obtain this representation, we apply the Schur complement identity in reverse.
For an arbitrary
$\mathbf X\in\mathbb R^{M\times N}$,
\begin{equation}
\mathbf K(\mathbf X)
=
\begin{pmatrix}
\mathbf X^\top\mathbf X&
\mathbf X^\top\mathbf E_{\mathrm{bc}}\\
\mathbf D\mathbf X&
\mathbf0_{M_{\mathrm{bc}}\times M_{\mathrm{bc}}}
\end{pmatrix},
\qquad
\mathcal P(\mathbf X)=\det\mathbf K(\mathbf X).
\end{equation}
Every entry of $\mathbf K(\mathbf X)$ is polynomial in the entries
of $\mathbf X$. Thus $\mathcal P$ is a polynomial on the entire
matrix space, including rank-deficient matrices. For full-column-rank
$\mathbf X$, the Schur complement gives
\begin{equation}
\mathcal P(\mathbf X)
=
(-1)^{M_{\mathrm{bc}}}
\det(\mathbf X^\top\mathbf X)\,
\det\left(
\mathbf D\mathbf X(\mathbf X^\top\mathbf X)^{-1}
\mathbf X^\top\mathbf E_{\mathrm{bc}}
\right).
\label{eq:boundary_polynomial_schur}
\end{equation}
In particular, on $\Theta$,
\begin{equation}
\mathcal P(\mathbf\Psi(\bm\theta))
=
(-1)^{M_{\mathrm{bc}}}
\det\mathbf G(\bm\theta)\,
\det\widetilde{\mathbf J}_{\mathrm{bc,bc}}(\bm\theta).
\end{equation}
Since $\det\mathbf G(\bm\theta)\ne0$ on $\Theta$, the two determinants vanish there at exactly the same parameter values.
Unlike the expression involving $\mathbf G^{-1}$, however, $\bm\theta\mapsto\mathcal P(\mathbf\Psi(\bm\theta))$ is real analytic on the full connected domain $\mathcal H^N$.
It remains to establish that this analytic function is not identically zero.

\emph{Step 2: Separate quadratic homogeneity.}
For $j\in\{1,\ldots,N\}$ and $t\in\mathbb R$, let
$\mathbf S_j(t)$ be diagonal with all diagonal entries equal to one
except entry $j$, which equals $t$. The matrix
$\mathbf X\mathbf S_j(t)$ is obtained by scaling column $j$ of
$\mathbf X$ by $t$. Direct multiplication shows that
\begin{equation}
\mathbf K(\mathbf X\mathbf S_j(t))=
\begin{pmatrix}\mathbf S_j(t)&\mathbf0\\\mathbf0&\mathbf I_{M_{\mathrm{bc}}}\end{pmatrix}
\mathbf K(\mathbf X)
\begin{pmatrix}\mathbf S_j(t)&\mathbf0\\\mathbf0&\mathbf I_{M_{\mathrm{bc}}}\end{pmatrix}.
\end{equation}
Taking determinants gives
\begin{equation}
\mathcal P(\mathbf X\mathbf S_j(t))=t^2\mathcal P(\mathbf X).
\label{eq:separate_quadratic_scaling}
\end{equation}
Because $\mathcal P$ is polynomial, this identity means that it is
homogeneous quadratic in column $j$ when all other columns are fixed.
The $(j,j)$ Gram entry scales by $t^2$, while the other entries in
the corresponding row or column scale by $t$, consistently with the
two-sided multiplication above. The same argument holds for each column.

\emph{Step 3: Non-identical vanishing on the feature family.}
Suppose that
\begin{equation}
\mathcal P\bigl([\bm z(\bm\theta_1)\ \cdots\ \bm z(\bm\theta_N)]\bigr)=0
\qquad\text{for every }(\bm\theta_1,\ldots,\bm\theta_N)\in\mathcal H^N.
\label{eq:polynomial_identical_zero_contradiction}
\end{equation}
The neuron parameters range over a product domain: each feature column
can be chosen independently of the others. By Step 2 and
Lemma~\ref{lem:separately_quadratic_extension},
\eqref{eq:polynomial_identical_zero_contradiction} would imply
$\mathcal P(\mathbf X)=0$ for every $\mathbf X\in\mathbb R^{M\times N}$.

To contradict this conclusion, choose $\mathbf X_0$ with orthonormal
columns whose range contains the boundary coordinate subspace
$\operatorname{range}(\mathbf E_{\mathrm{bc}})$. This is possible because
$M_{\mathrm{bc}}\leq N\leq M$: extend the $M_{\mathrm{bc}}$ orthonormal columns of $\mathbf E_{\mathrm{bc}}$
to an orthonormal set of $N$ vectors in $\mathbb R^M$.
Then
\begin{equation}
\mathbf X_0^\top\mathbf X_0=\mathbf I_N,
\qquad
\mathbf X_0\mathbf X_0^\top\mathbf E_{\mathrm{bc}}=\mathbf E_{\mathrm{bc}}.
\end{equation}
Substitution into Eq.~\eqref{eq:boundary_polynomial_schur} gives
\begin{equation}
\mathcal P(\mathbf X_0)
=(-1)^{M_{\mathrm{bc}}}\det(\mathbf D\mathbf E_{\mathrm{bc}})
=(-1)^{M_{\mathrm{bc}}}\det\mathbf D_{\mathrm{bc}}\ne0.
\end{equation}
Thus the function
\begin{equation}
F(\bm\theta_1,\ldots,\bm\theta_N)
:=\mathcal P\bigl([\bm z(\bm\theta_1)\ \cdots\ \bm z(\bm\theta_N)]\bigr)
\label{eq:analytic_boundary_polynomial}
\end{equation}
is not identically zero. Notice that $\mathbf X_0$ need not be realizable
by the feature family: the extension lemma transfers an alleged identity
on feature columns to arbitrary columns before $\mathbf X_0$ is chosen.

The function $F$ is real analytic on the connected open domain
$\mathcal H^N$. It is not identically zero by Step 3; hence its zero set has Lebesgue measure zero~\cite{mityagin2020zero}.
The joint sampling law has a density, and therefore $F\ne0$ almost surely.
Applying Eq.~\eqref{eq:boundary_polynomial_schur} to $\mathbf\Psi$ now gives
\begin{equation}
\det\widetilde{\mathbf J}_{\mathrm{bc,bc}}
=\frac{(-1)^{M_{\mathrm{bc}}} F}{\det(\mathbf\Psi^\top\mathbf\Psi)}\ne0
\qquad\text{almost surely}.
\end{equation}
\end{proof}

\begin{lemma}[Verification for standard analytic activations]
\label{lem:standard_activation_quadratic_span}
Suppose that the collocation points
$\bm x_1,\ldots,\bm x_M\in\mathbb R^d$ are pairwise distinct.
Define
\begin{equation}
\bm z(\bm\alpha,\beta)
=
\bigl(
\psi(\bm\alpha\cdot\bm x_i+\beta)
\bigr)_{i=1}^M.
\end{equation}
Then, on any nonempty open parameter set
$\mathcal H\subset\mathbb R^{d+1}$,
\begin{equation}
\operatorname{span}
\left\{
\bm z(\bm\alpha,\beta)
\bm z(\bm\alpha,\beta)^\top:
(\bm\alpha,\beta)\in\mathcal H
\right\}
=
\operatorname{Sym}_M
\end{equation}
for each of the activations
\begin{equation}
\psi(s)=\frac1{1+e^{-s}},
\qquad
\psi(s)=\tanh s,
\qquad
\psi(s)=e^{-s^2}.
\end{equation}
The Gaussian activation is used in ridge form.
\end{lemma}

\begin{proof}
By Lemma~\ref{lem:quadratic_span_equivalence}, it suffices
to show that a symmetric matrix $\mathbf Q$ satisfying
\begin{equation}
\bm z(\bm\alpha,\beta)^\top
\mathbf Q\bm z(\bm\alpha,\beta)=0
\qquad
\text{for every }(\bm\alpha,\beta)\in\mathcal H
\label{eq:activation_quadratic_identity}
\end{equation}
must be zero.
Real analyticity extends this identity to all of
$\mathbb R^{d+1}$.

Choose $\bm v\in\mathbb R^d$ such that
$t_i=\bm v\cdot\bm x_i$ are pairwise distinct.
Such a direction exists because the excluded directions
form a finite union of proper hyperplanes.
After reordering the points, assume
$t_1<\cdots<t_M$.
Restrict the parameters to
$\bm\alpha=a\bm v$ and $\beta=b$.

\emph{Logistic sigmoid.}
Fix $a>0$ and set
\begin{equation}
r=e^b,\qquad c_i=e^{at_i}.
\end{equation}
Then
\begin{equation}
\psi(at_i+b)=\frac{rc_i}{1+rc_i},
\end{equation}
and the assumed identity becomes
\begin{equation}
\sum_{i,j=1}^M
Q_{ij}
\frac{rc_i}{1+rc_i}
\frac{rc_j}{1+rc_j}=0,
\qquad r>0.
\label{eq:logistic_rational_identity}
\end{equation}
Multiplication by the common denominator produces
a polynomial that vanishes for every $r>0$.
It is therefore identically zero, so
Eq.~\eqref{eq:logistic_rational_identity} holds
as a rational identity.

Since the $c_i$ are distinct, at $r=-1/c_i$ only
the term
$Q_{ii}(rc_i/(1+rc_i))^2$
can have a double pole.
Its coefficient must vanish, giving $Q_{ii}=0$
for every $i$.

The simple-pole coefficient at $r=-1/c_1$ then gives
\begin{equation}
\sum_{j=2}^M Q_{1j}\frac{c_j}{c_j-c_1}=0.
\end{equation}
Since $a>0$ was arbitrary,
\begin{equation}
\sum_{j=2}^M
\frac{Q_{1j}}{1-e^{-a(t_j-t_1)}}=0
\qquad\text{for every }a>0.
\end{equation}
Taking $a\to\infty$ yields
$\sum_{j=2}^M Q_{1j}=0$.
Subtracting this identity gives
\begin{equation}
\sum_{j=2}^M
\frac{Q_{1j}}{e^{a(t_j-t_1)}-1}=0
\qquad\text{for every }a>0.
\label{eq:logistic_distance_identity}
\end{equation}
The positive distances $t_j-t_1$ are strictly increasing.
Multiplying Eq.~\eqref{eq:logistic_distance_identity}
by $e^{a(t_2-t_1)}$ and taking $a\to\infty$
gives $Q_{12}=0$.
Repeating with the next remaining distance gives
$Q_{13}=0$, and then successively $Q_{1j}=0$
for every $j>1$.

The first row and column of $\mathbf Q$ therefore vanish.
Applying the same argument to the remaining indices,
with $t_2$ as the smallest point, and proceeding
inductively gives $\mathbf Q=\mathbf0$.

\emph{Hyperbolic tangent.}
For fixed $a>0$, set
\begin{equation}
r=e^{2b},\qquad c_i=e^{2at_i}.
\end{equation}
Then
\begin{equation}
\tanh(at_i+b)=\frac{rc_i-1}{rc_i+1}.
\end{equation}
As in the logistic case, the quadratic identity is
a rational identity in $r$.
The double-pole coefficients at $r=-1/c_i$
give $Q_{ii}=0$ for every $i$.
The simple-pole coefficient at $r=-1/c_1$ gives
\begin{equation}
\sum_{j=2}^M Q_{1j}\frac{c_j+c_1}{c_j-c_1}=0.
\end{equation}
Using
\begin{equation}
\frac{c_j+c_1}{c_j-c_1}
=
1+\frac{2}{e^{2a(t_j-t_1)}-1},
\end{equation}
taking $a\to\infty$, and subtracting the resulting
constant identity yields
\begin{equation}
\sum_{j=2}^M
\frac{Q_{1j}}{e^{2a(t_j-t_1)}-1}=0
\qquad\text{for every }a>0.
\end{equation}
The same successive limiting argument eliminates
the first row and column. Induction gives
$\mathbf Q=\mathbf0$.

\emph{Gaussian ridge activation.}
Define
\begin{equation}
q_{ij}=
\begin{cases}
Q_{ii},&i=j,\\
2Q_{ij},&i<j.
\end{cases}
\end{equation}
Multiplying the quadratic identity by $e^{2b^2}$ gives
\begin{equation}
\sum_{i\le j}
q_{ij}
e^{-a^2(t_i^2+t_j^2)}
e^{-2ab(t_i+t_j)}=0
\qquad\text{for every }a,b\in\mathbb R.
\end{equation}
For fixed $a\ne0$, group terms according to the value
$s=t_i+t_j$.
Linear independence of exponentials in $b$ implies
\begin{equation}
\sum_{\substack{i\le j\\t_i+t_j=s}}
q_{ij}e^{-a^2(t_i^2+t_j^2)}=0
\qquad\text{for every }a\ne0
\end{equation}
for each such sum $s$.

Within a group with fixed sum $s$, the quantities
$t_i^2+t_j^2$ are distinct.
Indeed, the sum and sum of squares determine the product
\begin{equation}
t_it_j=\frac{s^2-(t_i^2+t_j^2)}2,
\end{equation}
and hence determine the unordered pair
$\{t_i,t_j\}$ uniquely.
Because the $t_i$ are distinct, they also determine
the unordered index pair uniquely.

Setting $u=a^2>0$ and using linear independence of
exponentials in $u$ therefore gives $q_{ij}=0$
for every $i\le j$.
Thus $\mathbf Q=\mathbf0$.
\end{proof}


\section{Proof of Proposition~\ref{prop:Psi_decay}}
\label{app:proof_5}

\begin{proof}
Consider first the one-dimensional case.
After an affine change of variables, we may take $\Omega=[-1,1]$, with $\mu$ denoting the corresponding sampling probability measure. The multidimensional case is discussed at the end.

Define the empirical random-feature kernel
\begin{equation}
k_N(x,y) := \frac{1}{N}\sum_{j=1}^N \psi_j(x)\psi_j(y),
\qquad x,y\in[-1,1],
\end{equation}
which is symmetric and positive semidefinite. 
Since the internal parameters $\bm{\theta}_j$ are drawn independently, $k_N$ is a Monte Carlo approximation of the kernel
\begin{equation}
k(x,y) := \mathbb{E}_{\bm{\theta}}\!\left[\psi(x;\bm{\theta})\psi(y;\bm{\theta})\right],
\label{eq:kernel_mean}
\end{equation}
and converges almost surely to $k$ at each fixed pair
$(x,y)$ as $N\to\infty$, by the law of large numbers.

Assumptionin Eq.~\eqref{eq:uniform_holomorphic_features} guarantees that $k$ has a holomorphic extension to $U\times U$.
Symmetry and positive semidefiniteness follow directly from its definition, so $k$ is a Mercer kernel.
Therefore, $k$ admits the expansion
\begin{equation}
k(x,y)
=\sum_{\ell=1}^{\infty}
\lambda_\ell\varphi_\ell(x)\varphi_\ell(y),
\end{equation}
where $(\lambda_\ell,\varphi_\ell)$ are the positive eigenpairs of the associated integral operator, with the eigenfunctions orthonormal in $L^2(\mu)$.

The associated integral operator
\begin{equation}
(\mathcal T_k f)(x)
=\int_{-1}^1 k(x,y)f(y)\,\mathrm d\mu(y)
\end{equation}
is compact, positive, and trace class on $L^2(\mu)$.
Choose a Bernstein ellipse $E_{R_0}\subset\mathbb C$,
with foci at $\pm1$ and parameter $R_0>1$, whose
closure lies in $U$.
The analytic-kernel approximation argument of
Little and Reade~\cite{little1984eigenvalues} yields
\begin{equation}
\lambda_\ell(\mathcal T_k)
\le C R_0^{-\ell}.
\label{eq:little_bound}
\end{equation}
For brevity, write
$\lambda_\ell=\lambda_\ell(\mathcal T_k)$.
The eigenvalue tail consequently satisfies
\begin{equation}
\Lambda_{\ge i}
:=\sum_{\ell=i}^{\infty}\lambda_\ell
\le
C\sum_{\ell=i}^{\infty}R_0^{-\ell}
=
\frac{C}{1-R_0^{-1}}R_0^{-i}
=:C_{\mathrm t}R_0^{-i}.
\label{eq:tail_sum_bound}
\end{equation}
We now proceed with the two cases.

\medskip
\noindent
\textit{(i) Asymptotic case ($N\to\infty$, $M$ fixed).}

Let $\mathbf K_M\in\R^{M\times M}$ denote the normalized kernel matrix
\begin{equation}
(\mathbf K_M)_{q\ell}
=\frac1M k(x_q,x_\ell).
\label{eq:normalized_kernel_matrix}
\end{equation}
For a truncation index $r\ge0$, define
\begin{equation}
k^{[r]}(x,y)
:=\sum_{\ell=1}^r
\lambda_\ell\varphi_\ell(x)\varphi_\ell(y),
\qquad
e_r(x,y):=k(x,y)-k^{[r]}(x,y),
\end{equation}
and let $\mathbf K_M^{[r]}$ be the corresponding normalized truncated kernel matrix.
Both $k^{[r]}$ and $e_r$ are positive semidefinite kernels.

We choose $r=i-1$, with the truncation for $r=0$ interpreted as zero. Since $\mathbf K_M^{[i-1]}$ has rank at most $i-1$, its $i$-th eigenvalue vanishes.
Weyl's inequality therefore gives
\begin{equation}
0\le\hat\lambda_i(\mathbf K_M) \le
\|\mathbf K_M-\mathbf K_M^{[i-1]}\|_2 \le
\frac1M\sum_{q=1}^M e_{i-1}(x_q,x_q),
\label{eq:kernel_tail_bound}
\end{equation}
where the last inequality follows by bounding the operator norm of the positive semidefinite remainder matrix by its trace.
This is the truncation argument underlying the bounds of Braun~\cite{braun2006accurate}, specialized to $r=i-1$.
For this choice, no estimate of the relative-error term is needed. In particular, we retain the full tail $\Lambda_{\ge i}$, without absorbing $\lambda_i$ into $\Lambda_{>i}$.

We use the following consequence of Chebyshev's inequality. For a square-integrable random element $Y$ of a real Hilbert space and $0<\eta<1$, with probability at least $1-\eta$,
\begin{equation}
\|Y\|
\le
\|\mathbb EY\|
+\eta^{-1/2}
\left(\mathbb E\|Y-\mathbb EY\|^2\right)^{1/2}
\le
(1+\eta^{-1/2})
\left(\mathbb E\|Y\|^2\right)^{1/2}.
\label{eq:second_moment_chebyshev}
\end{equation}
The second inequality follows from Cauchy--Schwarz and the variance identity.

Apply Eq.~\eqref{eq:second_moment_chebyshev} to the random vector with components
$\sqrt{\lambda_\ell}\varphi_\ell(x_q)$,
$q=1,\ldots,M$, $\ell\ge i$.
Its squared norm is $\sum_{q=1}^M e_{i-1}(x_q,x_q)$, and orthonormality gives
\begin{equation}
\mathbb E\left[
\sum_{q=1}^M e_{i-1}(x_q,x_q)
\right]
=M\Lambda_{\ge i}.
\end{equation}
Consequently, for each fixed $i$, with probability at least $1-\eta$,
\begin{equation}
\hat\lambda_i(\mathbf K_M)
\le
(1+\eta^{-1/2})^2\Lambda_{\ge i}
\le
(1+\eta^{-1/2})^2 C_{\mathrm t}R_0^{-i}.
\label{eq:eigenvalue_estimate}
\end{equation}

Let us now consider the normalized empirical Gram matrix
\begin{equation}
\mathbf G_{M,N}
:=\frac1{MN}\mathbf\Psi\mathbf\Psi^\top
\in\R^{M\times M}.
\label{eq:normalized_empirical_Gram}
\end{equation}
For fixed $M$, the law of large numbers applied to the finitely many matrix entries gives
\begin{equation}
\|\mathbf G_{M,N}-\mathbf K_M\|_2\longrightarrow0
\qquad\text{almost surely as }N\to\infty.
\end{equation}
By Weyl's inequality,
\begin{equation}
\left|
\hat\lambda_i(\mathbf G_{M,N})
-\hat\lambda_i(\mathbf K_M)
\right|
\le
\|\mathbf G_{M,N}-\mathbf K_M\|_2
\longrightarrow0.
\end{equation}
For fixed $M$, the limiting spectrum is therefore that of $\mathbf K_M$, and its decay is controlled by Eq.~\eqref{eq:eigenvalue_estimate}. In particular, with probability at least $1-\eta$ over the collocation points,
\begin{equation}
\lim_{N\to\infty}\hat\lambda_i(\mathbf G_{M,N})
=
\hat\lambda_i(\mathbf K_M)
\le
(1+\eta^{-1/2})^2 C_{\mathrm t}R_0^{-i}.
\label{eq:asymptotic_eigenvalue_bound}
\end{equation}

\medskip
\noindent
\textit{(ii) Finite-sample bound ($M,N$ finite).}

We now prove a non-asymptotic bound with high probability.
Since $|\psi_j|\le1$, each random feature belongs to $L^2(\mu)$ and admits, almost surely, the expansion
\begin{equation}
\psi_j
=\sum_{\ell=1}^{\infty}
a_{\ell j}\varphi_\ell,
\qquad
a_{\ell j}
:=\langle\psi_j,\varphi_\ell\rangle_{L^2(\mu)},
\label{eq:proof_expansion}
\end{equation}
with convergence in $L^2(\mu)$.
The coefficients satisfy
\begin{equation}
\mathbb E[a_{\ell j}a_{m j}]
=\lambda_\ell\delta_{\ell m}.
\label{eq:proof_coeff_mean}
\end{equation}
Indeed, by Fubini's theorem and orthonormality,
\begin{equation}
\mathbb E[a_{\ell j}a_{m j}]=
\int_\Omega\int_\Omega
k(x,y)\varphi_\ell(x)\varphi_m(y)
\,\mathrm d\mu(x)\,\mathrm d\mu(y)=
\lambda_m
\langle\varphi_\ell,\varphi_m\rangle_{L^2(\mu)}
=\lambda_\ell\delta_{\ell m}.
\end{equation}

For a truncation index $r$ with $0\le r\le\min(M,N)$, define the truncated eigenfunction evaluation matrix $\mathbf\Phi_r\in\R^{M\times r}$,
\begin{equation}
[\mathbf\Phi_r]_{q\ell}
=\varphi_\ell(x_q),
\qquad q=1,\ldots,M,\quad \ell=1,\ldots,r,
\label{eq:proof_Phi_r}
\end{equation}
the truncated coefficient matrix $\mathbf A_r\in\R^{r\times N}$,
\begin{equation}
[\mathbf A_r]_{\ell j}=a_{\ell j},
\label{eq:proof_A_r}
\end{equation}
and the truncated eigenvalue matrix
\begin{equation}
\mathbf\Lambda_r
=\operatorname{diag}(\lambda_1,\ldots,\lambda_r).
\label{eq:proof_Lambda_r}
\end{equation}
For finite-rank kernels, modes beyond the positive spectrum contribute zero.

Define the truncated empirical Gram matrix
\begin{equation}
\mathbf G_{M,N}^{[r]}
:=
\frac1{MN}
\mathbf\Phi_r\mathbf A_r\mathbf A_r^\top
\mathbf\Phi_r^\top.
\label{eq:proof_G_r}
\end{equation}
We compare this with the normalized truncated kernel matrix
\begin{equation}
\mathbf K_M^{[r]}
:=
\frac1M
\mathbf\Phi_r\mathbf\Lambda_r\mathbf\Phi_r^\top.
\label{eq:proof_K_r}
\end{equation}

For any $r$, the triangle inequality gives
\begin{equation}
\left|
\hat\lambda_i(\mathbf G_{M,N})-\lambda_i(\mathcal T_k)
\right|
\le{}
\left|
\hat\lambda_i(\mathbf G_{M,N})
-\hat\lambda_i(\mathbf G_{M,N}^{[r]})
\right|+
\left|
\hat\lambda_i(\mathbf G_{M,N}^{[r]})
-\hat\lambda_i(\mathbf K_M^{[r]})
\right|+
\left|
\hat\lambda_i(\mathbf K_M^{[r]})
-\lambda_i(\mathcal T_k)
\right|.
\end{equation}
We choose $r=i-1$. Since both truncated matrices have rank at most $i-1$, their $i$-th eigenvalues vanish. Consequently, the middle term vanishes and
\begin{equation}
\left|
\hat\lambda_i(\mathbf G_{M,N})-\lambda_i(\mathcal T_k)
\right|
\le
\|\mathbf G_{M,N}-\mathbf G_{M,N}^{[i-1]}\|_2
+\lambda_i(\mathcal T_k),
\label{eq:main_bound}
\end{equation}
where we have used Weyl's inequality for the first term. For an upper bound on the empirical eigenvalue itself, the same argument gives
\begin{equation}
0\le\hat\lambda_i(\mathbf G_{M,N})
\le
\|\mathbf G_{M,N}-\mathbf G_{M,N}^{[i-1]}\|_2.
\label{eq:empirical_eigenvalue_tail}
\end{equation}

It remains to bound the truncation error. Define the projection of $\psi_j$ onto the first $r$ eigenfunctions:
\begin{equation}
\psi_j^{[r]}(x)
:=\sum_{\ell=1}^r a_{\ell j}\varphi_\ell(x).
\label{eq:proof_psi_r}
\end{equation}
Let $\bm\psi_j$ and $\bm\psi_j^{[r]}$ denote the vectors of their values at the collocation points.

Expanding the difference, the terms involving only the truncated feature cancel. The remaining terms are
\begin{equation}
\begin{aligned}
\bm\psi_j\bm\psi_j^\top
-\bm\psi_j^{[r]}(\bm\psi_j^{[r]})^\top
={}&
(\bm\psi_j-\bm\psi_j^{[r]})
(\bm\psi_j-\bm\psi_j^{[r]})^\top+\\
&+
\bm\psi_j^{[r]}
(\bm\psi_j-\bm\psi_j^{[r]})^\top+
(\bm\psi_j-\bm\psi_j^{[r]})
(\bm\psi_j^{[r]})^\top.
\end{aligned}
\label{eq:proof_diff_correct}
\end{equation}

Averaging over $j=1,\ldots,N$, with the normalization in Eq.~\eqref{eq:normalized_empirical_Gram}, gives
\begin{equation}
\begin{aligned}
\mathbf G_{M,N}-\mathbf G_{M,N}^{[r]}
=\frac1{MN}\sum_{j=1}^N
\Big[
&(\bm\psi_j-\bm\psi_j^{[r]})
 (\bm\psi_j-\bm\psi_j^{[r]})^\top+\\
&+\bm\psi_j^{[r]}
 (\bm\psi_j-\bm\psi_j^{[r]})^\top+(\bm\psi_j-\bm\psi_j^{[r]})
 (\bm\psi_j^{[r]})^\top
\Big].
\end{aligned}
\label{eq:proof_diff_avg}
\end{equation}
Taking the operator norm and using $\|\bm a\bm b^\top\|_2=\|\bm a\|_2\|\bm b\|_2$, we obtain
\begin{equation}
\|\mathbf G_{M,N}-\mathbf G_{M,N}^{[r]}\|_2
\le\frac1{MN}\sum_{j=1}^N
\Big[ \|\bm\psi_j-\bm\psi_j^{[r]}\|_2^2+2\|\bm\psi_j^{[r]}\|_2
 \|\bm\psi_j-\bm\psi_j^{[r]}\|_2
\Big].
\label{eq:proof_norm_diff}
\end{equation}

For the tail-tail term, orthogonality and Eq.~\eqref{eq:proof_coeff_mean} give
\begin{equation}
\mathbb E_{\bm\theta}
\left[
\|\psi_j-\psi_j^{[r]}\|_{L^2(\mu)}^2
\right]
=
\mathbb E_{\bm\theta}
\left[\sum_{\ell=r+1}^{\infty}a_{\ell j}^2\right]
=
\sum_{\ell=r+1}^{\infty}\lambda_\ell
=\Lambda_{>r}.
\end{equation}
Since the collocation points are sampled from $\mu$ independently of the random features,
\begin{equation}
\mathbb E\left[
\frac1N\sum_{j=1}^N
\|\bm\psi_j-\bm\psi_j^{[r]}\|_2^2
\right]
=M\Lambda_{>r}.
\end{equation}
Apply Eq.~\eqref{eq:second_moment_chebyshev} to the residual matrix $\mathbf\Psi-\mathbf\Phi_r\mathbf A_r$, equipped with the Frobenius norm.
For each fixed $r$, with probability at least $1-\eta$,
\begin{equation}
\frac1N\sum_{j=1}^N
\|\bm\psi_j-\bm\psi_j^{[r]}\|_2^2
\le
M(1+\eta^{-1/2})^2\Lambda_{>r}.
\label{eq:proof_tailtail_bound}
\end{equation}

For the cross terms, the triangle inequality gives
\begin{equation}
\|\bm\psi_j^{[r]}\|_2
\le
\|\bm\psi_j\|_2
+\|\bm\psi_j-\bm\psi_j^{[r]}\|_2.
\end{equation}
Since $|\psi_j|\le1$, we have $\|\bm\psi_j\|_2\le\sqrt M$, and therefore
\begin{equation}
\begin{aligned}
\|\bm\psi_j^{[r]}\|_2
\|\bm\psi_j-\bm\psi_j^{[r]}\|_2
\le{}&
\sqrt M\,
\|\bm\psi_j-\bm\psi_j^{[r]}\|_2+\|\bm\psi_j-\bm\psi_j^{[r]}\|_2^2.
\end{aligned}
\end{equation}
By Cauchy--Schwarz,
\begin{equation}
\frac1N\sum_{j=1}^N
\|\bm\psi_j-\bm\psi_j^{[r]}\|_2
\le
\left(
\frac1N\sum_{j=1}^N
\|\bm\psi_j-\bm\psi_j^{[r]}\|_2^2
\right)^{1/2}.
\end{equation}
Using Eq.~\eqref{eq:proof_tailtail_bound}, we obtain, on the same event,
\begin{equation}
\begin{aligned}
\frac1N\sum_{j=1}^N
\|\bm\psi_j^{[r]}\|_2
\|\bm\psi_j-\bm\psi_j^{[r]}\|_2
\le{}&
M(1+\eta^{-1/2})\sqrt{\Lambda_{>r}}+M(1+\eta^{-1/2})^2\Lambda_{>r}.
\end{aligned}
\label{eq:proof_cross_bound}
\end{equation}

Substituting Eqs.~\eqref{eq:proof_tailtail_bound} and~\eqref{eq:proof_cross_bound} into Eq.~\eqref{eq:proof_norm_diff} yields
\begin{equation}
\begin{aligned}
\|\mathbf G_{M,N}-\mathbf G_{M,N}^{[r]}\|_2
\le{}&
3(1+\eta^{-1/2})^2\Lambda_{>r}+2(1+\eta^{-1/2})\sqrt{\Lambda_{>r}}.
\end{aligned}
\label{eq:proof_truncation_final}
\end{equation}
For our choice $r=i-1$, this becomes
\begin{equation}
\begin{aligned}
\|\mathbf G_{M,N}-\mathbf G_{M,N}^{[i-1]}\|_2
\le{}&
3(1+\eta^{-1/2})^2\Lambda_{\ge i}+2(1+\eta^{-1/2})\sqrt{\Lambda_{\ge i}}.
\end{aligned}
\label{eq:truncation_final_i}
\end{equation}

Since $\Lambda_{\ge i}\le C_{\mathrm t}R_0^{-i}$, combining Eq.~\eqref{eq:truncation_final_i} with Eq.~\eqref{eq:main_bound} gives
\begin{equation}
\begin{aligned}
\left|
\hat\lambda_i(\mathbf G_{M,N})-\lambda_i(\mathcal T_k)
\right|
\le{}&
3(1+\eta^{-1/2})^2C_{\mathrm t}R_0^{-i}+2(1+\eta^{-1/2})
 \sqrt{C_{\mathrm t}}R_0^{-i/2}
+C R_0^{-i}.
\end{aligned}
\label{eq:final_eigenvalue_bound}
\end{equation}
In particular, Eq.~\eqref{eq:empirical_eigenvalue_tail} gives
\begin{equation}
\begin{aligned}
\hat\lambda_i(\mathbf G_{M,N})
\le{}&
3(1+\eta^{-1/2})^2C_{\mathrm t}R_0^{-i}+2(1+\eta^{-1/2})
 \sqrt{C_{\mathrm t}}R_0^{-i/2}.
\end{aligned}
\label{eq:finite_eigenvalue_upper}
\end{equation}

Since the singular values of $\mathbf\Psi$ satisfy
\begin{equation}
\hat\sigma_i(\mathbf\Psi)
=\sqrt{MN}\sqrt{\hat\lambda_i(\mathbf G_{M,N})},
\end{equation}
we have
\begin{equation}
\frac{\hat\sigma_i(\mathbf\Psi)}
     {\hat\sigma_1(\mathbf\Psi)}
=
\sqrt{
\frac{\hat\lambda_i(\mathbf G_{M,N})}
     {\hat\lambda_1(\mathbf G_{M,N})}}.
\label{eq:sv_ratio}
\end{equation}
Thus the normalization factors associated with both $M$ and $N$ cancel in the ratio.

To obtain a bound on the ratio, we require a lowerm bound on $\hat\lambda_1(\mathbf G_{M,N})$.
Boundedness of the features and independence over $j$ give
\begin{equation}
\mathbb E\left[
\|\mathbf G_{M,N}-\mathbf K_M\|_{\mathrm F}^2
\,\middle|\,x_1,\ldots,x_M
\right]
\le\frac1N.
\label{eq:feature_covariance_concentration}
\end{equation}
Indeed, each matrix entry is $1/M$ times an average of $N$ independent variables of absolute value at most one.

Similarly, let
\begin{equation}
\bm v(x)
=\biggl(\sqrt{\lambda_\ell}\varphi_\ell(x)\biggr)_{\ell\ge1}
\in\ell^2.
\end{equation}
Then $\|\bm v(x)\|_{\ell^2}^2=k(x,x)\le1$. The empirical second-moment operator
\begin{equation}
\mathcal C_M
=\frac1M\sum_{q=1}^M
\bm v(x_q)\otimes\bm v(x_q)
\end{equation}
has the same nonzero eigenvalues as $\mathbf K_M$, while $\mathbb E\mathcal C_M$ has eigenvalues $\{\lambda_\ell\}_{\ell\ge1}$.
Independence of the collocation points yields
\begin{equation}
\mathbb E
\|\mathcal C_M-\mathbb E\mathcal C_M\|_{\mathrm{HS}}^2
\le\frac1M,
\end{equation}
where $\|\cdot\|_{\mathrm{HS}}$ denotes the Hilbert--Schmidt norm, which is the infinite-dimensional analogue of the matrix Frobenius norm and dominates the operator norm.

Chebyshev's inequality and the eigenvalue perturbation bound consequently imply, with probability at least
$1-2\delta/3$,
\begin{equation}
\hat\lambda_1(\mathbf G_{M,N})
\ge
\lambda_1
-\sqrt{\frac3{M\delta}}
-\sqrt{\frac3{N\delta}}.
\label{eq:leading_eigenvalue_lower}
\end{equation}
Under Eq.~\eqref{eq:finite_sample_condition},
\begin{equation}
\hat\lambda_1(\mathbf G_{M,N})
\ge\frac{\lambda_1}{2}.
\label{eq:leading_eigenvalue_positive}
\end{equation}
In the asymptotic case, the corresponding argument gives
\begin{equation}
\hat\lambda_1(\mathbf K_M)
\ge
\lambda_1-\sqrt{\frac3{M\delta}}
\ge\frac{\lambda_1}{2}
\end{equation}
with probability at least $1-\delta/3$, under Eq.~\eqref{eq:asymptotic_sample_condition}.

We finally make the truncation estimates simultaneous over the spectral index. Set
\begin{equation}
\eta_i=\frac{\delta}{3i(i+1)},
\qquad
\sum_{i=1}^{\infty}\eta_i=\frac{\delta}{3}.
\end{equation}
A union bound makes the relevant truncation estimates valid simultaneously for all indicated indices, with probability at least $1-\delta/3$.
Since $\eta_i<1$,
\begin{equation}
(1+\eta_i^{-1/2})^2
\le\frac4{\eta_i}
=\frac{12}{\delta}i(i+1).
\end{equation}
Write $A_\delta=12/\delta$. In the asymptotic regime, combining Eq.~\eqref{eq:asymptotic_eigenvalue_bound} with the leading-eigenvalue lower bound gives, with probability at least $1-\delta$,
\begin{equation}
\lim_{N\to\infty}
\frac{\hat\sigma_i(\mathbf\Psi)}
     {\hat\sigma_1(\mathbf\Psi)}
\le
\left(
\frac{2A_\delta C_{\mathrm t}}{\lambda_1}
\right)^{1/2}
\sqrt{i(i+1)}\,R_0^{-i/2},
\qquad 1\le i\le M.
\end{equation}
The limit over the random features holds almost surely.

For finite $M,N$, combining Eq.~\eqref{eq:finite_eigenvalue_upper} with Eq.~\eqref{eq:leading_eigenvalue_positive} gives, with probability at least $1-\delta$,
\begin{equation}
\begin{aligned}
\frac{\hat\sigma_i(\mathbf\Psi)}
     {\hat\sigma_1(\mathbf\Psi)}
\le{}&
\left(
\frac{6A_\delta C_{\mathrm t}}{\lambda_1}
\right)^{1/2}
\sqrt{i(i+1)}\,R_0^{-i/2}+
\left(
\frac{4\sqrt{A_\delta C_{\mathrm t}}}{\lambda_1}
\right)^{1/2}
[i(i+1)]^{1/4}R_0^{-i/4},
\end{aligned}
\label{eq:finite_sv_explicit}
\end{equation}
simultaneously for $1\le i\le\min(M,N)$.

For any fixed $1<R<R_0$, the polynomial factors in $i$ can be absorbed into the exponential bounds with base $R$. Consequently, the estimates of cases (i) and (ii) of Proposition~\ref{prop:Psi_decay} hold on the respective probability events. The constants are independent of $i$, $M$, and $N$ within the stated regimes.

The extension to higher dimensions uses the same truncation estimates, which depend on the domain only through the eigenvalue tail.
Consider $\Omega=\mathbb S^s\subset\mathbb R^{s+1}$, where $s=d-1\ge1$.
Uniform analyticity implies smoothness of all orders on the sphere, which is the setting of the spherical spectral decay results in ref.~\cite{castro2020super}.
More specifically, a finite collection of analytic coordinate charts permits uniform polynomial approximation of the kernel in its first variable: Taylor truncation of degree $n$ on smaller coordinate patches gives a kernel of rank at most $C_0(n+1)^s$ with uniform error at most $C_1\rho^{-n}$, for some $\rho>1$.
A fixed partition of the sphere subordinate to these patches combines the local approximations without increasing their total rank beyond this order.

Because $\mu$ is a probability measure, the operator norm of the approximation error is bounded by its uniform kernel error. The finite-rank approximation inequality therefore gives
\begin{equation}
\lambda_i(\mathcal T_k)
\le C_2\exp(-c_0 i^{1/s})
\end{equation}
for constants $C_2,c_0>0$.
Consequently, for any $0<c_1<c_0$,
\begin{equation}
\Lambda_{\ge i}
\le C_3\exp(-c_1 i^{1/s}).
\label{eq:spherical_tail_bound}
\end{equation}
Substituting Eq.~\eqref{eq:spherical_tail_bound} into the same asymptotic and finite-sample truncation bounds, and absorbing the polynomial factors in $i$ by a further arbitrarily small reduction of the exponential rate, yields, for some $R>1$,
\begin{equation}
\lim_{N\to\infty}
\frac{\hat\sigma_i(\mathbf\Psi)}
     {\hat\sigma_1(\mathbf\Psi)}
=O\!\left(R^{-i^{1/2s}}\right),
\qquad
\frac{\hat\sigma_i(\mathbf\Psi)}
     {\hat\sigma_1(\mathbf\Psi)}
=O\!\left(R^{-i^{1/4s}}\right)
\end{equation}
in the respective regimes and with the same probability guarantees. These rates are super-algebraic: for every fixed $p>0$, they are also $O(i^{-p})$.
\end{proof}

\color{black}
\section{Sensitivity Analysis: Effect of Truncation Tolerance}
\label{app:sensitivity}

In this appendix, we provide a brief sensitivity analysis to assess the robustness of the proposed shift-invert Arnoldi method with respect to the truncation tolerance $\tau$, and to further elucidate the numerical issues associated with the naive approach.

We consider only the one-dimensional Liouville-Bratu-Gelfand problem at $p=3$ on the upper (unstable) branch, with $M = 101$ collocation points and $N = 50$ neurons, i.e., the same configuration used in the main text. The truncation tolerance $\tau$ is varied over the range $\tau \in [10^{-14}, 10^{-2}]$, and the first $15$ eigenvalues are computed using both the naive approach, i.e., $\widetilde{\mathbf{J}}_u = \mathbf{J}_w \Psi^\dagger$, and the proposed shift-invert Arnoldi method. The relative error with respect to a finite-difference reference solution, computed with $1000$ discretization points, is reported for each eigenvalue.

\begin{figure}[ht!]
    \centering
    \begin{tabular}{cc}
        \includegraphics[width=0.4345\linewidth]{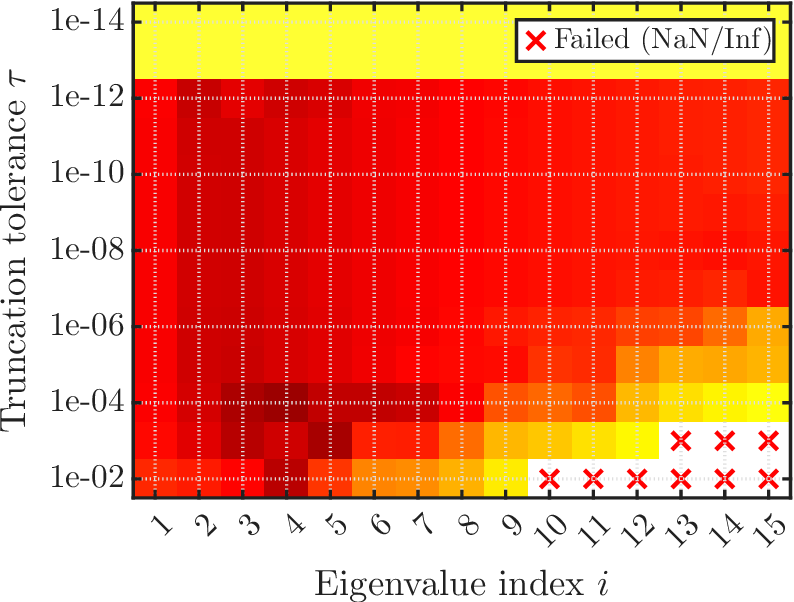} &
        \includegraphics[width=0.47\linewidth]{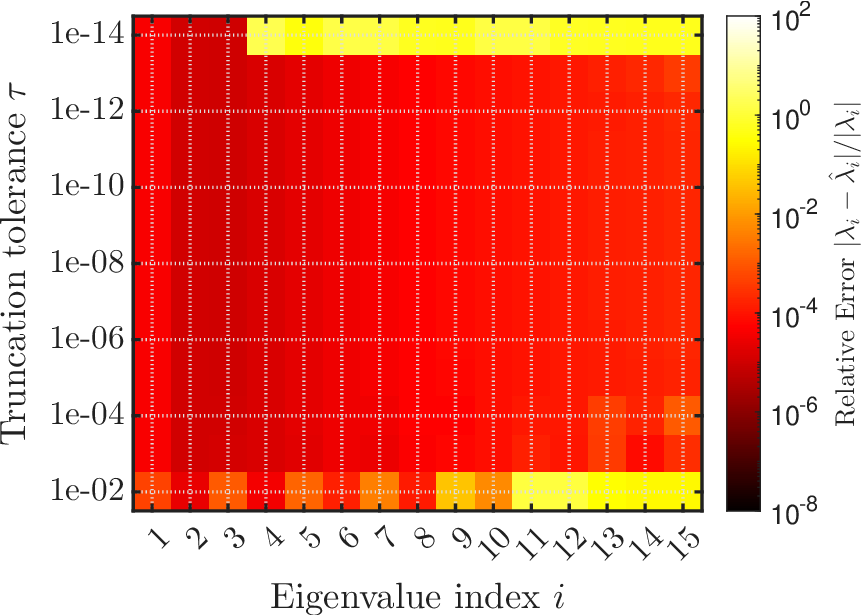} \\
        (a) Naive method & (b) Shift-invert Arnoldi\\
        \includegraphics[width=0.45\linewidth]{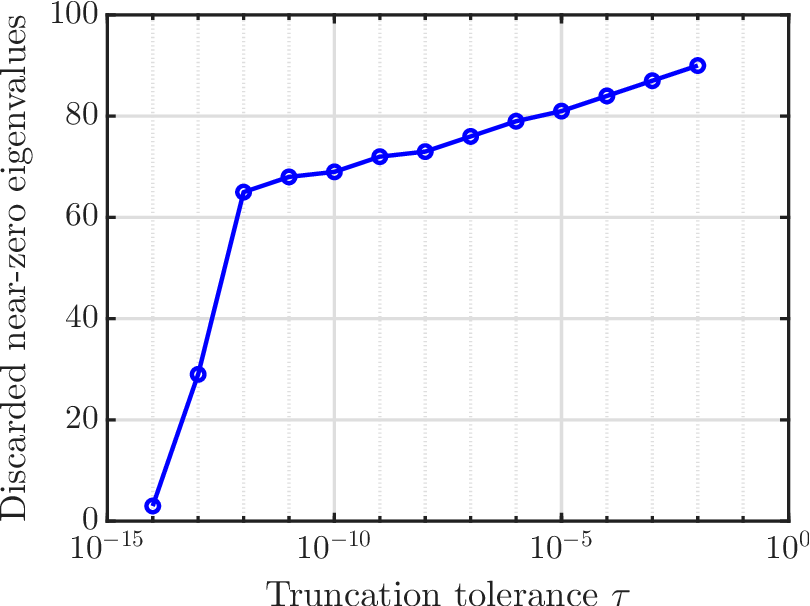} &
        \includegraphics[width=0.45\linewidth]{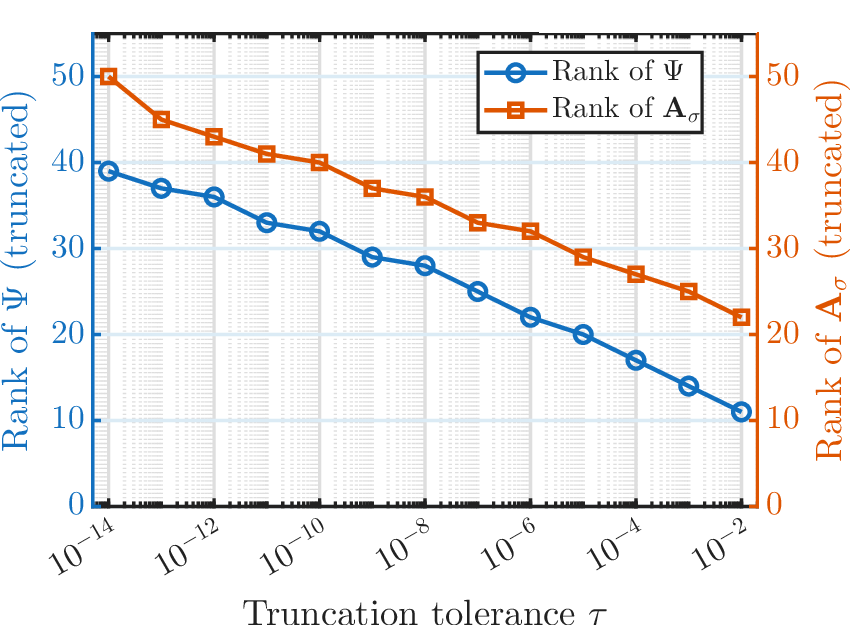} \\
        (c) Discarded eigenvalues vs $\tau$ & (d) Numerical rank vs $\tau$
    \end{tabular}
    \caption{\color{black} Sensitivity analysis of the computed eigenvalues with respect to the truncation tolerance $\tau$. Panels (a) and (b) show the relative error of the first 15 eigenvalues as a function of $\tau$ for the naive and shift-invert Arnoldi methods, respectively. Red $\times$ markers indicate numerical failures (NaN or Inf).
    Panel (c) shows the number of near-zero eigenvalues ($|\lambda| < 10^{-3}$) discarded in the naive method. Panel (d) shows the numerical rank of $\Psi$ and $\mathbf{A}_\sigma = \mathbf{J}_w - \sigma \mathbf{B}_{\Psi}$ as a function of $\tau$.
    \label{fig:sensitivity}}
\end{figure}

Fig.~\ref{fig:sensitivity}(a) reports the relative error for the naive method as a function of $\tau$ and the eigenvalue index $i$. The heatmap reveals that the naive method performs reliably across a broad range of tolerances, approximately $\tau \in [10^{-10}, 10^{-4}]$, with relative errors below $10^{-4}$ for the dominant eigenvalues. For very small tolerances, $\tau < 10^{-12}$, the pseudo-inverse $\Psi^\dagger$ becomes contaminated by rounding errors from near-zero singular values; as a consequence, the naive method fails to produce correct eigenvalues. For extremely large tolerances, $\tau > 10^{-4}$, excessive truncation discards too many singular directions, leading to degraded accuracy and, in some cases, numerical failure.

The proposed shift-invert Arnoldi method, shown in Fig.~\ref{fig:sensitivity}(b), exhibits similar robustness across the tolerance range, with a slight advantage for very small tolerances. Indeed, for $\tau = 10^{-14}$, the shift-invert method still recovers at least four dominant eigenvalues correctly, whereas the naive method fails entirely at this tolerance. Already at $\tau = 10^{-13}$, the shift-invert method recovers the first 15 leading eigenvalues correctly, whereas the naive method only begins to produce reliable results for $\tau \geq 10^{-12}$. This demonstrates that the shift-invert formulation is more resilient to the ill-conditioning of the random collocation matrix, as it avoids the explicit formation of the rank-deficient approximation of the physical Jacobian $\widetilde{\mathbf{J}}_u$.

Conversely, for large tolerances ($\tau > 10^{-4}$), both methods begin to degrade. However, the shift-invert method remains more robust: at $\tau = 10^{-3}$, the naive method fails for 3 eigenvalues (returning NaN or Inf), whereas the shift-invert method continues to produce accurate results. At $\tau = 10^{-2}$, the naive method fails for 6 eigenvalues, whereas the shift-invert method shows only mild degradation, still correctly recovering the first $10$ eigenvalues

The key difference between the two methods, however, lies not in the accuracy of the converged eigenvalues, but in the interpretation of the spectrum. As shown in Fig.~\ref{fig:sensitivity}(c), the number of near-zero spurious eigenvalues discarded in the naive method grows dramatically with $\tau$: from approximately 3 at $\tau = 10^{-14}$ to over 90 at $\tau = 10^{-2}$. Although one could, in principle, discard these modes and recover the correct eigenvalues, this strategy is highly problematic in bifurcation analysis. The physically relevant eigenvalue is precisely the one that crosses zero at the bifurcation point; discarding eigenvalues below a fixed threshold would eliminate the true critical eigenvalue exactly when it becomes most important. Without an independent criterion to distinguish physical from spurious modes, the naive approach remains unreliable for automated bifurcation detection.

Fig.~\ref{fig:sensitivity}(d) provides further insight into the observed behavior. The numerical rank of $\mathbf{A}_\sigma = \mathbf{J}_w - \sigma \mathbf{B}_{\Psi}$ decays more slowly than that of $\Psi$ as $\tau$ increases, indicating that the shifted pencil retains more spectral information than the random basis alone. This empirically explains why the shift-invert method remains accurate even when $\Psi$ is severely rank-deficient: the linearized PDE operator provides additional structure beyond the random features.

In summary, this sensitivity analysis confirms that both the naive and the proposed shift-invert Arnoldi methods are robust with respect to the choice of truncation tolerance over a wide range of $\tau$ values, with the shift-invert formulation offering slightly better resilience for extreme tolerances. Importantly, however, the naive approach introduces a large number of spurious near-zero eigenvalues that cannot be reliably distinguished from physical ones, rendering it unsuitable for bifurcation analysis. These results justify the choice $\tau = 10^{-8}$ used throughout the main text, as a safe choice for both methods, and demonstrate the practical advantage of the matrix-free shift-invert formulation.





\footnotesize




\begin{thebibliography}{10}
\expandafter\ifx\csname url\endcsname\relax
  \def\url#1{\texttt{#1}}\fi
\expandafter\ifx\csname urlprefix\endcsname\relax\def\urlprefix{URL }\fi
\expandafter\ifx\csname href\endcsname\relax
  \def\href#1#2{#2} \def\path#1{#1}\fi

\bibitem{raissi2019physics}
M.~Raissi, P.~Perdikaris, G.~E. Karniadakis, Physics-informed neural networks:
  A deep learning framework for solving forward and inverse problems involving
  nonlinear partial differential equations, Journal of Computational Physics
  378 (2019) 686--707.

\bibitem{karniadakis2021physics}
G.~E. Karniadakis, I.~G. Kevrekidis, L.~Lu, P.~Perdikaris, S.~Wang, L.~Yang,
  Physics-informed machine learning, Nature Reviews Physics 3~(6) (2021)
  422--440.

\bibitem{lu2021deepxde}
L.~Lu, X.~Meng, Z.~Mao, G.~E. Karniadakis, {DeepXDE:} a deep learning library
  for solving differential equations, SIAM Review 63~(1) (2021) 208--228.

\bibitem{meade1994numerical}
A.~J. Meade~Jr, A.~A. Fernandez, The numerical solution of linear ordinary
  differential equations by feedforward neural networks, Mathematical and
  Computer Modelling 19~(12) (1994) 1--25.

\bibitem{dissanayake1994neural}
M.~Dissanayake, N.~Phan-Thien, Neural-network-based approximations for solving
  partial differential equations, Communications in Numerical Methods in
  Engineering 10~(3) (1994) 195--201.

\bibitem{lagaris1998artificial}
I.~E. Lagaris, A.~Likas, D.~I. Fotiadis, Artificial neural networks for solving
  ordinary and partial differential equations, IEEE Transactions on Neural
  Networks 9~(5) (1998) 987--1000.

\bibitem{gerstberger1997feedforward}
R.~Gerstberger, P.~Rentrop, Feedforward neural nets as discretization schemes
  for {ODEs} and {DAEs}, Journal of Computational and Applied Mathematics
  82~(1-2) (1997) 117--128.

\bibitem{gonzalez1998identification}
R.~Gonz{\'a}lez-Garc{\'\i}a, R.~Rico-Mart{\`\i}nez, I.~G. Kevrekidis,
  Identification of distributed parameter systems: A neural net based approach,
  Computers \& Chemical engineering 22 (1998) S965--S968.

\bibitem{samaniego2020energy}
E.~Samaniego, C.~Anitescu, S.~Goswami, V.~M. Nguyen-Thanh, H.~Guo, K.~Hamdia,
  X.~Zhuang, T.~Rabczuk, An energy approach to the solution of partial
  differential equations in computational mechanics via machine learning:
  Concepts, implementation and applications, Computer Methods in Applied
  Mechanics and Engineering 362 (2020) 112790.

\bibitem{kavousanakis2026going}
M.~E. Kavousanakis, G.~Fabiani, A.~S. Georgiou, C.~Siettos, P.~G. Kevrekidis,
  I.~G. Kevrekidis, Going with the flow to solve for symmetry-driven {PDE}
  dynamics with physics-informed neural networks, Nature Communications
  17~(9420) (2026).

\bibitem{zou2025learning}
Z.~Zou, Z.~Wang, G.~E. Karniadakis, Learning and discovering multiple solutions
  using physics-informed neural networks with random initialization and deep
  ensemble, Proceedings of the Royal Society A 481~(2325) (2025) 20250205.

\bibitem{kiyani2025optimizing}
E.~Kiyani, K.~Shukla, J.~F. Urb{\'a}n, J.~Darbon, G.~E. Karniadakis, Optimizing
  the optimizer for physics-informed neural networks and {Kolmogorov-Arnold}
  networks, Computer Methods in Applied Mechanics and Engineering 446 (2025)
  118308.

\bibitem{wang2025turbulence}
S.~Wang, S.~Sankaran, X.~Fan, P.~Stinis, P.~Perdikaris, Simulating
  three-dimensional turbulence with physics-informed neural networks, arXiv
  preprint arXiv:2507.08972 (2025).

\bibitem{wei2018machine}
Q.~Wei, Y.~Jiang, J.~Z. Chen, Machine-learning solver for modified diffusion
  equations, Physical Review E 98~(5) (2018) 053304.

\bibitem{han2018solving}
J.~Han, A.~Jentzen, W.~E, Solving high-dimensional partial differential
  equations using deep learning, Proceedings of the National Academy of
  Sciences 115~(34) (2018) 8505--8510.

\bibitem{hu2025scorepinn}
Z.~Hu, Z.~Zhang, G.~E. Karniadakis, K.~Kawaguchi, Score-based physics-informed
  neural networks for high-dimensional fokker--planck equations, SIAM Journal
  on Scientific Computing 47~(3) (2025) C680--C705.

\bibitem{georgiou2025heatnets}
K.~Georgiou, G.~Fabiani, C.~Siettos, A.~N. Yannacopoulos, Heatnets: Explainable
  random feature neural networks for high-dimensional parabolic {PDE}s, arXiv
  preprint arXiv:2511.00886 (2025).

\bibitem{deryck2025approximation}
T.~De~Ryck, S.~Mishra, Y.~Shang, F.~Wang, Approximation theory and applications
  of randomized neural networks for solving high-dimensional {PDE}s, arXiv
  preprint arXiv:2501.12145 (2025).

\bibitem{alvarez2023discrete}
H.~Vargas~Alvarez, G.~Fabiani, N.~Kazantzis, C.~Siettos, I.~G. Kevrekidis,
  Discrete-time nonlinear feedback linearization via physics-informed machine
  learning, Journal of Computational Physics 492 (2023) 112408.

\bibitem{alvarez2024nonlinear}
H.~V. Alvarez, G.~Fabiani, N.~Kazantzis, I.~G. Kevrekidis, C.~Siettos,
  Nonlinear discrete-time observers with physics-informed neural networks,
  Chaos, Solitons \& Fractals 186 (2024) 115215.

\bibitem{patsatzis2024slow}
D.~Patsatzis, G.~Fabiani, L.~Russo, C.~Siettos, Slow invariant manifolds of
  singularly perturbed systems via physics-informed machine learning, SIAM
  Journal on Scientific Computing 46~(4) (2024) C297--C322.

\bibitem{kalia2021learning}
M.~Kalia, S.~L. Brunton, H.~G. Meijer, C.~Brune, J.~N. Kutz, Learning normal
  form autoencoders for data-driven discovery of universal, parameter-dependent
  governing equations, arXiv preprint arXiv:2106.05102 (2021).

\bibitem{bertalan2019learning}
T.~Bertalan, F.~Dietrich, I.~Mezi{\'c}, I.~G. Kevrekidis, On learning
  {Hamiltonian} systems from data, Chaos: An Interdisciplinary Journal of
  Nonlinear Science 29~(12) (2019) 121107.
\newblock \href {https://doi.org/10.1063/1.5128231}
  {\path{doi:10.1063/1.5128231}}.

\bibitem{li2024pinnoperator}
Z.~Li, H.~Zheng, N.~Kovachki, D.~Jin, H.~Chen, B.~Liu, K.~Azizzadenesheli,
  A.~Anandkumar, Physics-informed neural operator for learning partial
  differential equations, ACM/IMS Journal of Data Science 1~(3) (2024) 1--27.

\bibitem{wang2021learning}
S.~Wang, H.~Wang, P.~Perdikaris, Learning the solution operator of parametric
  partial differential equations with physics-informed {DeepONet}s, Science
  Advances 7~(40) (2021) eabi8605.
\newblock \href {https://doi.org/10.1126/sciadv.abi8605}
  {\path{doi:10.1126/sciadv.abi8605}}.

\bibitem{goswami2023physics}
S.~Goswami, A.~Bora, Y.~Yu, G.~E. Karniadakis, Physics-informed deep neural
  operator networks, in: Machine Learning in Modeling and Simulation: Methods
  and Applications, Springer, 2023, pp. 219--254.

\bibitem{fabiani2021numerical}
G.~Fabiani, F.~Calabr{\`o}, L.~Russo, C.~Siettos, Numerical solution and
  bifurcation analysis of nonlinear partial differential equations with extreme
  learning machines, Journal of Scientific Computing 89~(2) (2021) 1--35.

\bibitem{galaris2022numerical}
E.~Galaris, G.~Fabiani, I.~Gallos, I.~Kevrekidis, C.~Siettos, Numerical
  bifurcation analysis of pdes from lattice boltzmann model simulations: a
  parsimonious machine learning approach, Journal of Scientific Computing
  92~(2) (2022) 1--30.

\bibitem{shahab2024neural}
M.~L. Shahab, H.~Susanto, Neural networks for bifurcation and linear stability
  analysis of steady states in partial differential equations, Applied
  Mathematics and Computation 483 (2024) 128985.

\bibitem{shahab2025corrigendum}
M.~L. Shahab, H.~Susanto, Corrigendum to “neural networks for bifurcation and
  linear stability analysis of steady states in partial differential
  equations” [appl. math. comput. 483 (2024) 128985], Applied Mathematics and
  Computation 495 (2025) 129319.

\bibitem{shahab2025pinnlattices}
M.~L. Shahab, F.~A. Suheri, R.~Kusdiantara, H.~Susanto, Physics-informed neural
  networks for high-dimensional solutions and snaking bifurcations in nonlinear
  lattices, Physica D: Nonlinear Phenomena (2025) 134836.

\bibitem{liu2023adaptive}
Y.~Liu, W.~Liu, X.~Yan, S.~Guo, C.-a. Zhang, Adaptive transfer learning for
  {PINN}, Journal of Computational Physics 490 (2023) 112291.

\bibitem{karakoc2023collocation}
B.~G. Karako{\c{c}}, A.~Saha, D.~Y. Sucu, A collocation algorithm based on
  septic b-splines and bifurcation of traveling waves for {Sawada}-{Kotera}
  equation, Mathematics and Computers in Simulation 203 (2023) 12--27.

\bibitem{saha2023new}
A.~Saha, B.~G. Karako{\c{c}}, K.~K. Ali, New exact soliton solutions,
  bifurcation and multistability behaviors of traveling waves for the
  (3+1)-dimensional modified {Zakharov}-{Kuznetsov} equation with higher order
  dispersion, Mathematical Sciences and Applications E-Notes 11~(4) (2023)
  226--240.

\bibitem{blum1988training}
A.~Blum, R.~Rivest, Training a 3-node neural network is {NP}-complete, Advances
  in neural information processing systems 1 (1988).

\bibitem{froese2023training}
V.~Froese, C.~Hertrich, Training neural networks is {NP}-hard in fixed
  dimension, Advances in Neural Information Processing Systems 36 (2023)
  44039--44049.

\bibitem{almira2021negative}
J.~Almira, P.~Lopez-de Teruel, D.~Romero-Lopez, F.~Voigtlaender, Negative
  results for approximation using single layer and multilayer feedforward
  neural networks, Journal of mathematical analysis and applications 494~(1)
  (2021) 124584.

\bibitem{adcock2021gap}
B.~Adcock, N.~Dexter, The gap between theory and practice in function
  approximation with deep neural networks, SIAM Journal on Mathematics of Data
  Science 3~(2) (2021) 624--655.

\bibitem{karumuri2025efficient}
S.~Karumuri, L.~Graham-Brady, S.~Goswami, Efficient training of deep neural
  operator networks via randomized sampling, World Scientific Annual Review of
  Artificial Intelligence 03 (2025) 2540001.
\newblock \href {https://doi.org/10.1142/S2811032325400016}
  {\path{doi:10.1142/S2811032325400016}}.

\bibitem{fabiani2025random}
G.~Fabiani, Random projection neural networks of best approximation:
  Convergence theory and practical applications, SIAM Journal on Mathematics of
  Data Science 7~(2) (2025) 385--409.

\bibitem{fabiani2023parsimonious}
G.~Fabiani, E.~Galaris, L.~Russo, C.~Siettos, Parsimonious physics-informed
  random projection neural networks for initial value problems of {ODE}s and
  index-1 {DAE}s, Chaos: An Interdisciplinary Journal of Nonlinear Science
  33~(4) (2023).

\bibitem{schmidt1992feed}
W.~F. Schmidt, M.~A. Kraaijveld, R.~P. Duin, et~al., Feedforward neural
  networks with random weights, in: 11th IAPR International Conference on
  Pattern Recognition, Vol. II. Conference B: Pattern Recognition Methodology
  and Systems, IEEE Computer Society Press, 1992, pp. 1--4.

\bibitem{pao1994learning}
Y.-H. Pao, G.-H. Park, D.~J. Sobajic, Learning and generalization
  characteristics of the random vector functional-link net, Neurocomputing
  6~(2) (1994) 163--180.

\bibitem{jaeger2001echo}
H.~Jaeger, The “echo state” approach to analysing and training recurrent
  neural networks-with an erratum note, Bonn, Germany: German National Research
  Center for Information Technology GMD Technical Report 148~(34) (2001) 13.

\bibitem{huang2006extreme}
G.-B. Huang, Q.-Y. Zhu, C.-K. Siew, Extreme learning machine: theory and
  applications, Neurocomputing 70~(1-3) (2006) 489--501.

\bibitem{rahimi2007random}
A.~Rahimi, B.~Recht, Random features for large-scale kernel machines, Advances
  in neural information processing systems 20 (2007).

\bibitem{bolager2024sampling}
E.~L. Bolager, I.~Burak, C.~Datar, Q.~Sun, F.~Dietrich, Sampling weights of
  deep neural networks, Advances in Neural Information Processing Systems 36
  (2024).

\bibitem{bolager2026koopmaninformed}
E.~L. Bolager, A.~{\v{C}}ukarska, I.~Burak, Z.~Monfared, F.~Dietrich,
  Koopman-informed recurrent neural networks, Transactions on Machine Learning
  Research (2026).

\bibitem{barron1993universal}
A.~R. Barron, Universal approximation bounds for superpositions of a sigmoidal
  function, IEEE Transactions on Information theory 39~(3) (1993) 930--945.

\bibitem{igelnik1995stochastic}
B.~Igelnik, Y.-H. Pao, Stochastic choice of basis functions in adaptive
  function approximation and the functional-link net, IEEE Transactions on
  Neural Networks 6~(6) (1995) 1320--1329.

\bibitem{rahimi2008uniform}
A.~Rahimi, B.~Recht, Uniform approximation of functions with random bases, in:
  2008 46th annual allerton conference on communication, control, and
  computing, IEEE, 2008, pp. 555--561.

\bibitem{fabiani2025randonets}
G.~Fabiani, I.~G. Kevrekidis, C.~Siettos, A.~N. Yannacopoulos, Randonets:
  Shallow networks with random projections for learning linear and nonlinear
  operators, Journal of Computational Physics 520 (2025) 113433.
\newblock \href {https://doi.org/https://doi.org/10.1016/j.jcp.2024.113433}
  {\path{doi:https://doi.org/10.1016/j.jcp.2024.113433}}.

\bibitem{fabiani2024stability}
G.~Fabiani, E.~Bollt, C.~Siettos, A.~N. Yannacopoulos, Stability analysis of
  physics-informed neural networks for stiff linear differential equations,
  arXiv preprint arXiv:2408.15393 (2024).

\bibitem{dwivedi2020physics}
V.~Dwivedi, B.~Srinivasan, Physics informed extreme learning machine
  {(PIELM)}--a rapid method for the numerical solution of partial differential
  equations, Neurocomputing 391 (2020) 96--118.

\bibitem{calabro2021extreme}
F.~Calabr{\`o}, G.~Fabiani, C.~Siettos, Extreme learning machine collocation
  for the numerical solution of elliptic {PDE}s with sharp gradients, Computer
  Methods in Applied Mechanics and Engineering 387 (2021) 114188.

\bibitem{dong2021local}
S.~Dong, Z.~Li, Local extreme learning machines and domain decomposition for
  solving linear and nonlinear partial differential equations, Computer Methods
  in Applied Mechanics and Engineering 387 (2021) 114129.

\bibitem{schiassi2021extreme}
E.~Schiassi, R.~Furfaro, C.~Leake, M.~De~Florio, H.~Johnston, D.~Mortari,
  Extreme theory of functional connections: A fast physics-informed neural
  network method for solving ordinary and partial differential equations,
  Neurocomputing 457 (2021) 334--356.

\bibitem{dong2022computing}
S.~Dong, J.~Yang, On computing the hyperparameter of extreme learning machines:
  Algorithm and application to computational {PDE}s, and comparison with
  classical and high-order finite elements, Journal of Computational Physics
  463 (2022) 111290.

\bibitem{yan2022framework}
C.~A. Yan, R.~Vescovini, L.~Dozio, A framework based on physics-informed neural
  networks and extreme learning for the analysis of composite structures,
  Computers \& Structures 265 (2022) 106761.

\bibitem{sun2024local}
J.~Sun, S.~Dong, F.~Wang, Local randomized neural networks with discontinuous
  galerkin methods for partial differential equations, Journal of Computational
  and Applied Mathematics 445 (2024) 115830.

\bibitem{osorio2025physics}
J.~D. Osorio, M.~De~Florio, R.~Hovsapian, C.~Chryssostomidis, G.~E.
  Karniadakis, Physics-informed machine learning for solar-thermal power
  systems, Energy Conversion and Management 327 (2025) 119542.

\bibitem{li2025fourier}
X.~Li, J.~Wu, Y.~Huang, Z.~Ding, X.~Tai, L.~Liu, Y.-G. Wang, Fourier-feature
  induced physics informed randomized neural network method to solve the
  biharmonic equation, Journal of Computational and Applied Mathematics 468
  (2025) 116635.

\bibitem{fabiani2024task}
G.~Fabiani, N.~Evangelou, T.~Cui, J.~M. Bello-Rivas, C.~P. Martin-Linares,
  C.~Siettos, I.~G. Kevrekidis, Task-oriented machine learning assisted
  surrogates for tipping points of agent-based models, Nature Communications
  15~(4117) (2024) 1--13.
\newblock \href {https://doi.org/https://doi.org/10.1038/s41467-024-48024-7}
  {\path{doi:https://doi.org/10.1038/s41467-024-48024-7}}.

\bibitem{ahmadi2024ai}
N.~Ahmadi~Daryakenari, M.~De~Florio, K.~Shukla, G.~E. Karniadakis,
  Ai-aristotle: A physics-informed framework for systems biology gray-box
  identification, PLOS Computational Biology 20~(3) (2024) e1011916.

\bibitem{fabiani2025enabling}
G.~Fabiani, H.~Vandecasteele, S.~Goswami, C.~Siettos, I.~G. Kevrekidis,
  Enabling local neural operators to perform equation-free system-level
  analysis, arXiv preprint arXiv:2505.02308 (2025).

\bibitem{fabiani2025equation}
G.~Fabiani, C.~Siettos, I.~G. Kevrekidis, Equation-free coarse control of
  distributed parameter systems via local neural operators, arXiv preprint
  arXiv:2509.23975 (2025).

\bibitem{mishra2026eig}
R.~Mishra, G.~Krishnamurthi, B.~Srinivasan, S.~Natarajan, et~al., Eig-pielm: an
  efficient mesh-free method for eigenvalue problems using physics-informed
  extreme learning machines, Computer Methods in Applied Mechanics and
  Engineering 451 (2026) 118674.

\bibitem{ito1996nonlinearity}
Y.~Ito, Nonlinearity creates linear independence, Advances in Computational
  Mathematics 5 (1996) 189--203.

\bibitem{johnson1984extensions}
W.~B. Johnson, Extensions of lipschitz mappings into a hilbert space, Contemp.
  Math. 26 (1984) 189--206.

\bibitem{chan1982arclength}
T.~F.~C. Chan, H.~B. Keller, Arc-length continuation and multigrid techniques
  for nonlinear elliptic eigenvalue problems, SIAM Journal on Scientific and
  Statistical Computing 3~(2) (1982) 173--194.
\newblock \href {https://doi.org/10.1137/0903012} {\path{doi:10.1137/0903012}}.

\bibitem{zhang2006schur}
F.~Zhang, The Schur complement and its applications, Springer Science \&
  Business Media, 2006.

\bibitem{mityagin2020zero}
B.~S. Mityagin, The zero set of a real analytic function, Mathematical Notes
  107~(3-4) (2020) 529--530.

\bibitem{little1984eigenvalues}
G.~Little, J.~Reade, Eigenvalues of analytic kernels, SIAM journal on
  mathematical analysis 15~(1) (1984) 133--136.

\bibitem{castro2020super}
M.~H. Castro, T.~Jordao, A.~P. Peron, Super-exponential decay rates for
  eigenvalues and singular values of integral operators on the sphere, Journal
  of Computational and Applied Mathematics 364 (2020) 112334.

\bibitem{saad2011numerical}
Y.~Saad, Numerical methods for large eigenvalue problems: revised edition,
  SIAM, 2011.

\bibitem{ruhe1984rational}
A.~Ruhe, Rational krylov sequence methods for eigenvalue computation, Linear
  Algebra and its Applications 58 (1984) 391--405.

\bibitem{scott1982inverted}
D.~S. Scott, The advantages of inverted operators in rayleigh--ritz
  approximations, SIAM Journal on Scientific and Statistical Computing 3~(1)
  (1982) 68--75.

\bibitem{ericsson1980spectral}
T.~Ericsson, A.~Ruhe, The spectral transformation lanczos method for the
  numerical solution of large sparse generalized symmetric eigenvalue problems,
  Mathematics of Computation 35~(152) (1980) 1251--1268.

\bibitem{allen1972ground}
S.~M. Allen, J.~W. Cahn, Ground state structures in ordered binary alloys with
  second neighbor interactions, Acta Metallurgica 20~(3) (1972) 423--433.

\bibitem{boyd1986analytical}
J.~P. Boyd, An analytical and numerical study of the two-dimensional bratu
  equation, Journal of Scientific Computing 1 (1986) 183--206.

\bibitem{mohsen2014simple}
A.~Mohsen, A simple solution of the bratu problem, Computers \& Mathematics
  with Applications 67~(1) (2014) 26--33.

\bibitem{hajipour2018accurate}
M.~Hajipour, A.~Jajarmi, D.~Baleanu, On the accurate discretization of a highly
  nonlinear boundary value problem, Numerical Algorithms 79 (2018) 679--695.

\bibitem{fitzhugh1961impulses}
R.~FitzHugh, Impulses and physiological states in theoretical models of nerve
  membrane, Biophysical journal 1~(6) (1961) 445--466.

\bibitem{zhao2024bifurcation}
X.~E. Zhao, L.-Q. Chen, W.~Hao, Y.~Zhao, Bifurcation analysis reveals solution
  structures of phase field models, Communications on Applied Mathematics and
  Computation 6~(1) (2024) 64--89.

\bibitem{crandall1971bifurcation}
M.~G. Crandall, P.~H. Rabinowitz, Bifurcation from simple eigenvalues, Journal
  of Functional Analysis 8~(2) (1971) 321--340.

\bibitem{braun2006accurate}
M.~L. Braun, Accurate error bounds for the eigenvalues of the kernel matrix,
  The Journal of Machine Learning Research 7 (2006) 2303--2328.

\end{thebibliography}
\end{document}